\documentclass[12pt,reqno] {article} 

\usepackage{graphicx}
\usepackage[utf8]{inputenc}
\usepackage[english]{babel}
\usepackage{bbm}

\usepackage[nottoc]{tocbibind}
\usepackage{mathrsfs,amsfonts,amssymb,amsmath}

\usepackage{amsthm}
\usepackage{enumerate}
\usepackage{graphicx,cite}
\usepackage{romannum}
\usepackage{hyperref}
\usepackage{authblk}
\usepackage{graphicx}

\allowdisplaybreaks

\hypersetup{colorlinks=true, linkcolor=blue, citecolor=red}

\DeclareMathOperator{\diam}{diam}
\DeclareMathOperator{\Leb}{Leb}

\DeclareMathOperator{\interior}{int}

\newcommand{\RNum}[1]{\uppercase\expandafter{\romannumeral #1\relax}}

\numberwithin{equation}{section}

\newtheorem{theorem}{Theorem}[section]
\newtheorem{corollary}[theorem]{Corollary}

\newtheorem{lemma}[theorem]{Lemma}

\newtheorem{definition}[theorem]{Definition}
\newtheorem{remark}[theorem]{Remark}

\newtheorem{assumption}[theorem]{Assumption}

\usepackage[square,numbers]{natbib}
\begin{document}
\pagenumbering{arabic}

\title
{Maximal Large Deviations and Slow Recurrences in Weakly Chaotic Systems}


\author{Leonid Bunimovich\thanks{School of Mathematics, Georgia Institute of Technology, Atlanta, USA. \texttt{leonid.bunimovich@math.gatech.edu}}, Yaofeng Su\thanks{School of Mathematics, Georgia Institute of Technology, Atlanta, USA. \texttt{yaofeng.su@math.gatech.edu}}}

\date{\today}

\maketitle

\begin{abstract}
We prove a  maximal-type large deviation principle for dynamical systems with arbitrarily slow polynomial mixing rates.  Also several applications, particularly to billiard systems, are presented.
\end{abstract}

\tableofcontents

%
\section{Introduction}\ \par

A fundamental Birkhoff's ergodic theorem says that if $\mu$ is a measure preserved by an ergodic dynamics $f:X \to X$, then almost surely for any function $\phi$ with $\int \phi d\mu=0$,
\begin{equation*}\lim_{n\to \infty}\frac{\sum_{i \le n} \phi \circ f^i}{n}=0.  
\end{equation*}

Equivalently, for any $\epsilon>0$, $\mu$-a.s. $x \in X$ there is $N=N_x>0$, such that 
\begin{equation*}
   \sup_{n\ge N} \Big|\frac{\sum_{i \le n}\phi \circ f^i(x)}{n}\Big|\le \epsilon.
\end{equation*}

The large deviations approach provides quantitative estimates for $N$. If the dynamics is sufficiently strongly chaotic (e.g., it has fast mixing rates or, in other words, correlations decay exponentially) and $\phi$ is sufficiently smooth, then there is a constant $C>0$, such that for any $n, N \ge 1$
\begin{equation*}
    \mu\Big(\Big| \frac{\sum_{i \le n}\phi \circ f^i}{n}\Big|\ge \epsilon\Big) \approx e^{-Cn}, 
\end{equation*} and 
\begin{equation*}
    \mu\Big(\sup_{n \ge N}\Big| \frac{\sum_{i \le n}\phi \circ f^i}{n}\Big|\ge \epsilon\Big) \approx \sum_{n\ge N}e^{-Cn} \approx e^{-CN}.
\end{equation*}

If the dynamics is weakly chaotic (i.e., slow mixing rates, see e.g. \cite{melbourne09}), then this estimate becomes weaker (see \cite{melbourne09} for more details), \begin{equation*}
    \mu\Big(\Big| \frac{\sum_{i \le n}\phi \circ f^i}{n}\Big|\ge \epsilon\Big) \precsim n^{-a} \text{ where }a>0.
\end{equation*}

Then, on contrary to the case of exponential estimate, a direct summation over $n$ does not give analogous estimate 
\begin{equation*}
    \mu\Big(\sup_{n \ge N}\Big| \frac{\sum_{i \le n}\phi \circ f^i}{n}\Big|\ge \epsilon\Big)\precsim N^{-a} \text{ where }a >0.
\end{equation*}

We call this type of estimate a \textbf{maximal large deviation}. 

One of the motivations to consider maximal large deviations is due to a typical approach (called an inducing method) in ergodic theory. This approach is usually aimed on finding a reference set $Y \subsetneq X$ and its first return time $R: Y \to \mathbb{N}\bigcup \{0\}$, in order to prove  some nice property for the first return map $f^R: Y \to Y$. Then this property gets lifted to the original dynamics $f: X \to X$. This lifting technique relies on a quantitative estimate of the time $N$ in $\sup_{n\ge N} \frac{\sum_{i \le n}R \circ (f^R)^i}{n}$, which is related to the maximal large deviation property considered in the present paper. 

One of the main results (Theorem \ref{mld}) in the paper a proof of a maximal large deviation estimate for dynamical systems with arbitrarily slow polynomial mixing rates. Our estimates for maximal large deviations are optimal, in comparison to the corresponding estimates obtained in \cite{anothermld1, anothermld2} (See more details in Remark \ref{compare}). Such optimal estimates provide somewhat better answers to various questions which arise in applications.

The main body of the present paper is devoted to applications of Theorem \ref{mld}. We improve/strengthen the results in in the papers \cite{vaientiadv, demersadv, Subbb} to the systems with arbitrarily slow polynomial mixing rates. More precisely, \begin{enumerate}
    \item The paper \cite{vaientiadv} deduced the existence of  a Gibbs-Markov-Young (GMY) structure, introduced in \cite{Young2}, for sufficiently fast mixing rates. Our Theorems \ref{existgmy1} and \ref{existgym2} show that an arbitrarily slow polynomial mixing rate is enough to guarantee the existence of a GMY structure.
    \item The paper \cite{demersadv} using the inducing method to study certain limit laws for one-dimensional open systems with fast mixing. Our Theorem \ref{extendpoissonthm} extends it to arbitrarily slow polynomially 
     mixing one-dimensional dynamics (e.g. to intermittent maps).
    \item The Theorem \ref{fplthm} derives a convergence rate to Poisson limit laws for arbitrarily slow polynomially mixing billiards systems (e.g. for stadium-type billiards), while the technique employed in \cite{Subbb} does not do it.
\end{enumerate} 

The structure of the paper is the following. The section \ref{defnotation} provides some necessary definitions and notations used throughout the paper. A proof of the maximal large deviations theorem is given in section \ref{mldsection}. The section \ref{existencesection} deals with  Gibbs-Markov-Young structures. Applications to billiards and to one dimensional non-uniformly expanding maps are presented in the sections \ref{billiardsection} and \ref{onedsection}, respectively.


\section{Definitions and Notations}\label{defnotation}
We start by introducing some notations used throughout the paper.
\begin{enumerate}
 \item $C_z$ denotes a constant depending on $z$.
    \item The notation ``$a_n \precsim_z b_n$"  (``$a_n=O_{z}(b_n)$") means that there is a constant $C_z \ge 1$ such that (s.t.) $ a_n \le C_z  b_n$ for all $n \ge 1$, whereas the notation ``$a_n \precsim b_n$" (or ``$a_n=O(b_n)$") means that there is a constant $C \ge 1$ such that $ a_n \le C  b_n$ for all $n \ge 1$. Next, ``$a_n \approx_z b_n$" and $a_n=C_z^{\pm 1}b_n$ mean that there is a constant $C_z \ge 1$ such that  $ C_z^{-1}  b_n \le a_n \le C_z b_n$ for all $n \ge 1$. Further, the notations ``$a_n=C^{\pm1} b_n$" and ``$a_n \approx b_n$" means that there is a constant $C \ge 1$ such that $ C^{-1}  b_n \le a_n \le C b_n$ for all $n \ge 1$. Finally, ``$a_n =o(b_n)$" means that  $\lim_{n \to \infty}|a_n/b_n|=0$. 
      \item The notation $\mathbb{P}$ refers to a probability distribution on the probability space, where a random variable lives, and $\mathbb{E}$ denotes the expectation of a random variable. We do not specify the probability space and the probability when we use these two notations if the context is clear.
    \item $\mathbb{N}=\{1,2,3,\cdots\}$, $\mathbb{N}_0=\{0,1,2,3,\cdots\}$. 
    \item $S^1:=[-\pi/2,\pi/2]$ (endpoints are not identified).
    \item $\mu_A$ denotes a probability measure on the set $A$.
    \item $\mathcal{T}(\mathcal{A})$ denotes a tangent bundle of a (sub)manifold $\mathcal{A}$, while $\mathcal{T}_x(\mathcal{A})$ denotes a tangent space at the point $x \in \mathcal{A}$.
\end{enumerate}

\begin{definition}[Transfer operators]\label{transferoperator}\ \par
We say that $P$ is a transfer operator associated to a dynamical system $(\mathcal{M}, f, \mu_\mathcal{M})$, if it satisfies the relation $\int g \cdot P(\phi) d\mu_\mathcal{M}= \int g \circ f \cdot \phi d\mu_\mathcal{M} \text{ for all } \phi \in L^1, g \in L^{\infty}$. Suppose that the $\sigma$-algebra of $\mathcal{M}$ is $\mathcal{B}$, then $[P(\phi)] \circ f=\mathbb{E}[\phi|f^{-1}\mathcal{B}]$ almost surely.
\end{definition}
 
\section{A Theorem on Maximal Large Deviations}\label{mldsection}

In this section the following main technical theorem (called a theorem on maximal large deviations) of this paper is proved. We conclude this section with a remark about the papers \cite{anothermld1, anothermld2}. 

\begin{theorem}[Maximal large deviations]\label{mld} \par
Suppose that $\phi\in L^{\infty}$ has zero mean. If there are constants $p\in \mathbb{N}, C_{\phi}>0, \beta \in (0, p)$ such that
\begin{equation*}
 ||P^n(\phi)||^p_p\le C_{\phi} n^{-\beta}  \text{ for any } n \ge 1.
\end{equation*} 

Then for any $\epsilon>0, N \ge 1$  \begin{equation*}
    \mu_\mathcal{M}\Big(\sup_{n \ge N}\Big| \frac{\sum_{i \le n}\phi \circ f^i}{n}\Big|\ge \epsilon\Big)\precsim_{\epsilon, p, \beta}\max\{||\phi||^{p}_{\infty}, C_{\phi}\}\cdot  ||\phi||^{p}_{\infty} \cdot N^{-\beta },
\end{equation*} and \begin{equation*}
    \Big|\Big|\sup_{n \ge N}\Big| \frac{\sum_{i \le n}\phi \circ f^i}{n}\Big|\Big|\Big|^{2p}_{2p}\precsim_{p, \beta}\max\{||\phi||^{p}_{\infty}, C_{\phi}\}\cdot  ||\phi||^{p}_{\infty} \cdot N^{-\beta}.
\end{equation*} 

In particular, if the function $\phi\in L^{\infty}$ satisfies the inequality $||P^n(\phi)||_1\le C_{\phi} n^{-\beta}$ for some $C_{\phi}>0$ and any $n \ge 1$, then for any $\epsilon>0$
     \begin{gather*}
     ||P^n(\phi)||^{1+\left \lceil \beta \right \rceil}_{1+\left \lceil \beta \right \rceil}\le C_{\phi} ||\phi||^{\left \lceil \beta \right \rceil}_{\infty} n^{-\beta},\\ 
     \mu_\mathcal{M}\Big(\sup_{n \ge N}\Big| \frac{\sum_{i \le n}\phi \circ f^i}{n}\Big|\ge \epsilon\Big)\precsim_{\epsilon,  \beta}\max\{||\phi||^{1+\left \lceil \beta \right \rceil}_{\infty}, C_{\phi} ||\phi||^{\left \lceil \beta \right \rceil}_{\infty}\}\cdot  ||\phi||^{1+\left \lceil \beta \right \rceil}_{\infty}  \cdot N^{-\beta }.
     \end{gather*}
\end{theorem}

\begin{remark}
Unlike \cite{melbourne09}, we assume a certain decay rate of correlations in $L^p$ for bounded observables, which will be used later to establish an improved estimate (a faster decay rate) for maximal large deviations for unbounded observables.
\end{remark}

To prove Theorem \ref{mld} we will need two following results .
\begin{lemma}\label{ld}
Let the assumptions of Theorem \ref{mld} hold. Then for any $N \ge 1$, \begin{equation*}
    \big|\big|\frac{\sum_{i\le N}\phi \circ f^i}{N}\big|\big|^{2p}_{2p} \precsim_{p, \beta}||\phi||^p_{\infty} C_{\phi}N^{-\beta}.
\end{equation*}
\end{lemma}
\begin{proof}Since $p > \beta$, we have for any $j\le n \le N$, \begin{align*}
    ||\phi \cdot P^n(\sum_{i=j}^{n}\phi \circ f^i)||_{p} &\le ||\phi||_{\infty} \cdot ||P^n(\sum_{i=j}^{n}\phi \circ f^i)||_{p}\\
    &\le ||\phi||_{\infty} \cdot \sum_{i=j}^{n}||P^{n-i}\phi||_p\precsim_{p, \beta} ||\phi||_{\infty} \cdot C_{\phi}^{1/p} (n-1)^{1-\beta/p}.
\end{align*} 

By Rio's inequality (see page 1738 of \cite{melbourne09}), Definition \ref{transferoperator} and because $f_{*}\mu=\mu$, we have \begin{align*}
    \int \Big|\sum_{i \le N}\phi \circ f^i\Big|^{2p}d\mu
     &\precsim_p \Big[\sum_{n=1}^{N} \sup_{j \le n}\big|\big|\phi \circ f^n \cdot \mathbb{E}\Big[(\sum_{i=j}^{n}\phi \circ f^i)\Big| f^{-n}\mathcal{B}\Big]\big|\big|_{p}\Big]^{p}\\
     &\precsim_p \Big[\sum_{n=1}^{N} \sup_{j \le n}\big|\big|\phi \cdot P^n(\sum_{i=j}^{n}\phi \circ f^i)\big|\big|_{p}\Big]^{p}\\
    &\precsim_{p, \beta} ||\phi||^p_{\infty}C_{\phi} \Big[ \sum_{n=1}^{N} (n-1)^{1-\beta/p}\Big]^p\precsim_{p, \beta} ||\phi||^p_{\infty}C_{\phi} N^{2p-\beta},
\end{align*} which implies that $\big|\big|\frac{\sum_{i\le N}\phi \circ f^i}{N}\big|\big|^{2p}_{2p} \precsim_{p, \beta}||\phi||^p_{\infty} C_{\phi} N^{-\beta}$.
\end{proof}

\begin{lemma}\label{max}
Let again the assumptions of Theorem \ref{mld} hold. Then for any $N\ge 1$ \begin{equation*}
      \Big|\Big|\frac{\max_{n\le N}|\sum_{n\le i\le  N}\phi \circ f^i|}{N}\Big|\Big|^{2p}_{2p} \precsim_{p,\beta} \max\{||\phi||^{p}_{\infty}, C_{\phi}\} ||\phi||^{p}_{\infty}N^{-\beta}.
\end{equation*}
\end{lemma}
\begin{proof}
The method used here is inspired by \cite{Sutams}. It follows from Lemma \ref{ld} that \[\mathbb{E}|\sum_{i \le N}\phi \circ f^i|^{2p}\le g(N)^p,  \text{ where } g(N)\approx_{p, \beta} ||\phi||_{\infty}C_{\phi}^{1/p} N^{2-\beta/p}.\]
Let $a\ge b$ and $M_{a,b}:=\max_{b\le n\le a}|\sum_{n \le i \le a} \phi \circ f^i|$. If $n\ge N$ then $|\sum_{n\le i\le 2N}\phi\circ f^i|^{2p}\le M_{2N,N}^{2p}$. Now, if $n <N$ then  \begin{align*}
    \Big|\sum_{n\le i\le 2N}\phi\circ f^i\Big|^{2p}&\le \Big[M_{N,N}+\Big|\sum_{N\le i\le 2N}\phi \circ f^i\Big|\Big]^{2p}\\
    &=\sum_{j \le 2p}C^j_{2p}\cdot M_{N,N}^{j}\cdot  \Big|\sum_{N\le i\le 2N}\phi \circ f^i\Big|^{2p-j}\\
    &=M_{N,N}^{2p}+\Big|\sum_{N\le i\le 2N}\phi \circ f^i\Big|^{2p}+\sum_{0<j< 2p}C^j_{2p}\cdot M_{N,N}^{j}\cdot  \Big|\sum_{N\le i\le 2N}\phi \circ f^i\Big|^{2p-j}\\
    &\le M_{N,N}^{2p}+M_{2N,N}^{2p}+\sum_{0<j< 2p}C^j_{2p}\cdot M_{N,N}^{j}\cdot  \Big|\sum_{N\le i\le 2N}\phi \circ f^i\Big|^{2p-j}.
\end{align*}

Therefore, 
\begin{equation*}
    M_{2N,2N}^{2p}\le M_{2N,N}^{2p}+ M_{N,N}^{2p}+\sum_{0<j< 2p}C^j_{2p}\cdot M_{N,N}^{j}\cdot  \Big|\sum_{N\le i\le 2N}\phi \circ f^i\Big|^{2p-j}.
\end{equation*}

From H\"older's inequality, we obtain for any $j\in (0,2p)$ that \begin{align*}
   \mathbb{E}\Big(M_{N,N}^{j}\cdot  \Big|\sum_{N\le i\le 2N}\phi \circ f^i\Big|^{2p-j}\Big)&\le (\mathbb{E}M_{N,N}^{2p})^{j/2p}\cdot \Big(\mathbb{E}\Big|\sum_{N\le i\le 2N}\phi \circ f^i\Big|^{2p}\Big)^{(2p-j)/2p}\\
   &\le (\mathbb{E}M_{N,N}^{2p})^{j/2p}\cdot \Big(\mathbb{E}\Big|\sum_{ i\le N}\phi \circ f^i\Big|^{2p}\Big)^{(2p-j)/2p}\\
   &\le (\mathbb{E}M_{N,N}^{2p})^{j/2p}\cdot g(N)^{(2p-j)/2}.
\end{align*}

Therefore, 
\begin{align*}
    \mathbb{E}M_{2N,2N}^{2p}&\le \mathbb{E}M_{2N,N}^{2p}+ \mathbb{E}M_{N,N}^{2p} +\sum_{0<j< 2p}C^j_{2p}\cdot \mathbb{E} \Big(M_{N,N}^{j}\cdot  \Big|\sum_{N\le i\le 2N}\phi \circ f^i\Big|^{2p-j}\Big)\\
    & \le 2\mathbb{E}M_{N,N}^{2p}+\sum_{0<j< 2p}C^j_{2p}\cdot (\mathbb{E} M_{N,N}^{2p})^{j/2p} g(N)^{(2p-j)/2}.
\end{align*}

Similarly, we have \begin{align*}
    \mathbb{E}M_{2N+1,2N+1}^{2p}
    &\le \mathbb{E}M_{2N+1,N}^{2p}+\mathbb{E}M_{N,N}^{2p}+\sum_{0<j< 2p}C^j_{2p}\cdot (\mathbb{E} M_{N,N}^{2p})^{j/2p} g(N+1)^{(2p-j)/2}\\
    &\le \mathbb{E}M_{N+1,N+1}^{2p}+\mathbb{E}M_{N,N}^{2p}+\sum_{0<j< 2p}C^j_{2p}\cdot (\mathbb{E} M_{N,N}^{2p})^{j/2p} g(N+1)^{(2p-j)/2}.
\end{align*}

Suppose that $\mathbb{E}M_{N,N}^{2p}\le K g(N)^p$ for any $N \ge 1$. We will determine now a value of $K>0$. First \begin{align*}
    \mathbb{E}M_{2N,2N}^{2p}&\le 2K g(N)^p+\sum_{0<j<2p}C_{2p}^j K^{j/2p}g(N)^{j/2} g(N)^{p-j/2}\\
    &=Kg(N)^p \cdot \Big(2+\sum_{0<j<2p}C_{2p}^jK^{\frac{j}{2p}-1}\Big)\\
    &= Kg(2N)^p \cdot \frac{2+\sum_{0<j<2p}C_{2p}^jK^{\frac{j}{2p}-1}}{2^{2p-\beta}}.
\end{align*} 

Then

\begin{align*}
    \mathbb{E}M_{2N+1,2N+1}^{2p}
    &\le Kg(N+1)^p+Kg(N)^p+\sum_{0<j< 2p}C^j_{2p}\cdot K^{j/2p}g(N)^{j/2} g(N+1)^{(2p-j)/2}\\
    & \le 2Kg(N+1)^p+\sum_{0<j< 2p}C^j_{2p}\cdot K^{j/2p}g(N+1)^{j/2} g(N+1)^{(2p-j)/2}\\
    &\le 2Kg(N+1)^p+\sum_{0<j< 2p}C^j_{2p}\cdot K^{j/2p} g(N+1)^{p}\\
    &= Kg(2N)^p \cdot \Big(2+\sum_{0<j<2p}C_{2p}^jK^{\frac{j}{2p}-1}\Big) \cdot \frac{(N+1)^{2p-\beta}}{(2N)^{2p-\beta}}.
\end{align*}

Note that $2p-\beta >p\ge 1$. Therefore, $(2+\sum_{0<j<2p}C_{2p}^jK^{\frac{j}{2p}-1})/2^{2p-\beta}\le 1$, if $K=K_{p,\beta}$ is sufficiently large. Further, $\Big(2+\sum_{0<j<2p}C_{2p}^jK^{\frac{j}{2p}-1}\Big) \cdot \frac{(N+1)^{2p-\beta}}{(2N)^{2p-\beta}}\le 1$ if $K=K_{p,\beta}$, and for any sufficiently large positive integer $N \ge N_{p,\beta}$. Then $\mathbb{E}M_{N,N}^{2p}\precsim_{p,\beta} g(N)^p$ for any $N \ge N_{p,\beta}$, i.e., \begin{align*}
    \Big|\Big|\frac{\max_{n\le N}|\sum_{n\le i\le  N}\phi \circ f^i|}{N}\Big|\Big|^{2p}_{2p} \precsim_{p, \beta}||\phi||^p_{\infty} C_{\phi} N^{-\beta}
\end{align*}

If $N\le N_{p,\beta}$, then \begin{align*}
    \Big|\Big|\frac{\max_{n\le N}|\sum_{n\le i\le  N}\phi \circ f^i|}{N}\Big|\Big|^{2p}_{2p}\le ||\phi||^{2p}_{\infty} \le ||\phi||^{2p}_{\infty}N_{p,\beta}^{\beta} N^{-\beta}.
\end{align*}

Therefore, for any $N\ge 1$ 
\begin{align*}
    \Big|\Big|\frac{\max_{n\le N}|\sum_{n\le i\le  N}\phi \circ f^i|}{N}\Big|\Big|^{2p}_{2p} \precsim_{p,\beta} \max\{||\phi||^{p}_{\infty}, C_{\phi}\}\cdot  ||\phi||^{p}_{\infty}N^{-\beta},
\end{align*}which completes the proof.
\end{proof}

\begin{proof}[We turn now to a proof of Theorem \ref{mld}]
Our proof in many respects goes along the lines of \cite{Sutams}. Let $M_N:=\sup_{n \ge N}\big|\frac{\sum_{i\le n}\phi \circ f^i}{n}\big|$ for any $N\ge 1$. If $n\ge 2N$, then clearly $\big|\frac{\sum_{i\le n}\phi \circ f^i}{n}\big| \le M_{2N}$. Now, if $n\in [N,2N]$, then 
\begin{align*}
    \Big|\frac{\sum_{i\le n}\phi \circ f^i}{n}\Big|&=\Big|\frac{\sum_{i\le 3N}\phi \circ f^i}{n}-\frac{\sum_{n\le i\le 3N}\phi \circ f^i}{n}\Big|\\
    &\le \frac{3N}{n}\Big|\frac{\sum_{i\le 3N}\phi \circ f^i}{3N}\Big|+\frac{3N-n}{n}\Big|\frac{\sum_{n\le i\le 3N}\phi \circ f^i}{3N-n}\Big|\\
    &\le 3\Big|\frac{\sum_{i\le 3N}\phi \circ f^i}{3N}\Big|+\frac{3N-n}{n}\frac{\sup_{N\le n\le 2N}|\sum_{n\le i\le 3N}\phi \circ f^i|}{N}\\
    &\le 3\Big|\frac{\sum_{i\le 3N}\phi \circ f^i}{3N}\Big|+4\frac{\sup_{N\le n\le 3N}|\sum_{n\le i\le 3N}\phi \circ f^i|}{2N}.
\end{align*}

Therefore, for any $n \ge N$  \begin{equation*}
    \Big|\frac{\sum_{i\le n}\phi \circ f^i}{n}\Big| \le M_{2N}+ 3\Big|\frac{\sum_{i\le 3N}\phi \circ f^i}{3N}\Big|+4\frac{\sup_{N\le n\le 3N}|\sum_{n\le i\le 3N}\phi \circ f^i|}{2N},
\end{equation*} i.e., \begin{equation*}
    M_N\le M_{2N}+ 3\Big|\frac{\sum_{i\le 3N}\phi \circ f^i}{3N}\Big|+4\frac{\sup_{N\le n\le 3N}|\sum_{n\le i\le 3N}\phi \circ f^i|}{2N}.
\end{equation*} Moreover, \begin{equation*}
    N^{\beta}M_N\le 2^{-\beta}(2N)^{\beta}M_{2N}+ N^{\beta}\Big[3\Big|\frac{\sum_{i\le 3N}\phi \circ f^i}{3N}\Big|+4\frac{\sup_{N\le n\le 3N}|\sum_{n\le i\le 3N}\phi \circ f^i|}{2N}\Big].
\end{equation*}

Let $C_N:=3\Big|\frac{\sum_{i\le 3N}\phi \circ f^i}{3N}\Big|+4\frac{\sup_{N\le n\le 2N}|\sum_{n\le i\le 3N}\phi \circ f^i|}{2N}, c_N:=N^{\beta}C_N$, $b_N:= N^{\beta}M_N$. Then\begin{align*}
    b_N \le c_N+ 2^{-\beta}b_{2N}\le c_N+2^{-\beta} (c_{2N}+2^{-\beta}b_{2^2N})=c_N+2^{-\beta}c_{2N}+2^{-2\beta}b_{2^2N}\le \cdots.
\end{align*}

Inductively we have for any $n,N\ge 1$ that \begin{equation*}
    b_N \le \sum_{i \le n-1}2^{-i\beta}c_{2^iN}+2^{-n\beta}b_{2^nN}= \sum_{i \le n-1}2^{-i\beta}c_{2^iN}+N^{\beta}M_{2^nN}.
\end{equation*}

By Birkhoff’s ergodic theorem, $\lim_{n\to 0}M_{2^nN}=0$ almost surely. Let now $n\to \infty$. Then  for any $N\ge 1$ we have  \begin{equation*}
    b_N \le  \sum_{i\ge 0}2^{-i\beta}c_{2^iN} \text{ almost surely, and } 
    ||b_N||_{2p} \le  \sum_{i\ge 0}2^{-i\beta}||c_{2^iN}||_{2p},
\end{equation*}

Next we will estimate $||c_{2^iN}||_{2p}$ and show that it is uniformly bounded. By Lemmas \ref{ld} and \ref{max}, we have \begin{align*}
    ||C_N||_{2p}&\le 3\Big|\Big|\frac{\sum_{i\le 3N}\phi \circ f^i}{3N}\Big|\Big|_{2p}+4 \Big|\Big|\frac{\sup_{N\le n\le 3N}|\sum_{n\le i\le 3N}\phi \circ f^i|}{2N}\Big|\Big|_{2p}\\
    &= 3\Big|\Big|\frac{\sum_{i\le 3N}\phi \circ f^i}{3N}\Big|\Big|_{2p}+4 \Big|\Big|\frac{\sup_{ n\le 2N}|\sum_{n\le i\le 2N}\phi \circ f^i|}{2N}\Big|\Big|_{2p}\\
    &\precsim_{p, \beta}||\phi||^{1/2}_{\infty} C_{\phi}^{1/(2p)}N^{-\beta/(2p)}+\max\{||\phi||^{p}_{\infty}, C_{\phi}\}^{1/2p}\cdot  ||\phi||^{1/2}_{\infty}N^{-\beta/(2p)}\\
    &\precsim_{p, \beta}\max\{||\phi||^{p}_{\infty}, C_{\phi}\}^{1/2p}\cdot  ||\phi||^{1/2}_{\infty}N^{-\beta/(2p)}.
\end{align*} 

This relation implies that $||c_N||_{2p}\precsim_{p, \beta}\max\{||\phi||^{p}_{\infty}, C_{\phi}\}^{1/2p}\cdot  ||\phi||^{1/2}_{\infty}N^{\beta-\beta/(2p)}$, and \begin{align*}
    ||b_N||_{2p} &\precsim_{p, \beta}\max\{||\phi||^{p}_{\infty}, C_{\phi}\}^{1/2p}\cdot  ||\phi||^{1/2}_{\infty} \sum_{i\ge 0}2^{-i\beta} (2^iN)^{\beta-\beta/(2p)}\\
    &\precsim_{p, \beta}\max\{||\phi||^{p}_{\infty}, C_{\phi}\}^{1/2p}\cdot  ||\phi||^{1/2}_{\infty}  N^{\beta-\beta/(2p)}.
\end{align*}

Therefore, $||M_N||_{2p}=N^{-\beta}||b_N||_{2p}\precsim_{p, \beta}\max\{||\phi||^{p}_{\infty}, C_{\phi}\}^{1/2p}\cdot  ||\phi||^{1/2}_{\infty}  N^{-\beta/(2p)}$, and \begin{align*}
    \mu_{\mathcal{M}}\{M_N\ge \epsilon\}\precsim_{\epsilon, p, \beta}\max\{||\phi||^{p}_{\infty}, C_{\phi}\}\cdot  ||\phi||^{p}_{\infty}  N^{-\beta},
\end{align*}which completes a proof of Theorem \ref{mld}.
\end{proof}

\begin{remark}\label{compare}
The result of Lemma \ref{ld} is obtained for $\phi\in L^{\infty}$. It will be used later to deduce an estimate of large deviations from a rate of mixing. It was proved in \cite{anothermld1, anothermld2} that large deviations and maximal large deviations have similar decay rates for $L^2$-observables $\phi$. Unlike to \cite{anothermld1, anothermld2}, one can see from the proof of Theorem \ref{mld} and Lemma \ref{max} that the estimates for $\phi \in L^{\infty}$ can be used for $\phi \in L^{2p}$, and that large deviations and maximal large deviations share exactly the same decay rates for $L^{2p}$-observables $\phi$. Instead of explicitly writing down this optimal result here, we consider in Theorem \ref{mld} only the case when $\phi \in L^{\infty}$, since it is technically easier to derive mixing rates for bounded observables, and such observables are a natural choice in applications, which especially plays a crucial role in billiards considered in the section \ref{billiardsection}. \end{remark}

\section{Gibbs-Markov-Young Structures for Some Weakly Chaotic Dynamical Systems}\label{existencesection}

We start with the definition of the Gibbs-Markov-Young  structures.  Then we apply Theorem \ref{mld} to the problem of existence of Gibbs-Markov-Young structures, which was partially addressed in \cite{vaientiadv}.

Let $f : \mathcal{M} \to \mathcal{M}$ be a piecewise $C^{1+}$-endomorphism of a Riemannian manifold $\mathcal{M}$. Denote by $d$ the Riemannian distance in $\mathcal{M}$, and by $\Leb$ a normalised volume form on $\mathcal{M}$. By $C^{1+}$ we denote the class of continuously differentiable maps with 
H\"older continuous derivatives. 

The definition of a  Gibbs-Markov-Young structure below, can also be found in the section 1.1 of \cite{Young2} or Definition 1.1 of \cite{vaientiadv}.
\begin{definition}[Gibbs-Markov-Young (GMY) structures]\label{GMY}\ \par 

Suppose that there exists a ball $\Delta_0\subset \mathcal{M}$,
a countable partition $\mathcal{P}$ (mod $0$) of $\Delta_0$ into topological balls $U$ with smooth boundaries, and a return time function $R: \Delta_0 \to \mathbb{N}$, which is constant on each element of $\mathcal{P}$. We say that $f$ admits a Gibbs–Markov-Young (GMY) structure if the following properties hold:
\begin{enumerate}
    \item{Markov property:} for each $U\in \mathcal{P}$ and $R=R(U)$, $f^R: U\to \Delta_0$ is a $C^{1+}$-diffeomorphism (and, in particular, it is a bijection). Thus the induced map $F: \Delta_0 \to \Delta_0$, $F(x)=f^{R(x)}(x)$, is defined almost everywhere and satisfies a classical Markov property.
\item{Uniform expansion:} there exists $\rho\in (0,1)$ such that for almost all $x \in \Delta_0$ we have $||DF(x)^{-1}||\le \rho$. In particular the separation time $s(x,y)$, which is equal to the maximum integer such that $F^i
(x)$ and $F^i
(y)$ belong to the same element of the partition $\mathcal{P}$ for all $i \le  s(x,y)$, is defined and is a finite for almost all $x, y \in \Delta_0$.
\item{Bounded distortion:} there exists $C > 0$ such that for any two points $x,y \in \Delta_0$ with $s(x,y) < \infty $ \begin{equation*}
    \Big|\frac{\det DF(x)}{\det DF(y)}-1\Big|\le C  \rho^{s(F(x),F(y))},
\end{equation*}where $\det DF$ is the Radon-Nikodym derivative of $F$ with respect to $\Leb$ on $\Delta_0$.
\end{enumerate}
\end{definition}

Theorem \ref{mld} implies the following corollary. 
\begin{corollary}\label{optimalmld}If $f$ admits  a Gibbs–Markov-Young structure and $\Leb(x\in \Delta_0: R(x)>N)\approx N^{1+\beta}$, then there exists an open and dense set $H_0$ in the space of H\"older functions, such that for any zero mean function $\phi\in H_0$, and for any $\epsilon>0$,
\begin{equation*}
    \lim_{N\to \infty}\frac{\log \mu_{\mathcal{M}}\Big(\sup_{n \ge N}\Big| \frac{\sum_{i \le n}\phi \circ f^i}{n}\Big|\ge \epsilon\Big)}{\log N}=-\beta .
\end{equation*}
\end{corollary}
\begin{proof}
The following lower bound estimate was proved in \cite{melbourne09}. \begin{align*}
    N^{-\beta}\precsim \mu_{\mathcal{M}}\Big(\Big| \frac{\sum_{i \le N}\phi \circ f^i}{N}\Big|\ge \epsilon\Big) \precsim \mu_{\mathcal{M}}\Big(\sup_{n \ge N}\Big| \frac{\sum_{i \le n}\phi \circ f^i}{n}\Big|\ge \epsilon\Big).
\end{align*}

Together with the upper bound from Theorem \ref{mld} it concludes the proof.
\end{proof}

\begin{remark}
Note that the Corollary 1 provides a stronger statement than the Theorem 1.2 in \cite{melbourne09}. 
\end{remark}

In what follows we will consider GMY structures for an $C^{1+}$-endomorphism $f : \mathcal{M} \to \mathcal{M}$. Let $\mathcal{H}_{\alpha}$ be the space of H\"older continuous functions on $\mathcal{M}$ with a H\"older constant $\alpha>0$. Recall that the H\"older norm of a function $\phi$ is given by $||\phi||_{\mathcal{H}_{\alpha}}:=||\phi||_{\infty}+\sup_{x\neq y}\frac{|\phi(x)-\phi(y)|}{d(x,y)^{\alpha}}< \infty$. 

\begin{definition}
We say that a measure $\mu_{\mathcal{M}}$ on $\mathcal{M}$ is expanding if $\int \log||Df^{-1}|| d\mu_{\mathcal{M}} \in (-\infty, 0)$.
\end{definition}
\subsection{$C^{1+}$-local diffeomorphism}

Let $f : \mathcal{M} \to \mathcal{M}$ be a $C^{1+}$-local diffeomorphism, i.e., $Df(x)^{-1}, Df(x)$ are well-defined, uniformly bounded for all $x\in \mathcal{M}$, and $\log ||\big[Df(\cdot)\big]^{-1}||$ is H\"older continuous.

\begin{theorem}[Existence of GMY \RNum{1}]\label{existgmy1} \par
Suppose that $f : \mathcal{M} \to \mathcal{M}$ is a $C^{1+}$-local diffeomorphism, and  $f$
admits an ergodic expanding absolutely continuous (w.r.t. $\Leb$) invariant probability measure $\mu_{\mathcal{M}}$. Further, we assume that there is $\beta>0$ such that for any zero mean function $\phi \in \mathcal{H}_{\alpha}$ and for any $n\ge 1$, $\int |P^n(\phi)|d\mu_{\mathcal{M}}\le C_{\phi}n^{-\beta}$ for some constant $C_{\phi}>0$. Then  $f$ admits a Gibbs-Markov-Young structure.
\end{theorem}
\begin{remark}
Theorem 2 improves Theorem A in \cite{vaientiadv} by allowing $\beta$ to be just any positive number. This means that the dynamics has a GMY structure if the correlations decay polynomially or superpolynomially).
\end{remark}
\begin{proof}[Proof of Theorem \ref{existgmy1}]
We start with proving the first statement. Consider a zero mean H\"older function  $\phi(x):=\log||Df(x)^{-1}||-\int \log||Df(x)^{-1}|| d\mu_{\mathcal{M}}$. Then $\int |P^n(\phi)|d\mu_{\mathcal{M}}\le C_{\phi}n^{-\beta}$ for some $\beta>0$ and $C_{\phi}>0$. Let \[\mathcal{E}:=\inf\Big\{N: \frac{\sum_{i\le n}\phi \circ f^i}{n}\le -\frac{\int \log ||Df(x)^{-1}||d\mu_{\mathcal{M}}}{2}  \text{ for all }n \ge N\Big\},\]
\[\mathcal{N}:=\inf\Big\{N: \Big|\frac{\sum_{i\le n}\phi \circ f^i}{n}\Big| \le -\frac{\int \log ||Df(x)^{-1}||d\mu_{\mathcal{M}}}{2} \text{ for all }n \ge N \Big\},\]

Clearly $\{\mathcal{E}>N\}\subseteq  \{\mathcal{N}>N\} \subseteq \Big\{\sup_{n\ge N}\Big|\frac{\sum_{i\le n}\phi \circ f^i}{n}\Big|  \ge -\frac{\int \log ||Df(x)^{-1}||d\mu_{\mathcal{M}}}{2} \Big\}$. Following the argument in the proof of Theorem 3.1 of \cite{vaientiadv}, we obtain that there is a ball $\Delta_0$ such that \[\Leb(\mathcal{E}>N\bigcap \Delta_0)\precsim \mu_{\mathcal{M}}(\mathcal{E}>N \bigcap \Delta_0) \precsim N^{-\beta},\] where the last ``$\precsim$" is due to Theorem \ref{mld}. It follows from the proof of Theorem 3.1 in \cite{vaientiadv} that a Gibbs-Markov-Young structure with a return time $R$ exists, and that $\Leb\{R> N\}\precsim N^{-\beta}$.
\end{proof}

\subsection{$C^{1+}$-local diffeomorphisms outside critical/singular sets}

\begin{definition}
We say that $x$ is a critical point if $Df(x)$ is not invertible, and $x$ is a singular point if $Df(x)$ does not exist. Let $\mathcal{C}$ denotes the set of critical and singular points, and $d(x, \mathcal{C})$ is
the distance between a point $x \in \mathcal{M}$ and the set $\mathcal{C}$. We say that the set $\mathcal{C}$ of critical/singular points is nondegenerate if there are constants $B,d, \eta >0$, such that for all $\epsilon>0$ the following four conditions hold.
\begin{enumerate}
    \item $\Leb(\{x: d(x, \mathcal{C})\le \epsilon\})\le B\epsilon^d$ (in particular $\Leb(\mathcal{C}) = 0$);
    \item $B^{-1}d(x, \mathcal{C})^{\eta}\le ||Df(x)v||\le B d(x, \mathcal{C})^{-\eta}$ for every $x\in \mathcal{M}\setminus \mathcal{C}$ and $v\in \mathcal{T}_x\mathcal{M}$ with $||v||=1$;
    \item $|\log ||Df (x)^{-1}||- \log ||Df (y)^{-1}||| \le   B|\log(d(y, \mathcal{C}))-\log(d(x, \mathcal{C}))|$ for all $x,y \in \mathcal{M}\setminus \mathcal{C}$; 
    \item $|\log\det|Df (x)|- \log\det |Df (y)|| \le   B|\log d(y, \mathcal{C})-\log d(x, \mathcal{C})|$ for all $x,y \in \mathcal{M}\setminus \mathcal{C}$.
\end{enumerate}
\end{definition}
\begin{theorem}[Existence of GMY \RNum{2}]\label{existgym2} \par
Suppose that $f : \mathcal{M} \to  \mathcal{M}$ is a $C^{1+}$-local diffeomorphism outside a nondegenerate critical
set $\mathcal{C}$, and $f$
admits an ergodic expanding absolutely continuous (w.r.t. $\Leb$) invariant probability measure $\mu_{\mathcal{M}}$ with $\frac{d\mu_{\mathcal{M}}}{d\Leb}\in L^{1+\delta}$ for some $\delta>0$. Let also there is such $\beta>0$ that for any zero mean $\phi \in \mathcal{H}_{\alpha}$, $\int |P^n(\phi)|d\mu_{\mathcal{M}}\le C_{\phi}n^{-\beta}$ for some constant $C_{\phi}>0$ and any $n\ge 1$. Then  $f$ admits a Gibbs-Markov-Young structure (see Definition \ref{GMY}).
\end{theorem}
\begin{remark}
Theorem 3 proves the result of Theorem C in \cite{vaientiadv} for any positive $\beta$. Therefore the dynamics considered in this theorem has a GMY structure for all cases of  polynomial or superpolynomial decay of correlations. 
\end{remark}
\begin{proof}[Proof of Theorem \ref{existgym2}] Fix $\epsilon>0$. Define $\phi_1:=\log ||Df^{-1}||$, $\phi_3(x):=-\log d_{\delta}(x,\mathcal{C})$, \begin{eqnarray*}
 \phi_2(x):=
\begin{cases}
-\log d(x, \mathcal{C}),      & d(x, \mathcal{C})< \delta,\\
\frac{\log \delta}{\delta}(d(x, \mathcal{C})-2\delta),  &\delta \le d(x, \mathcal{C})< 2\delta, \\
0,  & d(x,\mathcal{C})\ge 2\delta, \\
\end{cases} \quad  d_{\delta}(x, \mathcal{C}):=\begin{cases}d(x, \mathcal{C}), & d(x, \mathcal{C})< \delta,\\1, & d(x, \mathcal{C})\ge \delta,\\
\end{cases}
\end{eqnarray*}where $\delta>0$ is so small that \[\lim_{n \to \infty}\frac{\sum_{i \le n}\phi_3 \circ f^i}{n}=\int \phi_3 d\mu_{\mathcal{M}}\le \lim_{n \to \infty}\frac{\sum_{i \le n}\phi_2 \circ f^i}{n}=\int \phi_2 d\mu_{\mathcal{M}}\le \epsilon\] (see  (4.2) in \cite{vaientiadv}). Let \begin{gather*}
    \mathcal{E}_1:=\inf\Big\{N: \frac{\sum_{i\le n}\phi_1 \circ f^i}{n}\le -\frac{\int \phi_1d\mu_{\mathcal{M}}}{2}  \text{ for all }n \ge N\Big\},\\
    \mathcal{N}_1:=\inf\Big\{N: \Big|\frac{\sum_{i\le n}\phi_1 \circ f^i}{n}-\int \phi_1 d\mu_{\mathcal{M}}\Big| \le -\frac{\int \phi_1d\mu_{\mathcal{M}}}{2} \text{ for all }n \ge N \Big\},
\end{gather*}
\begin{gather*}
    \mathcal{E}_2:=\inf\Big\{N: \frac{\sum_{i\le n}\phi_3 \circ f^i}{n}\le 2\epsilon \text{ for all }n \ge N\Big\},\\
    \mathcal{N}_2:=\inf\Big\{N: \Big|\frac{\sum_{i\le n}\phi_3 \circ f^i}{n}-\int \phi_3 d\mu_{\mathcal{M}}\Big| \le \epsilon \text{ for all }n \ge N \Big\}.
\end{gather*}

     Clearly\begin{gather*}
    \{\mathcal{E}_1>N\}\subseteq \{\mathcal{N}_1> N\}\subseteq \Big\{\sup_{n\ge N}\Big|\frac{\sum_{i\le n}\phi_1 \circ f^i}{n}-\int \phi_1 d\mu_{\mathcal{M}}\Big| \ge -\frac{\int \phi_1d\mu_{\mathcal{M}}}{2}\Big\},\\
       \{\mathcal{E}_2>N\}\subseteq \{\mathcal{N}_2> N\}\subseteq \Big\{\sup_{n\ge N}\Big|\frac{\sum_{i\le n}\phi_3 \circ f^i}{n}-\int \phi_3 d\mu_{\mathcal{M}}\Big| \ge \epsilon\Big\}.
\end{gather*} 

It follows from the proof of Theorem 4.2 in \cite{vaientiadv} that GMY structure with $\Leb(R>N)\precsim N^{-a}$ exists, if $ \mu_{\mathcal{M}}(\{\mathcal{E}_2>N\}\bigcup  \{\mathcal{E}_2> N\})$ is of order $O(N^{-a})$ for some $a>0$. Therefore we just need to find decay rates for $\mu_{\mathcal{M}}\big\{\sup_{n\ge N}\big|\frac{\sum_{i\le n}\phi_3 \circ f^i}{n}-\int \phi_3 d\mu_{\mathcal{M}}\big| \ge \epsilon\big\}$ and $\mu_{\mathcal{M}}\big\{\sup_{n\ge N}\big|\frac{\sum_{i\le n}\phi_1 \circ f^i}{n}-\int \phi_1 d\mu_{\mathcal{M}}\big| \ge -\frac{\int \phi_1d\mu_{\mathcal{M}}}{2}\big\}$. In order to do it we truncate the functions $\phi_1$ and $\phi_3$. Define \[\phi_{1,k}:=\phi_1\mathbbm{1}_{|\phi_1|\le k}+k\mathbbm{1}_{\phi_1>k}-k\mathbbm{1}_{\phi_1<-k}, \quad \phi_{3,k}:=\phi_3\mathbbm{1}_{|\phi_3|\le k}+k\mathbbm{1}_{\phi_3>k}-k\mathbbm{1}_{\phi_3<-k}.\]

Let $j=1,3$ and $u=\frac{-\int \phi_1 d\mu_{\mathcal{M}}}{2}, \epsilon$, respectively. By making use of the maximal ergodic theorem we obtain that
\begin{align*}
&\mu_{\mathcal{M}}\Big(\sup_{n \ge N}\Big|\frac{\sum_{i \le n}\big({\phi_j }-\mathbb{E} {\phi_j}\big) \circ f^i}{n}\Big|>u\Big)\\
    &\le \mu_{\mathcal{M}}\Big(\sup_{n \ge N}\Big|\frac{\sum_{i \le n}\big(\phi_{j,k}-\mathbb{E}\phi_{j,k}\big) \circ f^i}{n}\Big|> \frac{u}{2}\Big)\\
    &\quad +\mu_{\mathcal{M}}\Big(\sup_{n \ge N}\Big|\frac{\sum_{i \le n}\big[\phi_j-\phi_{j,k}-\mathbb{E}\big(\phi_j-\phi_{j,k}\big)\big] \circ f^i}{n}\Big|> \frac{u}{2}\Big)\\
      &\precsim_u \mu_{\mathcal{M}}\Big(\sup_{n \ge N}\Big|\frac{\sum_{i \le n}\big(\phi_{j,k}-\mathbb{E}\phi_{j,k}\big) \circ f^i}{n}\Big|>\frac{u}{2}\Big)+\mu_{\mathcal{M}}\Big(\big|\mathbb{E}(\phi_{j}-\phi_{j,k})\big|>\frac{u}{4}\Big)\\
      & \quad +\mu_{\mathcal{M}}\Big(\sup_{n \ge 1}\Big|\frac{\sum_{i \le n}\big(\phi_j-\phi_{j,k}\big) \circ f^i}{n}\Big|>\frac{u}{4}\Big)\\
      &\precsim_u \mu_{\mathcal{M}}\Big(\sup_{n \ge N}\Big|\frac{\sum_{i \le n}\big(\phi_{j,k}-\mathbb{E}\phi_{j,k}\big) \circ f^i}{n}\Big|>\frac{u}{2}\Big)+||\phi_{j}-\phi_{j,k}||_1
\end{align*}

By Lemma 4.3 of \cite{vaientiadv} we have for some $\xi>0$ that $||\phi_{j}-\phi_{j,k}||_1=\int_{t \ge k}\mu_{\mathcal{M}}\{|\phi_{j,k}|\ge t\} dt \precsim e^{-\xi k}$. Apply now the Theorem \ref{mld} and Lemma 5.1 of \cite{vaientiadv}. Here we choose $k$ following the argument on the page 1226 of this paper. Thus \[\mu_{\mathcal{M}}\Big(\sup_{n \ge N}\Big|\frac{\sum_{i \le n}\big(\phi_{j}-\mathbb{E}\phi_{j}\big) \circ f^i}{n}\Big|>u\Big)\precsim_{u,\delta} N^{-\beta+\delta} \text{ for a small }\delta>0.\]

Therefore, a GMY structure with $\Leb\{R>N\}\precsim_{\delta} N^{-\beta+\delta}$ exists for sufficiently small $\delta>0$.
\end{proof}

\section{Applications to Two-dimensional Billiards}\label{billiardsection}

In this section we study some two-dimensional billiards. First we introduce several basic notions and properties of billiards, and give a definition of Chernov-Markarian-Zhang structures (see Definition \ref{cmz}). Then we will apply Theorem \ref{mld} to prove analogous statement (Theorem \ref{fplthm}) for this structure. Finally, applications to some specific classes of billiards will be considered.

Let $Q$ be a bounded connected domain in $\mathbb{R}^2$ (sometimes called a billiard table). The boundary $\partial Q$ consists of finitely many (at least) $C^3$-smooth curves. Each smooth component of the boundary $\partial Q$ is either flat, or dispersing (convex inwards), or focusing (convex outwards of a billiard table). 

We assume that each focusing boundary component is an arc of a circle, but not a full circle. Moreover, the corresponding full circle does not intersect any other boundary component. In what follows we will call this condition the simplest focusing chaos or SFC-condition. These billiards belong to the class of hyperbolic dynamical systems. 

The phase space of a billiard map is $\mathcal{M}:= \partial Q \times S^1$. A billiard map $f: \mathcal{M}\to \mathcal{M}$ sends a configuration of unit vectors with footpoints on the boundary at reflection times to the
configuration just after the next moments of reflection off the boundary. Denote $x=(q, \phi) \in \mathcal{M}$, $\pi_{\partial Q}x=q$, $dx=(dq, d\phi)\in \mathcal{T}_x\mathcal{M}$. Let $\mathcal{K}=\mathcal{K}(q)$ be the curvature of the boundary at a point $q \in \partial Q$. 

We list now several basic properties and formulas  (see e.g. \cite{CMbook}) for the class of two-dimensional billiards under consideration. 

The billiard map $f$ preserves a probability measure \begin{equation}\label{SRB}
    d\mu_{\mathcal{M}}:=(2\Leb_{\partial Q} \partial Q)^{-1}\cos \phi d\phi dq,
\end{equation}    
which is the projection of the phase volume (for the billiard flow) to the boundary.  

There exist unstable and stable cones $C^u,C^s \subseteq \mathcal{T}(\mathcal{M})$ defined as follows. For any $x=(q, \phi)$\begin{gather*}
    C^u_x:=\{(dq,d\phi)\in \mathcal{T}_x\mathcal{M}: \mathcal{K}\le d\phi/dq \le \infty\}
\end{gather*}for dispersing and flat boundary components, and \begin{gather*}
    C^u_x:=\{(dq,d\phi)\in \mathcal{T}_x\mathcal{M}: \mathcal{K}\le d\phi/dq \le 0\}
\end{gather*}for focusing components, and
\begin{gather*}
    C^s_x:=\{(dq,d\phi)\in \mathcal{T}_x\mathcal{M}: -\infty \le d\phi/dq \le -\mathcal{K}\}
\end{gather*}for dispersing and flat boundary components and \begin{gather*}
    C^s_x:=\{(dq,d\phi)\in \mathcal{T}_x\mathcal{M}: 0\le d\phi/dq \le -\mathcal{K}\}
\end{gather*}for focusing components. 

Let  $C^u:=\bigcup_{x\in \mathcal{M}}C^u_x$,   $C^s:=\bigcup_{x\in \mathcal{M}}C^s_x$ in $\mathcal{T}\mathcal{M}$. Then for all $m\ge 1$, \[(Df)C^u \subseteq C^u, \quad (Df)^{-1}C^s \subseteq C^s,\]
\begin{align}\label{invariantcone}
    (Df)^m\mathcal{T}(\{q\}\times S^1) \subseteq C^u, \quad (Df)^{-m}\mathcal{T}(\{q\}\times S^1) \subseteq C^s.
\end{align}

Now we fix a reference subset $X\subseteq \mathcal{M}$ with $\mu_{\mathcal{M}}(X)>0$. The first return time to $X$ is $R: X\to \mathbb{N}$, and the first return map is $f^R: X \to X$. Suppose that $X$ can be partitioned into countably many connected pieces
\begin{equation}\label{parition}
    X=\bigcup_i X_i,
\end{equation} so that $R$ is constant on each $X_i$, and  \[\mu_{\mathcal{M}} (\partial X_i)=0,\quad \interior{X_i} \bigcap \interior{X_j}=\emptyset \text{ for } i \neq j.\]
\begin{definition}[Singularities in $\mathcal{M}$ and $X$]\label{sing} \par
 Denote by $\mathbb{S}\subseteq X$ a singularity set for $f^R$. We assume that $\mathbb{S}$ has zero Lebesgue measure, and that $\mathbb{S}^c \subseteq X$ consists of countably many open connected components.  Let the set $\mathbb{S}_0$ consists of the discontinuities and the points where the map $f^R$ is not differentiable. Singularities of billiards include orbits which hit singularities of the boundary of the billiard table and orbits tangent to dispersing components of the boundary. Then $ \mathbb{S}_0 \subseteq \mathbb{S}$.
 
 We extend the singularities in $X$ to the singularities $\bigcup_{i}\bigcup_{0\le k < R|_{X_i}} f^{k} (\mathbb{S} \bigcap X_i)$  in $\mathcal{M}$, and denote this set by $\pmb{\mathbb{S}}$.
 \end{definition}

 \begin{definition}[Unstable and stable manifolds/curves in $X$]\label{unstablecone} \par
 An unstable (resp. stable) manifold in $X$ is a connected component of $(\bigcup_{i \ge 0}(f^R)^{i}\mathbb{S})^c$ (resp. $(\bigcup_{i \ge 0}(f^R)^{-i}\mathbb{S})^c$. A closed and connected part of the unstable (resp. stable) manifold will be called an unstable (resp. stable) disk. All (un)stable manifolds/disks are one-dimensional.
 
 We denote each unstable (resp. stable) manifold/disk by $\gamma^u$ (resp. $\gamma^s$), and its tangent vectors by $v^u$ (resp. $v^s$).
 
 Unstable (resp. stable) curve by $ \pmb{{\gamma}^u}$ (resp. $ \pmb{{\gamma}^s}$) is a smooth curve, such that $\mathcal{T}\pmb{{\gamma}^u} \in C^u$ (resp. $ \mathcal{T}\pmb{{\gamma}^s} \in C^s$).
 
 Denote the length of (un)stable manifolds/curves $\gamma$  by $\diam \gamma$. 
\end{definition}
\begin{remark}
Note that (un)stable manifolds are also (un)stable curves, but not vice versa. Observe also that (un)stable cones in \cite{CMbook} are in $\mathcal{T}(X)$, while our $C^u, C^s$ are in $\mathcal{T}(\mathcal{M})$. So the (un)stable curves in \cite{CMbook} are different from ours $\pmb{\gamma^u}$ and $\pmb{\gamma^s}$. We will need $\pmb{\gamma^u}, C^u$ and $\pmb{\gamma^s}, C^s$ to address the questions considered in the present section.
\end{remark}

\begin{definition}[Chernov-Markarian-Zhang (CMZ) structures]\label{cmz}\ \par
We say that $(X, f^R)$ is a CMZ structure of a two-dimensional mixing billiard $(\mathcal{M}, f, \mu_{\mathcal{M}})$ if there are constants $C>0$ and $\beta \in (0,1)$, such that the following conditions hold
\begin{enumerate}
    \item Hyperbolicity. For any $n\in \mathbb{N}$, $v^u$ and $v^s$,\begin{align*}
        |D(f^R)^n v^u|\ge C\beta^{-n} |v^u|, \quad  |D(f^R)^{n} v^s|\le C\beta^n |v^s|,
    \end{align*}where $|\cdot|$ is the Riemannian metric induced from $\mathcal{M}$ to (un)stable manifolds.
    \item SRB measures and u-SRB measures. Dynamical system $(X, f^R,\mu_X) $ is mixing, where $\mu_X:=\frac{\mu_\mathcal{M}}{\mu_{\mathcal{M}}(X)}\big|_X$. The conditional distribution on each $\gamma^u$ (say $\mu_{\gamma^u}$) is absolutely continuous w.r.t. Lebesgue measure $\Leb_{\gamma^u}$ on $\gamma^u$.
    \item Distortion bounds. Let $d_{\gamma^u}(\cdot, \cdot)$ be the distance measured along $\gamma^u$. By $\det D^uf^R$ we denote the Jacobian of $Df^R$ along the unstable manifolds. Then, if $x, y\in X$ belong to a $\gamma^u$, and $(f^R)^n$ is smooth on $\gamma^u$, the following relation holds
    \begin{align*}
        \log \frac{\det D^u(f^R)^n(x)}{\det D^u(f^R)^n(y)}\le \psi\left[d_{\gamma^u}\Big((f^R)^nx,(f^R)^ny\Big)\right],
    \end{align*}where $\psi(\cdot)$ is some function, which does not depend on $\gamma^u$, and $\lim_{s\to 0^+}\psi(s)=0$.
    \item Bounded curvatures. The curvatures of all $\gamma^u$ are uniformly
bounded by some constant $C$.
\item Absolute continuity. Consider a holonomy map $h : \gamma^u_1 \to \gamma^u_2$,  which maps a point $x \in \gamma^u_1$ to a point $h(x)\in \gamma^u_2$, such that both $x$ and $h(x)$ belong to the same $\gamma^s$. We assume that the holonomy map satisfies the following relation \begin{align*}
    \frac{\det D^u(f^R)^n(x)}{\det D^u(f^R)^n\big(h(x)\big)}=C^{\pm 1} \text{ for all }n \ge 1 \text{ and }x\in \gamma^u_1,
\end{align*}
\item Growth lemmas. There exist $N \in \mathbb{N}$, a sufficiently small $\delta_0>0$, and
constants  $ \kappa, \sigma > 0$, which satisfy the following conditions. For any
sufficiently small $\delta>0$ and for any disk on a smooth stable manifold  $\gamma^u$ with $\diam \gamma^u \le \delta_0$, denote by $U_{\delta} \subseteq \gamma^u$ a $\delta$-neighborhood of the subset $\gamma^u \bigcap \bigcup_{0\le i \le N}(f^R)^{-i}\mathbb{S}$ within the set $\gamma^u$. Then there exists an open subset $V_{\delta} \subseteq \gamma^u \setminus U_{\delta}$, such that $\Leb_{\gamma^u}\big(\gamma^u \setminus (U_{\delta}\bigcup V_{\delta})\big)=0$, and for any $\epsilon >0$ \begin{gather*}
    \Leb_{\gamma^u}(r_{V_{\delta},N}<\epsilon)\le 2\epsilon \beta+\epsilon C \delta_0^{-1} \Leb_{\gamma^u} (\gamma^u),\\
    \Leb_{\gamma^u}(r_{U_{\delta},0}< \epsilon)\le C \delta^{-\kappa} \epsilon,\\
    \Leb_{\gamma^u}(U_{\delta})\le C \delta^{\sigma},
\end{gather*}where $r_{U_{\delta},0}(x):=d_{\gamma^u}(x, \partial U_{\delta})$, $r_{V_{\delta}, N}(x):=d_{(f^R)^N\gamma^u}\big((f^R)^Nx, \partial (f^R)^NV_{\delta}(x)\big)$, and $V_{\delta}(x)$ is this very connected component of $V_{\delta}$, which contains $x$.

\item Polynomial mixing rates: $\int R^{1+\xi} d\mu_X< \infty$ for some $\xi>0$.
\end{enumerate}
\begin{remark}
 In the growth lemmas a positive integer $N$ is usually chosen as a sufficiently large number.  It follows from \cite{Chernovjsp, CZnon} that $(X, f^R, \mu_X)$ and $(\mathcal{M}, f, \mu_{\mathcal{M}})$ can be modelled by hyperbolic Young towers \cite{Young}. A mixing rate of $(X, f^R, \mu_{X})$ is of order $O(\rho^n)$ for some $\rho \in (0,1)$.
\end{remark}

Consider now the first return tower \begin{align*}
    \Delta:=\{(x,n)\in X \times \{0,1,2,\cdots\}: n < R(x)\}.
\end{align*}

Dynamics $F: \Delta \to \Delta$ is defined as $F(x, n)=(x, n+1)$ if $n+1\le R(x)-1$, or as $F(x,n)=(f^Rx, 0)$ if $n=R(x)-1$. A projection $\pi:\Delta \to \mathcal{M}$ is $\pi(x,n):=f^n(x)$ as $\pi \circ F=f \circ \pi$. Define an $F$-invariant probability on $\Delta$ by  $\mu_{\Delta}:=(\pi^{-1})_{*}\mu_{\mathcal{M}}$. Finally we introduce projections $\pi_{X}: \Delta \to X$, $\pmb{\pi}_X: \mathcal{M} \to X$ and $\pi_{\mathbb{N}}: \Delta \to \mathbb{N}_0$, so that for any $(x,n) \in \Delta$ \begin{equation}\label{projections} \pi_{X}(x,n)=x,\quad \pi_{\mathbb{N}}(x,n)=n,\quad \pmb{\pi}_X=\pi_X \circ \pi^{-1}.
\end{equation}

We identify $\Delta_0:=X \times \{0\} \bigcap \Delta$ with $X$, and $F^R$ with $f^R$. Thus $\pi: X \to X$ is an identity map.

Note that $\pi:\Delta \to \mathcal{M}$ is bijective. Hence $(\Delta,F)$ is identical to $(\mathcal{M},f)$, and $(X, f^R)=(\Delta_0, F^R)$ is a CMZ structure on $(\Delta, F)$. 
\end{definition}

\begin{definition}[Point processes]\label{dynamicalpointprocess} \par
We define now a dynamical point process (a random counting measure) $\mathcal{N}^{r,q, T}$ on $[0,T]$ for any $T>0$. 
Let $A\subseteq [0,T]$ be a measurable set. Then
\begin{align*}
    \mathcal{N}^{r,q, T}(A):=\sum_{i \cdot \mu_{\mathcal{M}}(B_r(q)\times S^1)\in A}\mathbbm{1}_{B_r(q)\times S^1}\circ f^i=_d \sum_{i \cdot \mu_{\Delta}(\pi^{-1}B_r(q)\times S^1)\in A}\mathbbm{1}_{\pi^{-1}B_r(q)\times S^1}\circ F^i.
\end{align*}

We say that $\mathcal{P}$ is a Poisson point process on $\mathbb{R}^+\bigcup\{0\}$ if
\begin{enumerate}
    \item $\mathcal{P}$ is a random counting measure on $\mathbb{R}^+\bigcup\{0\}$.
    \item $\mathcal{P}(A)$ is a Poisson-distributed random variable for any Borel set $A \subseteq \mathbb{R}^+\bigcup\{0\}$.
    \item If the sets $A_1, A_2, \cdots, A_n \subseteq \mathbb{R}^+\bigcup\{0\}$ are pairwise disjoint, then  $\mathcal{P}(A_1), \cdots, \mathcal{P}(A_n)$ are independent.
    \item $\mathbb{E}\mathcal{P}(A)=\Leb(A)$ for any Borel set $A\subseteq \mathbb{R}^+\bigcup\{0\}$.
\end{enumerate}
\end{definition}

\begin{definition}[Convergence rates for Poisson approximations]\label{fpl}\ \par
For any $T>0$ consider the $\sigma$-algebra $\mathcal{C}$ on the space of point processes on $[0,T]$, defined as \begin{align}\label{sigmaalg}
\sigma \left\{\pi^{-1}_AB: A \subseteq [0,T],B \subseteq \mathbb{N}_0\right\},
\end{align}
where $A, B$ are Borel sets and
$\pi_A$ is an evaluation map defined on the space of counting measures, so that for any counting measure $\mathcal{N}$
\[\pi_A\mathcal{N}:=\mathcal{N}(A).\]

Now we can define convergence rates for the Poisson approximation of a dynamical point process as\[d_{TV}\left(\mathcal{N}^{r,q,T},\mathcal{P}\right):=\sup_{C\in \mathcal{C}}\left|\mu(\mathcal{N}^{r,q,T} \in C)-\mathbb{P}(\mathcal{P} \in C)\right| \precsim_{q, T} r^a,\]
where a constant $a>0$ is a convergence rate.
\end{definition}

\begin{definition}[Sections and s-quasi-sections]\label{quasisection} \par
Recall that $\pi_X: \Delta \to X $ is defined by $\pi_X(x,n)=x$. We say that $S_r\subseteq \mathcal{M}$ is a section if $\pi_X: \pi^{-1}S_r\to X$ is injective for any small $r>0$. Given $s>0$ a set $B_r(q)\times S^1$ is a s-quasi-section if there exists a section $S_r \subseteq B_r(q)\times S^1$, such that  \[\mu_{\mathcal{M}}\big(B_r(q)\times S^1\setminus S_r\big)\precsim_q r^{1+s}.\]
\end{definition}

\begin{assumption}[\textbf{Some geometric assumptions}]\label{assumption}\ \par
There exist constants $C, s, g>0$, $K \in \mathbb{N}$ and $\alpha \in (0,1]$, such that
\begin{enumerate}
    \item For a.e. $q \in \partial Q$ the set $B_r(q)\times S^1$ is a s-quasi-section for any sufficiently small $r>0$.
    \item $\bigcup_{i\ge 1}\partial X_i\subseteq \mathbb{S}$ (see the definition of $X_i$ in (\ref{parition})). This implies that $\gamma^u, \gamma^s$ are entirely contained in one of the sets $X_i$ (see Definition \ref{unstablecone} for $\gamma^u$ and $\gamma^s$).  
    \item For any $ \gamma^k \subseteq  \mathbb{S}^c \bigcap X$ ($k=u$ and $s$) and any $ \pmb{{\gamma}^u} \subseteq  \mathbb{S}^c \bigcap X$\begin{gather*}
        \diam f^j{{\gamma}^k}\le C (\diam \gamma^k)^\alpha \text{ for all }j\in[0, R|_{\gamma^k}),\\
\diam f^j\pmb{{\gamma}^u}\le C (\diam \pmb{{\gamma}^u})^\alpha \text{ when }j= R|_{\pmb{\gamma^u}},\\
\diam \pmb{{\gamma}^u}\le C (\diam f^j\pmb{{\gamma}^u})^\alpha \text{ for all }j\in[0, R|_{\pmb{\gamma^u}}].
    \end{gather*}
    \item $\diam \gamma^u \le C \diam f^i \gamma^u$, where $i\in[0, R|_{\gamma^u})$ and $\gamma^u \subseteq X$.
    \item $\diam f^i \gamma^u \le C \diam (f^R)^K \gamma^u$ where $i\in[0, R|_{\gamma^u})$ and $(f^R)^K$ is smooth on $\gamma^u\subseteq X$.
    \item $\mu_X(N_{\epsilon}(\mathbb{S}))\le C \epsilon^g$, where $N_{\epsilon}(\mathbb{S})$ is the $\epsilon$-neighborhood of $\mathbb{S}$.
    \item Slopes of all vectors in $\{v \in Df^i C_x^u: x \in X, i\in[0,R(x)) \}$ are uniformly bounded.
    \item Tangent vectors to each curve in $\bigcup_{i\ge 0 }(f^R)^{-i}\mathbb{S}$ are uniformly transversal to the vectors in $\mathcal{T}\gamma^u$ or $\mathcal{T}\gamma^s$. (Here``uniformly transversal" means that the angles between two vectors are bounded away from $0$).
    \item The set $\pmb{\mathbb{S}}$ partitions any vertical curve $\{q\} \times S^1$ into countably many connected pieces $\bigcup_{i}I_{q,i}$. This property implies that for each $i$ we have that  $\pmb{\pi}_XI_{q,i}\subseteq X_j $ for some $j$). For any sub-interval $I \subseteq I_{q,i}$,  \[\diam \pmb{\pi}_XI\le C(\diam I)^{\alpha},\quad \diam I\le C (\diam f^R \pmb{\pi}_X I)^{\alpha}.\]
\end{enumerate}
\begin{remark}Some conditions in Assumption \ref{assumption} have been verified in \cite{Subbb} for particular classes of slowly mixing billiards .
\end{remark}
\end{assumption}
\begin{theorem}[Convergence rates]\ \label{fplthm}\ \par
Suppose that a two-dimensional mixing billiard system $(\mathcal{M}, f, \mu_{\mathcal{M}})$ has a CMZ structure $(X, f^R)$ (see Definition \ref{cmz}), and the Assumption \ref{assumption} holds. Then, for $\Leb_{\partial Q}$-a.s. $q\in\partial Q$, \[d_{TV}\left(\mathcal{N}^{r,q,T},\mathcal{P}\right) \precsim_{q, T} r^a,\]
(see Definition \ref{fpl}), where convergence rate $a$ depends only on $s,g,\xi, K, \alpha$. A formula for $a$ is given in the Lemma \ref{rateconclusion}.
\end{theorem}
\begin{remark}\ \par 
\begin{enumerate}
    \item Unlike \cite{vaientinullset, penebacktoball, Subbb}, the Theorem \ref{fplthm} establishes  convergence rate for Poisson approximations just under  requirement that a billiard system has an arbitrarily slow polynomial mixing rate. 
    \item We apply a new technique (maximal large deviations, see Theorem \ref{mld}) in order to obtain this convergence rate for Poisson approximations, whereas the technique of \cite{Subbb} does not allow that. 
    
    On the other hand, the methods of \cite{peneetds, Su} established convergence rates for Poisson approximations under conditions that contraction rates along (un)stable manifolds are sufficiently large. These conditions fail for billiards considered in our paper. A new via maximal large deviations does not require such strong conditions, and it is applicable to arbitrarily slow polynomially mixing billiard systems.
    
    \item In order to compare this to the mild assumptions in our previous paper \cite{Subbb}, note that the Assumption \ref{assumption} has more conditions on the singularities and (un)stable curves/manifolds, which are needed to provide estimates for convergence rates of Poisson approximations. Some conditions in the Assumption \ref{assumption} are natural for general hyperbolic systems, while the others are more specific for two-dimensional billiards. Notably, these conditions do not require fast rates of mixing.
    
    \item The paper \cite{peneijm} studied different holes in $\mathcal{M}$ for Sinai billiards with bounded horizons and for a diamond billiard. Although in the present paper we consider a special type of holes, i.e., the holes in the form of $B_r(q)\times S^1$, our technique can be adapted for more general holes. A type of holes, which we study here, is the most natural one for billiard systems. Consideration of general type holes would make the paper much longer and even more technical.
\end{enumerate}
\end{remark}

\begin{corollary}[The first hitting]\ \par
Under the same conditions as in Theorem \ref{fplthm} consider a moment of time when the first hitting of (passage through) a hole occurs, i.e., $\tau_{r,q}(x):=\inf\left\{n \ge 1: f^n(x) \in B_r(q)\times S^1\right\}$ for any $x \in \mathcal{M}$. Then for any $t>0$, and for almost every $q \in \partial Q$, the following relation holds 
\begin{equation*}
    \Big|\mu_{\mathcal{M}}\Big\{\tau_{r,q}>t\big/\mu\Big(B_r(q)\times S^1\Big)\Big\}-e^{-t}\Big|\precsim_{q,t} r^a \text{ for some }a>0.
\end{equation*}
\end{corollary}
\begin{proof}
Clearly, $\mu_{\mathcal{M}}\left\{\tau_{r,q}>t/\mu\big(B_r(q)\times S^d\big)\right\}=\mu_{\mathcal{M}}\left\{\mathcal{N}^{r,q}\left[0,t/\mu\big(B_r(q)\times S^d\big)\right]=0\right\}$. The corollary follows now from the Theorem \ref{fplthm}.
\end{proof}

\subsection{Exponentially and polynomially hyperbolic Young towers}
In order to prove Theorem \ref{fplthm}, we need to consider two hyperbolic Young towers (see \cite{Young,Young2}). It follows from \cite{Chernovjsp, CZnon, Markarianetds}, that a CMZ structure (see Definition \ref{cmz}) has a hyperbolic product structure $\Lambda \subsetneq X$, such that

\begin{enumerate}

    \item $\Lambda=\left(\bigcup \gamma^s\right) \bigcap \left(\bigcup \gamma^u\right)$ for a family of unstable disks $\Gamma^u:=\{\gamma^u\}$ and a family of stable disks $\Gamma^s:=\{\gamma^s\}$.
    
    \item $\dim \gamma^s+\dim \gamma^u=\dim \mathcal{M}=2$,
    
    \item each $\gamma^s$ intersect each $\gamma^u$ at
    exactly one point,
    
    \item stable and unstable disks are transversal, and the angles between them are uniformly bounded away from 0.
\end{enumerate}

A Markov property also holds, i.e., there exist pairwise disjoint $s$-subsets $\Lambda_1,\Lambda_2, \cdots \subseteq \Lambda$, such that $\Lambda_i=\left(\bigcup_{\gamma^s \in \Gamma^s_i} \gamma^s\right) \bigcap \left(\bigcup_{\gamma^u \in \Gamma^u} \gamma^u\right)$ for some $\Gamma^s_i\subseteq \Gamma^s$, and \begin{enumerate}
        \item $\Leb_{\gamma}\left(\Lambda \setminus(\bigcup_{i \ge 1}\Lambda_i)\right)=0$ on each $\gamma \in \Gamma^u$,
        \item There exist a return time function $R_p: \Lambda \to \mathbb{N}$ and a return map $f^{R_p}: \Lambda \to \Lambda$, such that for each $i \ge 1$ a value of $R_p|_{\Lambda_i}$ is  constant. Besides,   $f^{R_p} (\Lambda_i)$ is a $u$-subset (i.e., each $f^{R_p}(\Lambda_i)=\left(\bigcup_{\gamma^s \in \Gamma^s} \gamma^s\right) \bigcap \left(\bigcup_{\gamma^u \in \Gamma^u_i} \gamma^u\right)$ for some $\Gamma^u_i\subseteq \Gamma^u$). Moreover, for all $x \in \Lambda_i$, \begin{align*}
          f^{R_p}\Big(\gamma^s(x)\Big) \subseteq \gamma^s\Big(f^{R_p}(x)\Big), \quad f^{R_p}\Big(\gamma^u(x)\Big) \supseteq \gamma^u\Big(f^{R_p}(x)\Big),
        \end{align*}where $\gamma^u(y)$ (resp. $\gamma^s(y)$) is an element of $\Gamma^u$ (resp. $\Gamma^s$)
which contains $y\in \Lambda$.
\item One can consider also another return time $R_e: \Lambda \to \mathbb{N}$, such that $R_e|_{\Lambda_i}$ is a constant and $R^{R_e}=R_p$.
\end{enumerate}

\subsubsection{Polynomially hyperbolic Young towers}\label{polytower} \par

Define a tower $\Delta_p$ and a map $F_p: \Delta_p \to \Delta_p$ as 
\begin{gather*}
    \Delta_p:=\{(x,l) \in \Lambda \times \mathbb{N}: 0 \le l < R_p(x) \},\\
    F_p(x,l):=\begin{cases}
 (x,l+1),      &l < R_p(x)-1\\
\left(f^{R_p}(x),0\right),  & l=R_p(x)-1\\
\end{cases}.
\end{gather*}

At first, we introduce a family of partitions $(\mathcal{Q}_k)_{k \ge 0}$ of $\Delta_p$ as 
\[\mathcal{Q}_0:=\{\Lambda_i \times \{l\}, i \ge 1, l < R_p|_{\Lambda_i}\}, \quad \mathcal{Q}_k:= \bigvee_{0 \le i \le k} F_p^{-i} \mathcal{Q}_0.\]

Next, a projection $\pi_p: \Delta_p \to \Delta$ is defined as 
\[\pi_p(x,l):=F^l(x).\]

There exist probability measures $\mu_{\Delta_p}$, $\mu_{\Lambda}$ on $\Delta_p$ and $\Lambda$, respectively, such that  
\begin{equation}\label{allpolymeasure}
    F\circ \pi_p=\pi_p \circ F_p, \quad (\pi_p)_{*}\mu_{\Delta_p}=\mu_{\Delta}, \quad (F_p)_{*}\mu_{\Delta_p}=\mu_{\Delta_p},  \quad \left(f^{R_p}\right)_{*}\mu_{\Lambda}=\mu_{\Lambda}.
\end{equation}

Let $\gcd\{R_p\}=1$. We suppose that there is a constant $C>0$, such that for any $m>k\ge 1$, any $(Q_i)_{i \ge 1} \subseteq \mathcal{Q}_k$, any $h: \Delta_p \to \mathbb{R}$ satisfying $||h||_{\infty}\le 1$, and $h(x,l)=h(y,l)$ for any $x,y \in \gamma^s \in \Gamma^s$, and for any allowable $l \in \mathbb{N}$ (i.e., $h$ is $\sigma(\bigcup_{k\ge 0}\mathcal{Q}_k)$-measurable), we have the following estimate for decay of correlations
\begin{equation}\label{polydecorrelation}
    \Big|\int \mathbbm{1}_{\bigcup_{i \ge 1} Q_i}  h \circ F_p^{m} d\mu_{\Delta_p}-\mu_{\Delta_p}\Big(\bigcup_{i \ge 1} Q_i\Big) \int h d\mu_{\Delta_p}\Big| \le C(m-k)^{-\xi} \mu_{\Delta_p}\Big(\bigcup_{i \ge 1} Q_i\Big).
\end{equation}

\subsubsection{Exponentially hyperbolic Young towers}\ \par
Define a tower $\Delta_e$ and a map $F_e: \Delta_e \to \Delta_e$ as 
\begin{gather*}
    \Delta_e:=\{(x,l) \in \Lambda \times \mathbb{N}: 0 \le l < R_e(x) \},\\
    F_e(x,l):=\begin{cases}
 (x,l+1),      &l < R_e(x)-1\\
\left((f^R)^{R_e}(x),0\right),  & l=R_e(x)-1\\
\end{cases}.
\end{gather*}

Next, let a projection $\pi_e: \Delta_e \to X $ is defined as \[\pi_e(x,l):=(f^R)^l(x).\]

Consider now  the following equivalence relation $\sim$ on $\Delta_e$

\[(x,n) \sim (y,m) \text{ if and only if }  x,y \in \gamma^s \text{ for some } \gamma^s \in \Gamma^s \text{ and } n=m.\]

By making use of this equivalence relation we can define a quotient tower $\widetilde{\Delta_{e}}:=\Delta_{e}/\sim$ with a canonical projection $\widetilde{\pi_{\Delta_e}}:\Delta_{e} \to \widetilde{\Delta_{e}}$. The corresponding quotient map $\widetilde{F_e}: \widetilde{\Delta_{e}} \to \widetilde{\Delta_{e}}$ is given via the relation $F_e=\widetilde{F_e} \circ \widetilde{\pi_{\Delta_e}}$.

There exist probability measures $\mu_{\Delta_e}$ and  $\mu_{\widetilde{\Delta_e}}$ on $\Delta_e, \widetilde{\Delta_e}$, respectively, such that  
\begin{align}
    &f^R\circ \pi_e=\pi_e \circ F_e, \quad (\pi_e)_{*}\mu_{\Delta_e}=\mu_{X}, \quad \mu_{\widetilde{\Delta_e}}=(\widetilde{\pi_{\Delta_e}})_{*}\mu_{\Delta_e},\nonumber \\
    & \quad \quad \quad \quad (F_e)_{*}\mu_{\Delta_e}=\mu_{\Delta_e}, \quad (\widetilde{F_e})_{*}\mu_{\widetilde{\Delta_e}}=\mu_{\widetilde{\Delta_e}}. \label{allexpmeasure}
\end{align}

Consider a Banach space $\mathcal{B}$ be $\{v: \widetilde{\Delta_e}\to \mathbb{R}: ||v||_{\mathcal{B}}:=||v||_m+||v||_l< \infty\}$, where \[||v||_m=\sup_{n}\frac{\big|\big|v|_{\widetilde{\Delta_{e,n}}}\big|\big|_{\infty}}{e^{\epsilon n}},\quad ||v||_l=\sup_{n,i}\sup_{x,y \in \widetilde{\Delta_{e,n,i}}}\frac{\big|v(x)-v(y)\big|}{d_{\widetilde{\Delta_e}}(x,y)e^{\epsilon n}}, \quad \epsilon>0,\]
\[\widetilde{\Delta_{e,n}}:=\Delta_{e} \bigcap (\Lambda \times \{n\})/\sim, \quad \widetilde{\Delta_{e,n,i}}:=\Delta_{e} \bigcap (\Lambda_i \times \{n\})/\sim,\] and \[d_{\widetilde{\Delta_e}} \text{ is a metric defined on }\widetilde{\Delta_e} \text{ (see its definition in \cite{Young})}.\] 

Now, following \cite{Young, balint}, fix any $p \in \mathbb{N}$, and choose $\epsilon=\epsilon_p>0$ to be sufficiently small. Then we have $\mathcal{B}\subsetneq L^p$, i.e., $||v||_p \precsim ||v||_{\mathcal{B}}$ for any $v \in \mathcal{B}$. Denote the transfer operator of $\widetilde{F_e}$ by $P$, and suppose that $\gcd\{R_e\}=1$. Then there is a constant $C=C_p>0$, such that for any zero mean  $v\in \mathcal{B}$, and for any $n\ge 1$, \begin{equation}\label{expdecorrelation}
    ||P^nv||_p \precsim ||P^nv||_{\mathcal{B}}\precsim_p ||v||_{\mathcal{B}} \cdot e^{-Cn}.
\end{equation}

\subsection{Maximal large deviations for $R:X \to \mathbb{N}$}
\begin{lemma}\label{mldforR}
$\mu_X\big(\sup_{n\ge N}\big|\frac{\sum_{i \le n}(R-\mathbb{E}R)\circ (f^R)^i}{n}\big|\ge \epsilon\big)\precsim_{\delta, \epsilon} N^{-\xi(1-\delta)}$ for any  $\delta\in (1/2,1)$.
\end{lemma}
\begin{proof}
Suppose that $\gcd\{R_e\}=1$, i.e., a hyperbolic Young tower $(\Delta_e, F_e, \mu_{\Delta_e})$ is mixing. First we will consider the function $R\circ \pi_e$. By Assumption \ref{assumption}, any $\gamma^s \subsetneq \Lambda$ and $\gamma^u \subseteq \Lambda_k$ (for any $k\ge 1$) are entirely contained in some $X_j$. Hence the values $R|_{\gamma^s}=R|_{\gamma^u}, R_e|_{\gamma^s}=R_e|_{\gamma^u}$ are constant. For each $i < R_e|_{\gamma^s}=R_e|_{\gamma^u}$ the sets $(f^R)^i\gamma^s, (f^R)^i\gamma^u$ are also entirely contained in some $X_j$. Then $R|_{(f^R)^i\gamma^s}=R|_{(f^R)^i\gamma^u}$ is constant. Therefore, $R\circ \pi_e$ is constant along $F_e^i\gamma^s, F_e^i \gamma^u \subseteq \Delta_e$ for any $i < R_e|_{\gamma^s}=R_e|_{\gamma^u}$. Hence, we can define $\widetilde{R\circ \pi_e}: \widetilde{\Delta_e} \to \mathbb{N}$, so that $\widetilde{R\circ \pi_e} \circ \widetilde{\pi_{\Delta_e}}=R \circ \pi_e$ and $||\widetilde{R\circ \pi_e}||_l=0$.

Consider now for any $k\in \mathbb{N}$ a truncated  function $\min\{k, \widetilde{R \circ \pi_e}\}$. This function is bounded in $\mathcal{B}$, with $||\min\{k, \widetilde{R \circ \pi_e}\}||_l=0$ and $||\min\{k, \widetilde{R \circ \pi_e}\}||_{\mathcal{B}} \le k$. By (\ref{expdecorrelation}), for any $p>p'>1, n\ge 1$ we have \begin{align*}
    &\big|\big|P^n\big(\min\{k, \widetilde{R \circ \pi_e}\}-\mathbb{E}\min\{k, \widetilde{R \circ \pi_e}\}\big)\big|\big|_{p}\precsim_p k \cdot e^{-Cn}, \\
    &\implies \big|\big|P^n\big(\min\{k, \widetilde{R \circ \pi_e}\}-\mathbb{E}\min\{k, \widetilde{R \circ \pi_e}\}\big)\big|\big|^p_{p}\precsim_{p,p'} k^p \cdot n^{-p'}.
\end{align*} 

It follows from the Theorem \ref{mld} that for any $N\ge 1$ \begin{equation*}
    \Big|\Big|\sup_{n \ge N}\Big|\frac{\sum_{i \le n}\big(\min\{k, \widetilde{R \circ \pi_e}\}-\mathbb{E}\min\{k, \widetilde{R \circ \pi_e}\}\big) \circ \widetilde{F_e}^i}{n}\Big|\Big|\Big|^{2p}_{2p}\precsim_{p,p'} k^{2p} \cdot n^{-p'}.
\end{equation*}

Then, together with (\ref{allexpmeasure}) and the maximal ergodic theorem, we obtain that
\begin{align*}
&\mu_{X}\Big\{\sup_{n \ge N}\Big|\frac{\sum_{i \le n}\big({R }-\mathbb{E} {R}\big) \circ (f^R)^i}{n}\Big|>\epsilon \Big\}\\
    &=\mu_{\Delta_e}\Big\{\sup_{n \ge N}\Big|\frac{\sum_{i \le n}\big({R \circ \pi_e}-\mathbb{E} {R \circ \pi_e}\big) \circ {F_e}^i}{n}\Big|>\epsilon \Big\}\\
    &= \mu_{\widetilde{\Delta_e}}\Big\{\sup_{n \ge N}\Big|\frac{\sum_{i \le n}\big(\widetilde{R \circ \pi_e}-\mathbb{E} \widetilde{R \circ \pi_e}\big) \circ \widetilde{F_e}^i}{n}\Big|>\epsilon \Big\}\\
    &\precsim \mu_{\widetilde{\Delta_e}}\Big\{\sup_{n \ge N}\Big|\frac{\sum_{i \le n}\big(\min\{k, \widetilde{R \circ \pi_e}\}-\mathbb{E}\min\{k, \widetilde{R \circ \pi_e}\}\big) \circ \widetilde{F_e}^i}{n}\Big|> \epsilon/2\Big\}\\
    &\quad +\mu_{\widetilde{\Delta_e}}\Big\{\sup_{n \ge N}\Big|\frac{\sum_{i \le n}\Big[\widetilde{R \circ \pi_e}-\min\{k, \widetilde{R \circ \pi_e}\}-\mathbb{E}\big(\widetilde{R \circ \pi_e}-\min\{k, \widetilde{R \circ \pi_e}\}\big)\Big] \circ \widetilde{F_e}^i}{n}\Big|> \epsilon/2\Big\}\\
      &\precsim \mu_{\widetilde{\Delta_e}}\Big\{\sup_{n \ge N}\Big|\frac{\sum_{i \le n}\big(\min\{k, \widetilde{R \circ \pi_e}\}-\mathbb{E}\min\{k, \widetilde{R \circ \pi_e}\}\big) \circ \widetilde{F_e}^i}{n}\Big|>\epsilon/2\Big\}\\
    &\quad +\mu_{\widetilde{\Delta_e}}\Big\{\sup_{n \ge N}\Big|\frac{\sum_{i \le n}\big[\widetilde{R \circ \pi_e}-\min\{k, \widetilde{R \circ \pi_e}\}\big] \circ \widetilde{F_e}^i}{n}\Big|> \epsilon/2\Big\}\\
    &\precsim_{p,p',\epsilon} k^{2p} \cdot N^{-p'}+\big|\big|\widetilde{R \circ \pi_e}-\min\{k, \widetilde{R \circ \pi_e}\}\big|\big|_1\precsim_{p,p',\epsilon} k^{2p} \cdot N^{-p'}+ k^{-\xi}.
\end{align*} 

For any fixed $\delta \in (1/2,1)$ choose $k=N^{1-\delta}$, $p>\xi/(2\delta-1)$ and $p'=\xi+2p(1-\delta)$.  Then the condition $\gcd\{R_e\}=1$ concludes a proof. 

If $N_e:=\gcd\{R_e\}>1$, then \begin{equation*}\mu_X\Big(\sup_{n\ge N}\Big|\frac{\sum_{i \le n}(R-\mathbb{E}R)\circ (f^R)^{iN_e}}{n}\Big|\ge \epsilon\Big)\precsim_{\delta, \epsilon} N^{-\xi(1-\delta)}.
\end{equation*}

Using that $(f^R)_{*} \mu_X=\mu_X$, we have for any $k\in [0, N_e)$ 
\[\mu_X\Big(\sup_{n\ge N}\Big|\frac{\sum_{i \le n}(R-\mathbb{E}R)\circ (f^R)^{iN_e+k}}{n}\Big|\ge \epsilon\Big)\precsim_{\delta,\epsilon} N^{-\xi(1-\delta)}.\]

Now summation over $k=0,1,\cdots N_e-1$ concludes a proof.
\end{proof}

\subsection{Poisson approximations for polynomial hyperbolic Young towers}

Throughout this section, in order to simplify notations, we denote $H_r:=\pi^{-1}(B_r(q)\times S^1)$, $\pmb{X_i}:=\mathbbm{1}_{ H_r} \circ F^i$, $n:=\left\lfloor \frac{T}{\mu_{\Delta}(H_r)}\right\rfloor$. Also denote by $h \in [0,1]$ a function $h$ with values in $[0,1]$. 

Let $\{\pmb{\hat{X}_i}\}_{i\ge 0}$ be i.i.d. random variables, such that  $\pmb{X_i}=_d\pmb{\hat{X}_i}$ for each $i \ge 0$. For any interval $J\subseteq [0,T]$, define $\pmb{X_J}:=(\pmb{X_j})_{j \in J}$, $\pmb{\hat{X_J}}:=(\pmb{\hat{X_j}})_{j \in J}$ and $J':=J/\mu_{\Delta}(H_r):=\{x:x\mu_{\Delta}(H_r)\in J\}\subseteq [0,n]$. By applying Theorems 2 and 3 of \cite{chenmethod} to $(\pmb{\hat{X}_{i}})_{i \ge 0}$ one gets for any disjoint intervals $J_1, \cdots, J_m \subseteq [0,T]$ that \begin{align*}
    \sup_{h\in[0,1]} \Big|\mathbb{E}h\Big(\mathcal{P}(J_1), \cdots, \mathcal{P}(J_m)\Big)-h\left(\pmb{\hat{X}_{J'_1}}, \cdots, \pmb{\hat{X}_{J'_m}}\right)\Big|\le 4  n  \mu_{\Delta}(H_r)^2 \precsim_T \mu_{\Delta}(H_r).
\end{align*} 

By making use of the Definition \ref{fpl} we obtain the following  functional Poisson approximation
\begin{align}
    d_{TV}\left(\mathcal{N}^{r,q,T},\mathcal{P}\right)&\le\sup_{\text{disjoint }J_i \subseteq [0,T], h\in[0,1]} \Big|\mathbb{E}h\Big(\mathcal{P}(J_1), \cdots, \mathcal{P}(J_m)\Big)-h\Big(\mathcal{N}^{r,q,T}(J_1), \cdots, \mathcal{N}^{r,q,T}(J_m)\Big)\Big|\nonumber\\
    &=\sup_{\text{disjoint }J_i \subseteq [0,T], h\in[0,1]} \Big|\mathbb{E}h\Big(\mathcal{P}(J_1), \cdots, \mathcal{P}(J_m)\Big)-h\left(\pmb{\hat{X}_{J'_1}}, \cdots, \pmb{\hat{X}_{J'_m}}\right)\Big|\nonumber\\
    &\quad +\sup_{\text{disjoint }J_i \subseteq [0,T], h\in[0,1]} \Big|\mathbb{E}h\left(\pmb{{X}_{J'_1}}, \cdots, \pmb{{X}_{J'_m}}\right)-h\left(\pmb{\hat{X}_{J'_1}}, \cdots, \pmb{\hat{X}_{J'_m}}\right)\Big|\nonumber\\
    & \precsim_T \mu_{\Delta}(H_r)+ \sup_{h\in [0,1]}|\mathbb{E}h(\pmb{X_1}, \cdots \pmb{X_n})-\mathbb{E}h(\pmb{\hat{X}_1}, \cdots, \pmb{\hat{X}_n})|\nonumber\\
    & \precsim_T \mu_{\mathcal{M}}\big(B_r(q)\times S^1\big)+ \sup_{h\in [0,1]}|\mathbb{E}h(\pmb{X_1}, \cdots \pmb{X_n})-\mathbb{E}h(\pmb{\hat{X}_1}, \cdots, \pmb{\hat{X}_n})|.\label{totalconvergencerate}
    \end{align}

Now we will deal with $\sup_{h\in [0,1]}|\mathbb{E}h(\pmb{X_1}, \cdots \pmb{X_n})-\mathbb{E}h(\pmb{\hat{X}_1}, \cdots, \pmb{\hat{X}_n})|$. Let
\[A_N:=\Big\{x\in X: \frac{R}{N} \ge \epsilon \text{ or }\sup_{n\ge N}\big|\frac{\sum_{i < n}(R-\mathbb{E}R)\circ (f^R)^{iN_e+k}}{n}\big|\ge \epsilon \Big\} \times \mathbb{N} \bigcap \Delta.\]

Then for any $k\in \mathbb{N}$,
\begin{align*}
    A_N \subseteq &\Big\{x\in X: \frac{R}{N} \ge \epsilon \text{ or }\sup_{n\ge N}\big|\frac{\sum_{i < n}(R-\mathbb{E}R)\circ (f^R)^{iN_e+k}}{n}\big|\ge \epsilon \Big\} \times \{0,1,2,\cdots k-1\}\\
    & \bigcup \{(x,n)\in \Delta: x\in X, n \ge k\}.
\end{align*}
    Next, by Lemma \ref{mldforR}, we have \[\mu_{\Delta}(A_N)\precsim_{\epsilon, \delta} \sup_{k}\Big\{\max\{N^{-1-\xi}, N^{-\xi(1-\delta)}\}\cdot k+k^{-\xi}\Big\}\precsim_{\epsilon, \delta}N^{-\xi^2(1-\delta)/(1+\xi)}.\]

Let $m(N) \in \mathbb{N}$ (where $m(N)\ge N$, will be determined later). Then $\mu_{\Delta}(F^{m(N)}A_N)=\mu_{\Delta}(A_N)\precsim N^{-\xi^2(1-\delta)/(1+\xi)}$.
\begin{lemma}\label{truncatehole}
Let $N=\big\lfloor j^{(4+4\xi)/[\xi^2(1-\delta)]}\big\rfloor$. For $\Leb_{\partial Q}$-a.s. $q\in \partial Q$,  \[\sup_{r>0}\frac{1}{\mu_{\Delta}(H_r)}\int_{H_r} \mathbbm{1}_{F^{m(N)}A_N}d\mu_{\Delta}\precsim_{q,\delta} N^{-\xi^2(1-\delta)/(2+2\xi)}.\]
\end{lemma}
\begin{proof}
By the Hardy-Littlewood maximal inequality, 
\[\sup_{\epsilon>0}\epsilon \Leb_{\partial Q}\Big\{q\in \partial Q: \sup_r\frac{1}{\mu_{\Delta}(H_r)}\int_{H_r} \mathbbm{1}_{F^{m(N)}A_N}d\mu_{\Delta}\ge \epsilon \Big\}\precsim \mu_{\Delta}(F^{m(N)}A_N)\precsim_{\delta} N^{-\xi^2(1-\delta)/(1+\xi)}. \]

Choose $\epsilon=N^{-\xi^2(1-\delta)/(2+2\xi)} , N=\big\lfloor j^{(4+4\xi)/[\xi^2(1-\delta)]}\big\rfloor$. Then by Borel-Cantelli lemma, for $\Leb$-a.e. $q\in \partial Q$, there is $j_q \in \mathbb{N}$, such that for any $j \ge j_q$ we have  \[\sup_r\frac{1}{\mu_{\Delta}(H_r)}\int_{H_r} \mathbbm{1}_{F^{m(N)}A_N}d\mu_{\Delta}\le N^{-\xi^2(1-\delta)/(2+2\xi)},\]

and the lemma is proved.
\end{proof}

This lemma shows that \[\mu_{\Delta}(H_r \bigcap F^{m(N)}A^c_N) \approx_q \mu_{\Delta}(H_r)\approx r, and \quad \bar{n}:=\left\lfloor \frac{T}{\mu_{\Delta}(H_r\bigcap F^{m(N)}A_N^c)}\right\rfloor \approx_q n.\] 

Let $X_i':=\mathbbm{1}_{F^{m(N)}A^c_N\bigcap H_r} \circ F^i$, and $\{{\hat{X}_i}\}_{i\ge 0}$ be i.i.d. random variables, such that  ${X'_i}=_d{\hat{X}_i}$ for each $i \ge 0$. Throughout this section, $N$ will be the same as in Lemma \ref{truncatehole}.

\begin{lemma}\label{subtract}For $\Leb_{\partial Q}$-a.e. $q\in \partial Q$, the following inequalities hold.
\begin{gather*} \sup_{h\in[0,1]}|\mathbb{E}h(X'_0, X'_1,  \cdots, X'_n)-\mathbb{E}h(\pmb{X_0}, \pmb{X_1},  \cdots, \pmb{X_n})|\precsim_{q,\delta} N^{-\xi^2(1-\delta)/(2+2\xi)},\\
\sup_{h\in[0,1]}|\mathbb{E}h(\hat{X}_0, \hat{X}_1,  \cdots, \hat{X}_n)-\mathbb{E}h(\pmb{\hat{X}_0}, \pmb{\hat{X}_1},  \cdots, \pmb{\hat{X}_n})|\precsim_{q,\delta} N^{-\xi^2(1-\delta)/(2+2\xi)},\\
        \sup_{h\in [0,1]}|\mathbb{E}h(\pmb{X_1}, \cdots \pmb{X_n})-\mathbb{E}h(\pmb{\hat{X}_1}, \cdots, \pmb{\hat{X}_n})|\precsim_{q,\delta} \sup_{h\in [0,1]}|\mathbb{E}h(X_1, \cdots X_n)-\mathbb{E}h(\hat{X}_1, \cdots, \hat{X}_n)|\\
        \quad \quad \quad \quad \quad \quad \quad \quad \quad \quad \quad  +N^{-\xi^2(1-\delta)/(2+2\xi)}.
      \end{gather*}
\end{lemma}
\begin{proof}At first note that $|X'_j-\pmb{X_j}|\le \mathbbm{1}_{H_r\bigcap F^{m(N)}A_N}\circ F^i$. Then
\begin{align*}
    &|\mathbb{E}h(X'_0, X'_1,  \cdots, X'_n)-\mathbb{E}h(\pmb{X_0}, \pmb{X_1},  \cdots, \pmb{X_n})|\\
    &\le \sum_{j \le n} |\mathbb{E}h(X'_0,\cdots X'_j, \pmb{X_{j+1}}\cdots, \pmb{X_n})-\mathbb{E}h(X'_0,\cdots X'_{j-1},\pmb{X_j} \cdots, \pmb{X_n})|\\
    &\le \sum_{j \le n}|\mathbb{E}X'_j-\pmb{X_j}|\precsim n \mu_{\Delta}(H_r \bigcap F^{m(N)}A_N)\precsim_{q,\delta} N^{-\xi^2(1-\delta)/(2+2\xi)}
\end{align*}holds for $\Leb_{\partial Q}$-a.e. $q \in \partial Q$ due to the Lemma \ref{truncatehole}.
Finally the relations
\begin{align*}
    \sup_{h\in[0,1]}&|\mathbb{E}h(\hat{X}_0, \hat{X}_1,  \cdots, \hat{X}_n)-\mathbb{E}h(\pmb{\hat{X}_0}, \pmb{\hat{X}_1},  \cdots, \pmb{\hat{X}_n})|\\
    &\le \sum_{j \le n} \sup_{h\in[0,1]}|\mathbb{E}h(\hat{X}_0,\cdots \hat{X}_j, \pmb{\hat{X}_{j+1}}\cdots, \pmb{\hat{X}_n})-\mathbb{E}h(\hat{X}_0,\cdots \hat{X}_{j-1},\pmb{\hat{X}_j} \cdots, \pmb{\hat{X}_n})|\\
    &\le \sum_{j \le n}\sup_{h\in [0,1]}|\mathbb{E}h(\hat{X}_j)-\mathbb{E}h(\pmb{\hat{X}_j})|\\
    &=\sum_{j \le n}\sup_{h\in [0,1]}|\mathbb{E}h(X'_j)-\mathbb{E}h(\pmb{X_j})|\\
    &\precsim n \mu_{\Delta}(H_r \bigcap F^{m(N)}A_N)\precsim_{q,\delta} N^{-\xi^2(1-\delta)/(2+2\xi)}
\end{align*}hold for $\Leb_{\partial Q}$-a.e. $q \in \partial Q$ due to the Lemma \ref{truncatehole}.
\end{proof}

This lemma shows that in order to estimate $\sup_{h\in [0,1]}|\mathbb{E}h(\pmb{X_1}, \cdots \pmb{X_n})-\mathbb{E}h(\pmb{\hat{X}_1}, \cdots, \pmb{\hat{X}_n})|$, we just need to estimate $\sup_{h\in [0,1]}|\mathbb{E}h(X'_1, \cdots X'_n)-\mathbb{E}h(\hat{X}_1, \cdots, \hat{X}_n)|$. By Lemma 1 of \cite{Su}, for any $p\in (2m(N),n]$ we have that $\sup_{h\in [0,1]}|\mathbb{E}h(X'_1, \cdots X'_n)-\mathbb{E}h(\hat{X}_1, \cdots, \hat{X}_n)|\precsim_T R_1+R_2+R_3$, where \begin{align}
    &R_1:=\sum_{0 \le l \le n-p}\sup_{h\in[0,1]} \Big| \mathbb{E}\Big[\mathbbm{1}_{X'_0=1}  h(X'_p,\cdots, X'_{p+l})\Big]-\mathbb{E}\mathbbm{1}_{X'_0=1}  \mathbb{E}h(X'_p,\cdots, X'_{p+l})\Big| \label{poissoncorrelation}\\
    &R_2:=n\cdot \mathbb{E} \left(\mathbbm{1}_{X'_0=1}  \mathbbm{1}_{\sum_{1\le j \le p-1}X'_j\ge 1}\right) \label{shortreturn}\\
    &R_3:=p\cdot n \cdot  \mu_{\Delta}\big(X'_0=1)^2+ p\cdot \mu_{\Delta}\big(X'_0=1\big)\precsim_q p\cdot \mu_{\Delta}\big(X'_0=1\big) .\nonumber
\end{align} 
\begin{remark} 
The quantities $R_1, R_2, R_3$ in Lemma 1 of \cite{Su} are represented in terms of $B_r(z)$. However, as it was shown above, exactly the same arguments give similar representations of $R_1, R_2, R_3$ in terms of $\{X_0'=1\}$.
\end{remark}

The relation (\ref{shortreturn}) is called a short return. In the following subsection we will apply a polynomial hyperbolic Young tower $(\Delta_p, F_p, \mu_{\Delta_p})$ for estimation of convergence rates in (\ref{poissoncorrelation}).
\begin{remark}\label{gcdremark}
The paper \cite{Subbb} established that $\sup_{h\in [0,1]}|\mathbb{E}h(X'_1, \cdots X'_n)-\mathbb{E}h(\hat{X}_1, \cdots, \hat{X}_n)|$ is dominated by $R_1,R_2,R_3$, even though $(\Delta_p, F_p, \mu_{\Delta_p})$ is not necessarily mixing, i.e., $\gcd \{R_p\}\ge 1$. Therefore, without loss of generality, in the following subsections we assume that $\gcd\{R_e\}=\gcd\{R_p\}=1$. Hence (\ref{expdecorrelation}) and (\ref{polydecorrelation}) hold. 
\end{remark}

\subsubsection{Convergence rates for (\ref{poissoncorrelation})}
Throughout this subsection the relation $Q\in \mathcal{Q}_{2m(N)}\bigcap \pi^{-1}_pA_N^c$ will mean that there is $Q' \in \mathcal{Q}_{2m(N)}$, such that $Q=Q'\bigcap \pi^{-1}_pA_N^c$. We choose now $Q \in \mathcal{Q}_{2m(N)}\bigcap \pi^{-1}_pA_N^c$ (see section \ref{polytower}) so that $Q \bigcap \pi_p^{-1}F^{-m(N)}(H_r \bigcap F^{m(N)}A_N^c)\neq \emptyset$. By collecting all such $Q$ we have 
\[\pi_p^{-1}F^{-m(N)}(H_r \bigcap F^{m(N)}A_N^c) \subseteq \bigcup Q, \quad F^{m(N)}\pi_pQ \bigcap H_r\neq \emptyset, \quad Q \subseteq \pi_p^{-1}A_N^c. \] 

Denote $\overline{A}:=\bigcup Q, A:=\pi_p^{-1}F^{-m(N)}(H_r \bigcap F^{m(N)}A_N^c)$, $\partial A:=\bigcup_{Q \bigcap \bar{A}\setminus A}\neq \emptyset $, then $\overline{A}\setminus A\subseteq \partial A$. By Assumption \ref{assumption} $\partial X_i \subseteq \mathbb{S}$ for each $i\ge 1$. Then for each $i\in [0, 2m(N)]$ the set $F^i(\pi_p Q)$ is entirely contained in one of the sets $(X_j\times \mathbb{N}) \bigcap \Delta$. 

\begin{lemma}\label{numberreturntoX}
Let $m(N)> (\mathbb{E}R+\epsilon)N, N>7, and \epsilon< \mathbb{E}R/18$. Suppose that  $\pi_p Q, F(\pi_p Q), \cdots, F^{m(N)}(\pi_p Q)$ had $a$ visits to $X$. Then \[a\in \Big[\frac{7m(N)}{8(\mathbb{E}R+\epsilon)}, \frac{m(N)}{\mathbb{E}R-2\epsilon}\Big].\]

If $ F^{m(N)}(\pi_p Q), \cdots, F^{2m(N)}(\pi_p Q)$ visit $X$ now $b$ times. Then \[b\in \Big[ \frac{m(N)}{2(\mathbb{E}R-2\epsilon)}, \frac{2m(N)}{\mathbb{E}R-2\epsilon}\Big].\]

\end{lemma}
\begin{proof}
Recall that $Q=Q'\bigcap \pi^{-1}_pA_N^c$ for some $Q'\in \mathcal{Q}_{2m(N)}$. Thus the same holds for $Q'$.
Observe that different elements in $\pi_pQ'$ share the same values $a,b$. Indeed, for each $k\in [0, m(N)]$, $F^k(\pi_p Q')$ is entirely contained in one of the sets $(X_j\times \mathbb{N}) \bigcap \Delta$, and different elements in $\pi_pQ'$ return to $X$ at the same time. Since $Q \subseteq Q'$, then $Q$ shares the same $a,b$.

Since $\pi_p Q \bigcap A_N^c \neq \emptyset$, then there is $w \in \pi_p Q \bigcap A_N^c$, and $\pi_{X}w \in F^{-\pi_{\mathbb{N}}(w)}\pi_pQ \bigcap \pi_X A_N^c \subsetneq X$. Therefore for any $n\ge N$,
\[\sup_{n\ge N}\Big|\frac{R^n(\pi_Xw)}{n}-\mathbb{E}R\Big|\le \epsilon, \quad  \frac{R(\pi_Xw)}{n}\le \frac{R(\pi_Xw)}{N}\le \epsilon, \]
where $R^n:=\sum_{i < n}R\circ (f^R)^i$.

\textbf{Claim.} If $j > (\mathbb{E}R+\epsilon)N$ and $j\in [R^m(\pi_Xw), R^{m+1}(\pi_Xw)]$ for some $m\ge 1$, then $m \ge N$.

This statement holds because if $R^N(\pi_Xw) \in [(\mathbb{E}R-\epsilon) N, (\mathbb{E}R+\epsilon) N]$, then $j > R^N(\pi_Xw)$. If now $m<N$, then $j \le R^{m+1}(\pi_Xw)\le R^N(\pi_Xw)$. This contradiction concludes a proof of the claim.

Note that $m(N)+\pi_{\mathbb{N}}(w)\ge m(N) >(\mathbb{E}R+\epsilon)N$. Then \begin{align*}
    &m(N)+\pi_{\mathbb{N}}(w)\in [R^a(\pi_Xw), R^{a+1}(\pi_Xw)],\\
    & m(N)\in [R^a(\pi_Xw)-R(\pi_Xw), R^{a+1}(\pi_Xw)],
\end{align*} and
\begin{align*}
    &2m(N)+\pi_{\mathbb{N}}(w)\in [R^{b+a}(\pi_Xw), R^{b+a+1}(\pi_Xw)],\\
    &2m(N)\in [R^{b+a}(\pi_Xw)-R(\pi_Xw), R^{b+a+1}(\pi_Xw)].
\end{align*}
 
These relations imply that\begin{align*}
    &\frac{m(N)}{a}\in \Big[\frac{R^a(\pi_Xw)-R(\pi_Xw)}{a}, \frac{R^{a+1}(\pi_Xw)}{a}\Big]\subseteq [\mathbb{E}R-2\epsilon, (\mathbb{E}R+\epsilon)(1+N^{-1})],\\
    &\implies a \in \Big[\frac{m(N)}{(\mathbb{E}R+\epsilon)(1+N^{-1})}, \frac{m(N)}{\mathbb{E}R-2\epsilon}\Big]\subseteq \Big[\frac{7m(N)}{8(\mathbb{E}R+\epsilon)}, \frac{m(N)}{\mathbb{E}R-2\epsilon}\Big].
\end{align*}

\begin{align*}
    &\frac{2m(N)}{b+a}\in \Big[\frac{R^{b+a}(\pi_Xw)-R(\pi_Xw)}{b+a}, \frac{R^{b+a+1}(\pi_Xw)}{b+a}\Big]\subseteq [\mathbb{E}R-2\epsilon, (\mathbb{E}R+\epsilon)(1+N^{-1})],\\
    &\implies b+a \in \Big[\frac{2m(N)}{(\mathbb{E}R+\epsilon)(1+N^{-1})}, \frac{2m(N)}{\mathbb{E}R-2\epsilon}\Big].
\end{align*}

Therefore,
\begin{align*}
    b&\in \Big[\frac{2m(N)}{(\mathbb{E}R+\epsilon)(1+N^{-1})}-\frac{m(N)}{\mathbb{E}R-2\epsilon}, \frac{2m(N)}{\mathbb{E}R-2\epsilon}-\frac{m(N)}{(\mathbb{E}R+\epsilon)(1+N^{-1})} \Big]\\
    &\subseteq \Big[ \frac{m(N)}{2(\mathbb{E}R-2\epsilon)}, \frac{2m(N)}{\mathbb{E}R-2\epsilon}\Big],
\end{align*}where ``$\subseteq$" is due to $ N>7, \epsilon< \mathbb{E}R/18$, and  $\frac{2m(N)}{(\mathbb{E}R+\epsilon)(1+N^{-1})}-\frac{m(N)}{\mathbb{E}R-2\epsilon} \ge \frac{m(N)}{2(\mathbb{E}R-2\epsilon)}$.
\end{proof}
Now we can estimate a size of $\pi F^{m(N)}(\pi_pQ)$. In what follows we assume that $m(N)> (\mathbb{E}R+\epsilon)N, N>7, \epsilon< \mathbb{E}R/18$.
\begin{lemma}
$\diam \pi F^{m(N)}(\pi_pQ) \precsim  \beta^{\frac{m(N)\alpha}{2\mathbb{E}R-4\epsilon}}$.
\end{lemma}
\begin{proof}
Recall that $Q=Q'\bigcap \pi^{-1}_pA_N^c$ for some $Q'\in \mathcal{Q}_{2m(N)}$. By Lemma \ref{numberreturntoX}, any $\gamma^s \subsetneq \pi_pQ'$ visits $X$ at least $\frac{7m(N)}{8(\mathbb{E}R+\epsilon)}$ times in the time window $[0, m(N)]$, and any $\pi F^{m(N)}\gamma^u \subsetneq \pi F^{m(N)}(\pi_pQ')$ visits $X$ at least $\frac{m(N)}{2(\mathbb{E}R-2\epsilon)}$ times in the time window $[m(N), 2m(N)]$. Using the Assumption \ref{assumption} and the Definition \ref{cmz} we have for any $\gamma^u, \gamma^s \subsetneq \pi_pQ'$ that
\[\diam \pi F^{m(N)}\gamma^s\precsim \beta^{\frac{7m(N)\alpha}{8(\mathbb{E}R+\epsilon)}},\quad \diam \pi F^{m(N)}\gamma^u\precsim \beta^{\frac{m(N)\alpha}{2(\mathbb{E}R-2\epsilon)}}.\]

Since $\pi_pQ'$ has a product structure, then $\gamma^u$ and $\gamma^s$ intersect  exactly at one point in $\pi_pQ'$. Therefore \begin{align*}
    \diam \pi F^{m(N)}(\pi_pQ)&\le \diam \pi F^{m(N)}(\pi_pQ')\\
    &\le \sup_{\gamma^u, \gamma^s \subsetneq \pi_pQ'}\{\diam \pi F^{m(N)}\gamma^s+\diam \pi F^{m(N)}\gamma^u\}\precsim \beta^{\frac{m(N)\alpha}{2\mathbb{E}R-4\epsilon}},
\end{align*}

which concludes a proof of lemma.
\end{proof}

We can estimate now the measure of $\pi F^{m(N)}\pi_p \partial A$.
\begin{lemma}\label{boundarymeasure}
$\mu_{\mathcal{M}}(\pi F^{m(N)}\pi_p \partial A)\precsim \beta^{\frac{m(N)\alpha}{2\mathbb{E}R-4\epsilon}}$.
\end{lemma}
\begin{proof}
For any $Q \subseteq \partial A$ there are $x,y \in Q \subseteq \pi_p^{-1}A_N^c$, such that \[\pi F^{m(N)}\pi_px \in \pi F^{m(N)}A_N^c \bigcap  B_r(q) \times S^1, \quad \pi F^{m(N)}\pi_py \notin \pi F^{m(N)}A_N^c \bigcap B_r(q) \times S^1\]
Further, there exists such constant $C>0$, that \[\diam(\pi F^{m(N)}\pi_px,\pi F^{m(N)}\pi_py)\le \diam \pi F^{m(N)}\pi_pQ \le C \beta^{\frac{m(N)\alpha}{2\mathbb{E}R-4\epsilon}}.\]

Therefore, \[\pi F^{m(N)}\pi_pQ \subseteq \Big[B_{r+C \beta^{\frac{m(N)\alpha}{2\mathbb{E}R-4\epsilon}}}(q) \setminus B_{r-C \beta^{\frac{m(N)\alpha}{2\mathbb{E}R-4\epsilon}}}(q)\Big]\times S^1,\]which implies that $\mu_{\mathcal{M}}(\pi F^{m(N)}\pi_p \partial A)\precsim \beta^{\frac{m(N)\alpha}{2\mathbb{E}R-4\epsilon}}$.
\end{proof}

We can estimate now a convergence rate for (\ref{poissoncorrelation}).

\begin{lemma}\label{rateforpoissoncorrelation}
Let $m(N)> (\mathbb{E}R+\epsilon)N, N>7, \epsilon< \mathbb{E}R/18$, then (\ref{poissoncorrelation}) $\precsim_T (p-2m(N))^{-\xi}+n^2\beta^{\frac{m(N)\alpha}{2\mathbb{E}R-4\epsilon}}$.
\end{lemma}
\begin{proof} It follows from (\ref{allpolymeasure}) that
\begin{align*}
   &\left| \mathbb{E}\Big[\mathbbm{1}_{X'_0=1}  h(X'_p,\cdots, X'_{p+l})\Big]-\mathbb{E}\mathbbm{1}_{X'_0=1}  \mathbb{E}h(X'_p,\cdots, X'_{p+l})\right|\\
   &=\left|\int \mathbbm{1}_{A}h(\mathbbm{1}_{A}\circ F_p^p, \cdots, \mathbbm{1}_{A}\circ F^{p+l}_p)d\mu_{\Delta_p}-\int \mathbbm{1}_{A}d\mu_{\Delta_p}\int h(\mathbbm{1}_{A}\circ F_p^p, \cdots, \mathbbm{1}_{A}\circ F^{p+l}_p)d\mu_{\Delta_p}\right|\\
   &=\Big|\int \mathbbm{1}_{\overline{A}}  h\left(\mathbbm{1}_{\overline{A}},\cdots,\mathbbm{1}_{\overline{A}}\circ F_p^{l}\right) \circ F_p^p d\mu_{\Delta_p}- \mu_{\Delta_p}(\overline{A}) \int h\left(\mathbbm{1}_{\overline{A}},\cdots,\mathbbm{1}_{\overline{A}}\circ F_p^{l}\right)  d\mu_{\Delta_p}\\
   & \quad -\int \mathbbm{1}_{\overline{A}\setminus A}  h\left(\mathbbm{1}_{\overline{A}},\cdots,\mathbbm{1}_{\overline{A}}\circ F_p^{l}\right) \circ F_p^p d\mu_{\Delta_p}+ \mu_{\Delta_p}(\overline{A}\setminus A) \int h\left(\mathbbm{1}_{\overline{A}},\cdots,\mathbbm{1}_{\overline{A}}\circ F_p^{l}\right)  d\mu_{\Delta_p}\\
   &\quad+\int \mathbbm{1}_{A} h\left(\mathbbm{1}_{A},\cdots, \mathbbm{1}_{A}\circ F_p^{l}\right) \circ F_p^p d\mu_{\Delta_p}-\int \mathbbm{1}_{A}  h\left(\mathbbm{1}_{\overline{A}},\cdots,\mathbbm{1}_{\overline{A}}\circ F_p^{l}\right) \circ F_p^p d\mu_{\Delta_p}\\
   &\quad-\int \mathbbm{1}_{A} d\mu_{\Delta_p}  \int h\left(\mathbbm{1}_{A},\cdots, \mathbbm{1}_{A}\circ F_p^{l}\right) \circ F_p^p d\mu_{\Delta_p}+\mu_{\Delta_p}(A) \int h\left(\mathbbm{1}_{\overline{A}},\cdots,\mathbbm{1}_{\overline{A}}\circ F_p^{l}\right)  d\mu_{\Delta_p}\Big|\\
   &\le \left|\int \mathbbm{1}_{\overline{A}}  h\left(\mathbbm{1}_{\overline{A}},\cdots,\mathbbm{1}_{\overline{A}}\circ F_p^{l}\right) \circ F_p^p d\mu_{\Delta_p}-\mu_{\Delta_p}(\overline{A})  \int h\left(\mathbbm{1}_{\overline{A}},\cdots,\mathbbm{1}_{\overline{A}}\circ F_p^{l}\right)  d\mu_{\Delta_p}\right|\\
   &\quad+2\mu_{\Delta_p}(\overline{A}\setminus A)+2\int \mathbbm{1}_{A}\mathbbm{1}_{\bigcup_{j \le l}F_p^{-j}\overline{A}\setminus A}\circ F_p^p d\mu_{\Delta_p}+ 2\mu_{\Delta_p}(A)  \int \mathbbm{1}_{\bigcup_{j \le l}F_p^{-j}\overline{A}\setminus A} d\mu_{\Delta_p}\\
   &\precsim \left|\int \mathbbm{1}_{\overline{A}}  h\left(\mathbbm{1}_{\overline{A}},\cdots,\mathbbm{1}_{\overline{A}}\circ F_p^{l}\right) \circ F_p^p d\mu_{\Delta_p}-\mu_{\Delta_p}(\overline{A}) \int h\left(\mathbbm{1}_{\overline{A}},\cdots,\mathbbm{1}_{\overline{A}}\circ F_p^{l}\right) d\mu_{\Delta_p}\right| +l\mu_{\Delta_p}(\overline{A}\setminus A)\\
   &\precsim \mu_{\Delta_p}(\overline{A})(p-2m(N))^{-\xi} +n\mu_{\Delta_p}(\partial A)\\
   & \precsim \mu_{\Delta_p}(A)(p-2m(N))^{-\xi}+(p-2m(N))^{-\xi}\mu_{\Delta_p}(\partial A)+n \mu_{\Delta_p}(\partial A)\\
   &\precsim \mu_{\Delta}(H_r)(p-2m(N))^{-\xi}+n\mu_{\Delta_p}(\pi_p^{-1}F^{-m(N)}\pi^{-1}\pi F^{m(N)}\pi_p \partial A)\\
   &\precsim_T n^{-1}(p-2m(N))^{-\xi}+n\mu_{\mathcal{M}}(\pi F^{m(N)}\pi_p \partial A) \precsim_T n^{-1}(p-2m(N))^{-\xi}+n \beta^{\frac{m(N)\alpha}{2\mathbb{E}R-4\epsilon}},
\end{align*}where the last ``$\precsim_T$" is due to Lemma \ref{boundarymeasure}.

Therefore, (\ref{poissoncorrelation}) $\precsim_T n n^{-1}(p-2m(N))^{-\xi}+n^2\beta^{\frac{m(N)\alpha}{2\mathbb{E}R-4\epsilon}}\precsim_T (p-2m(N))^{-\xi}+n^2\beta^{\frac{m(N)\alpha}{2\mathbb{E}R-4\epsilon}}$.
\end{proof}

\subsubsection{Convergence rates for (\ref{shortreturn})}
In order to estimate convergence rate for (\ref{shortreturn}) we will study first the return statistic. Throughout this subsection we define for any $x,y \in \mathcal{M}$, \[\pmb{\pi}_X(x):=\pi_X(\pi^{-1}x), \quad \pmb{d}_{\partial Q}(x,y):=d_{\partial Q}(\pi_{\partial Q} x, \pi_{\partial Q}y).\]

Since the (un)stable manifolds are monotone curves in $X$, we denote the distance between $x\in X$ and the left endpoint of $\gamma^k(x)$ by $\diam_l \gamma^k(x)$, and the distance between $x$ and the right endpoint of $\gamma^k(x)$ by $\diam_r \gamma^k(x)$, where $k=u,s$.

For any non-negative numbers $j \le i$  denote
\[\Lambda_{r,i}:=\{x\in \mathcal{M}: \pmb{d}_{\partial Q}(x, f^{-i}x) \le r\},\]
\[\Lambda_{r,i,j}:=\{x\in \mathcal{M}: \pmb{d}_{\partial Q}(x, f^{-i}x) \le r, \text{ the orbit }x, f^{-1}x, \cdots, f^{-i}x \text{ visits } X \text{ exactly }j\text{ times} \}.\]

Observe that  $\Lambda_{r,i}:=\bigcup_j \Lambda_{r,i,j}$, and $\mu_{\mathcal{M}}(\Lambda_{r,i})=\sum_{j\le i}\mu_{\mathcal{M}}(\Lambda_{r,i, j})$. 

\begin{lemma}\label{lentghofmfd}
 There exists a constant $C>0$, such that $M\in \mathbb{N}$ for any $h>0$, and for $k=u,s$ we have 
 \[\mu_{X}\Big\{N_{h}\Big[\bigcup_{0 \le i \le M}(f^R)^{-i}\mathbb{S}\Big]\Big\} \precsim_M  h^{g\alpha^M},\]\[\mu_{X}(x \in X: \diam_l \gamma^k(x)<h) \le C h^g, \quad \mu_{X}(x \in X: \diam_r \gamma^k(x)<h) \le C h^g,\]
 where $g, \alpha$ are the same as in the Assumption \ref{assumption}, and $\gamma^u, \gamma^s$ are unstable and stable manifolds (see Definition \ref{unstablecone}).
\end{lemma}
\begin{proof}
By the Assumption \ref{assumption}, $\mu_X(N_{\epsilon}(\mathbb{S}))\le C \epsilon^g$. It is known that $\gamma^u \subsetneq (\bigcup_{i \ge 0}(f^R)^{i}\mathbb{S})^c$. Let $\diam_l \gamma^u(x)<h$. Then $\gamma^u(x)$ has a left endpoint in $(f^R)^i\mathbb{S}$ for some $i\ge 0$, i.e., $(f^R)^{-i}\gamma^u(x)$ has an endpoint in $\mathbb{S}$. The distance between this endpoint and $(f^R)^{-i}x$ is at most $C\beta^ih$ due to the Definition \ref{cmz}. In other words, $x\in (f^R)^iN_{C\beta^ih}(\mathbb{S})$ for some $i\ge 0$. Therefore
\[\mu_{X}(x\in X: \diam_l \gamma^u(x)<h) \le \mu_X\Big(\bigcup_{i\ge 0}(f^R)^iN_{C\beta^ih}\mathbb{S}\Big)\le \sum_{i\ge 0}\mu_X(N_{C\beta^ih}(\mathbb{S}))\precsim h^g.\]

The arguments for other cases are similar. Therefore we omit the corresponding proofs. 

Let $x\in N_{h}(\bigcup_{0 \le i \le M}(f^R)^{-i}\mathbb{S})$. Since one of the $\gamma^u(x)$, $\gamma^s(x)$ uniformly transverses  the curves in $(\bigcup_{0 \le i \le M}(f^R)^{-i}\mathbb{S})$ (see Assumption \ref{assumption}), then there exists a constant $C>0$, such that if $\diam_l \gamma^k(x)\ge Ch$, and $\diam_r \gamma^k(x)\ge Ch$ ($k=u$ and $s$), then one of the $\gamma^u(x), \gamma^s(x)$ must intersect $(\bigcup_{0 \le i \le M}(f^R)^{-i}\mathbb{S})$. 

Let $\gamma^u(x)$ intersect $(f^R)^{-i}\mathbb{S}$ (for some $i \in [0,M]$) at the point $y$. Denote the  unstable disk connecting $x$ and $y$ by $\gamma^u_{x,y} \subseteq \gamma^u(x)$. Then $\diam \gamma^u_{x,y} \le C'' h$ for some $C''>0$, and $(f^R)^{i}y \in \mathbb{S}$. By Assumption \ref{assumption}  there is a constant $C_M>0$, such that \[\diam (f^R)^i\gamma^u_{x,y} \precsim (\diam \gamma^u_{x,y})^{\alpha^i}\le C_M h^{\alpha^M}, \quad (f^R)^i x \in N_{C_M h^{\alpha^M}}(\mathbb{S}).\]

Let now $\gamma^s(x)$ intersect $(f^R)^{-i}\mathbb{S}$ (for some $i \in [0,M]$) at the point $y$. Denote the stable disk connecting $x$ and $y$ by $\gamma^s_{x,y} \subseteq \gamma^s(x)$. Then  $\diam \gamma^s_{x,y} \le C''' h$ for some $C'''>0$, and $(f^R)^{i}y \in \mathbb{S}$. According to the Definition \ref{cmz} there exists a constant $C'>0$, such that \[\diam (f^R)^i\gamma^u_{x,y} \precsim (\diam \gamma^u_{x,y})\le C' h, \quad (f^R)^i x \in N_{C' h}(\mathbb{S}).\]

By argueing as before we get that there is a constant $C_M'>0$, such that \[N_{h}\Big(\bigcup_{0 \le i \le M}(f^R)^{-i}\mathbb{S}\Big) \subseteq \bigcup_{j=r,l}\bigcup_{k=u,s} \{ x \in X:\diam_j \gamma^k(x)< Ch \}\bigcup \bigcup_{0\le i\le M}(f^R)^{-i}N_{C'_M h^{\alpha^M}}(\mathbb{S}).\]

Therefore, \begin{align*}
    \mu_X\Big(N_{h}(\bigcup_{0 \le i \le M}(f^R)^{-i}\mathbb{S})\Big) &\le \mu_X \Big(\bigcup_{j=r,l}\bigcup_{k=u,s} \{ x \in X:\diam_j \gamma^k(x)< Ch \}\Big)\\
    &\quad + \mu_X\Big(\bigcup_{0\le i\le M}(f^R)^{-i}N_{C'_M h^{\alpha^M}}(\mathbb{S})\Big)\\
    &\precsim_M h^{g}+ h^{g\alpha^M} \precsim_M h^{g\alpha^M},
\end{align*} which concludes a proof.
\end{proof}

Now we can estimate a measure of $\pmb{\pi}_X\Lambda_{r,i,j}$.

\begin{lemma}\label{1}
There is $J_K>0$ such that for any $i\ge j > J_K$, $\mu_X(\pmb{\pi}_X\Lambda_{r,i,j})\precsim_K r^{g/(g+1)}$, where $K, g$ are defined in Assumption \ref{assumption}.
\end{lemma}
\begin{proof}
For any $\gamma^u \subseteq X$, any $k < R|_{\gamma^u}$, and any $x,y \in f^k\gamma^u \bigcap \Lambda_{r,i,j}$ we have \[\pmb{d}_{\partial Q}(x,y) \le \pmb{d}_{\partial Q}(f^{-i}x,x)+\pmb{d}_{\partial Q}(f^{-i}y,y)+\pmb{d}_{\partial Q}(f^{-i}x,f^{-i}y) \le 2r+ \pmb{d}_{\partial Q}(f^{-i}x,f^{-i}y).  \]

By Assumption \ref{assumption} the slops of $f^k\gamma^u, f^{-i}f^k \gamma^u$ are uniformly bounded. Therefore
\[d_{f^k\gamma^u}(x,y) \precsim 2r+ d_{f^{-i}f^k\gamma^u}(f^{-i}x,f^{-i}y).  \]

Also by Assumption \ref{assumption}, if  $i \ge j > K+2$, then  \[d_{\gamma^u}(\pmb{\pi}_Xx,\pmb{\pi}_X y) \precsim d_{f^k\gamma^u}(x,y), \]
\[d_{f^{-i}f^k\gamma^u}(f^{-i}x,f^{-i}y) \precsim d_{(f^R)^K\pmb{\pi}_Xf^{-i}f^k\gamma^u}((f^R)^K\pmb{\pi}_Xf^{-i}x,(f^R)^K\pmb{\pi}_Xf^{-i}y)\precsim \beta^{j-K}d_{\gamma^u}(\pmb{\pi}_Xx,\pmb{\pi}_X y)\],
where the last ``$\precsim$" is due to the CMZ structure (see the Definition \ref{cmz}).

Therefore, there exists a constant $C>0$, such that \[d_{\gamma^u}(\pmb{\pi}_Xx,\pmb{\pi}_X y) \le 2Cr + C\beta^{j-K} d_{\gamma^u}(\pmb{\pi}_Xx,\pmb{\pi}_X y). \]

Then there is a constant $J_K>K+2$, such that $C\beta^{j-K}<C\beta^{J_K-K}< 1$ for any $j \ge J_K$, and \[d_{\gamma^u}(\pmb{\pi}_Xx,\pmb{\pi}_X y) \le \frac{2Cr}{1-C\beta^{J_K-K}}\]

By Definition \ref{cmz} all $\gamma^u$ have uniformly bounded curvatures. Hence
\[ \Leb_{\gamma^u}(\pmb{\pi}_X\Lambda_{r,i,j})\precsim_{K} r.\]

 From Lemma \ref{lentghofmfd} and the fact that the family of unstable manifold $\gamma^u$ is a measurable partition of $X$, we obtain \begin{align*}
  \mu_X(\pmb{\pi}_X\Lambda_{r,i,j})&=\int_{\diam \gamma^u(x)< h}\mu_{\gamma^u(x)}(\pmb{\pi}_X\Lambda_{r,i,j})d\mu_X +\int_{\diam \gamma^u(x)\ge h}\mu_{\gamma^u(x)}(\pmb{\pi}_X\Lambda_{r,i,j})d\mu_X \\
  &\precsim \mu_X(x \in X: \diam \gamma^u(x)<h)+\int_{\diam \gamma^u(x)\ge h}\mu_{\gamma^u(x)}(\pmb{\pi}_X\Lambda_{r,i,j})d\mu_X \\
  &\precsim h^g+\int_{\diam \gamma^u(x)\ge h}\frac{\Leb_{\gamma^u(x)}(\pmb{\pi}_X\Lambda_{r,i,j})}{\Leb_{\gamma^u(x)}(\gamma^u(x))}d\mu_X \\
  &\precsim_{K} h^g+h^{-1}r \precsim_K r^{g/(1+g)},
\end{align*}where was choosen $h=r^{1/(1+g)}$.
\end{proof}

 This lemma allows to simplify (\ref{shortreturn}). Recall that $S_r$ is a section in the $s$-quasi-section $B_r(q) \times S^1$ (see Definition \ref{quasisection}). Let
\begin{align*}
    \bar{\Lambda}_{r,j}(q):=&\{x\in S_r\subseteq B_r(q) \times S^1: \text{there is }i \ge 0 \text{ such that the orbit }x, fx, \cdots, f^{i}x \text{ visits } X\\
    &\text{ exactly }j\text{ times and }f^i(x) \in S_r \}.
\end{align*}

\begin{lemma}\label{2}
The value of (\ref{shortreturn}) equals $n\cdot \mathbb{E} \left(\mathbbm{1}_{X'_0=1}  \mathbbm{1}_{\sum_{1\le j \le p-1}X'_j\ge 1}\right)\le n\sum_{i\le p}\mathbb{E}\mathbbm{1}_{B_r(q)\times S^1} \mathbbm{1}_{B_r(q)\times S^1} \circ f^i$ \[\precsim_{q,K,T} p \cdot r^{s}+n\cdot p\cdot \sup_{1\le j \le J_K+1} \mathbb{E}\mathbbm{1}_{\pmb{\pi}_X\bar{\Lambda}_{3r,j}(q)} \mathbbm{1}_{ \pmb{\pi}_XS_{3r}}+n \cdot \sum_{J_K\le i\le p}\sum_{j > J_K}\mathbb{E}\mathbbm{1}_{\pmb{\pi}_X S_r} \mathbbm{1}_{\pmb{\pi}_X\Lambda_{2r,i, j}}.\]
\end{lemma}
\begin{proof}From the relation $f_{*} \mu_{\mathcal{M}}=\mu_{\mathcal{M}}$ we have 
\begin{align*}
    \mathbb{E}& \left(\mathbbm{1}_{X'_0=1}  \mathbbm{1}_{\sum_{1\le i \le p-1}X'_i\ge 1}\right) \le \sum_{i\le p} \mathbb{E}\mathbbm{1}_{X'_0=1} \mathbbm{1}_{X'_i=1}\\
    &\le \sum_{i\le p}\mathbb{E}\mathbbm{1}_{B_r(q)\times S^1} \mathbbm{1}_{B_r(q)\times S^1} \circ f^i\\
    &=\sum_{i\le p}\mathbb{E}\mathbbm{1}_{B_r(q)\times S^1} \mathbbm{1}_{B_r(q)\times S^1} \circ f^{-i}\\
    &\le \sum_{i\le p}\mathbb{E}\mathbbm{1}_{B_r(q)\times S^1} \mathbbm{1}_{\Lambda_{2r,i}}\\
    & \precsim_q p \cdot r^{s+1}+\sum_{i\le p}\mathbb{E}\mathbbm{1}_{\pi S_{r}} \mathbbm{1}_{\Lambda_{2r,i}}\\
    & \precsim_q p \cdot r^{s+1}+\sum_{i\le p}\mathbb{E}\mathbbm{1}_{ S_r} \mathbbm{1}_{\bigcup_{j \le J_K}\Lambda_{2r,i, j}}+\sum_{i\le p}\sum_{j \ge J_K}\mathbb{E}\mathbbm{1}_{ S_r} \mathbbm{1}_{\Lambda_{2r,i, j}}\\
    &\precsim_q p\cdot r^{s+1}+ \sum_{i\le p}\mathbb{E}\mathbbm{1}_{ S_r} \mathbbm{1}_{\bigcup_{j \le J_K}\Lambda_{2r,i, j}}+\sum_{J_K\le i\le p}\sum_{j > J_K}\mathbb{E}\mathbbm{1}_{ S_r} \mathbbm{1}_{\Lambda_{2r,i, j}}\\
    & \precsim_q p\cdot r^{1+s}+\sum_{i\le p}\mathbb{E}\mathbbm{1}_{ S_r} \mathbbm{1}_{\bigcup_{j \le J_K}\Lambda_{2r,i, j}}+\sum_{J_K\le i\le p}\sum_{j > J_K}\mathbb{E}\mathbbm{1}_{\pmb{\pi}_X S_r} \mathbbm{1}_{\pmb{\pi}_X\Lambda_{2r,i, j}},
    \end{align*}where the last two ``$\precsim_q$" are due to the inequality $i \ge j >J_k$, and the fact that $S_r$ is a section in the $s$-quasi-section $B_r(q) \times S^1$. For any $x \in  S_r \bigcap \Lambda_{2r,i, j}$ and for some $j \le J_K$ we have $\pmb{d}_{\partial Q}(x, f^{-i}x) \le 2r$, and $f^{-i}x \in B_{2r}(\pi_{\partial Q}x) \times S^1 \subseteq B_{3r}(q) \times S^1$. Therefore \[\pmb{\pi}_Xx\in \pmb{\pi}_X S_r,\quad (f^R)^{-j}\pmb{\pi}_Xx \in \pmb{\pi}_XB_{3r}(q)\times S^1 \text{ or }(f^R)^{-j+1}\pmb{\pi}_Xx \in \pmb{\pi}_XB_{3r}(q)\times S^1,\]
    where the second case (with ``$-j+1$") only happens if $ f^{-i}x\in X$ and $j\ge 2$.
    
    Since $S_r$ is a section, and $(f^R)_{*}\mu_X=\mu_X$, if $ S_r \bigcap \Lambda_{2r,i,j}\neq \emptyset$,  then $j \ge 1$, and  \begin{align*}
    \mathbb{E}\mathbbm{1}_{ S_r} \mathbbm{1}_{\Lambda_{2r,i, j}} &\precsim \mathbb{E}\mathbbm{1}_{\pmb{\pi}_X S_r} \mathbbm{1}_{(f^R)^{j} \pmb{\pi}_XB_{3r}(q)\times S^1}+\mathbb{E}\mathbbm{1}_{\pmb{\pi}_X S_r} \mathbbm{1}_{(f^R)^{j-1} \pmb{\pi}_XB_{3r}(q)\times S^1} \\
    &\precsim \mathbb{E}\mathbbm{1}_{(f^R)^{-j}\pmb{\pi}_X S_r} \mathbbm{1}_{ \pmb{\pi}_XB_{3r}(q)\times S^1}+\mathbb{E}\mathbbm{1}_{(f^R)^{-j+1}\pmb{\pi}_X S_r} \mathbbm{1}_{ \pmb{\pi}_XB_{3r}(q)\times S^1}\\
    &\precsim \mathbb{E}\mathbbm{1}_{(f^R)^{-j}\pmb{\pi}_XB_{r}(q)\times S^1 } \mathbbm{1}_{ \pmb{\pi}_XB_{3r}(q)\times S^1}+\mathbb{E}\mathbbm{1}_{(f^R)^{-j+1}\pmb{\pi}_XB_{r}(q)\times S^1 } \mathbbm{1}_{ \pmb{\pi}_XB_{3r}(q)\times S^1}\\
    &\precsim \mathbb{E}\mathbbm{1}_{(f^R)^{-j}\pmb{\pi}_XB_{3r}(q)\times S^1 } \mathbbm{1}_{ \pmb{\pi}_XB_{3r}(q)\times S^1}+\mathbb{E}\mathbbm{1}_{(f^R)^{-j+1}\pmb{\pi}_XB_{3r}(q)\times S^1 } \mathbbm{1}_{ \pmb{\pi}_XB_{3r}(q)\times S^1}\\
    &\precsim_{q,s} \mathbb{E}\mathbbm{1}_{(f^R)^{-j}\pmb{\pi}_X S_{3r}} \mathbbm{1}_{ \pmb{\pi}_XS_{3r}}+\mathbb{E}\mathbbm{1}_{(f^R)^{-j+1}\pmb{\pi}_X S_{3r}} \mathbbm{1}_{ \pmb{\pi}_XS_{3r}}+r^{1+s},
    \end{align*} where the terms with $(f^R)^{-j+1}$ appear only if $j\ge 2$.
    
    For any $x \in \pmb{\pi}_X S_{3r} \bigcap (f^R)^{-j}\pmb{\pi}_X S_{3r}$ there is an unique $k\ge 0$, such that $m=f^kx \in  S_{3r}$, $f^{i}(m) \in  S_{3r}$ (some $i\ge 1$), and the orbit $m, f(m), f^2(m), \cdots, f^{i}(m)$ visits $X$ $j$ or $j+ 1$ times. Therefore, $\pmb{\pi}_X S_{3r} \bigcap (f^R)^{-j}\pmb{\pi}_X S_{3r} \subseteq \pmb{\pi}_X S_{3r} \bigcap  \pmb{\pi}_X[\bar{\Lambda}_{3r,j}(q)\bigcup \bar{\Lambda}_{3r,j+1}(q)]$, and
    \begin{align*}
      \mathbb{E}\mathbbm{1}_{ S_r} \mathbbm{1}_{\bigcup_{1\le j \le J_K}\Lambda_{2r,i, j}}&\le \sum_{1\le j\le J_K}\mathbb{E}\mathbbm{1}_{ S_r} \mathbbm{1}_{\Lambda_{2r,i, j}}\\
      &\precsim_{q,s} \sum_{1\le j\le J_K} \mathbb{E}\mathbbm{1}_{(f^R)^{-j}\pmb{\pi}_X S_{3r}} \mathbbm{1}_{ \pmb{\pi}_XS_{3r}}+r^{1+s}\\
      &\precsim_{q,s} \sum_{1\le j\le J_K+1} \mathbb{E}\mathbbm{1}_{\pmb{\pi}_X\bar{\Lambda}_{3r,j}(q)} \mathbbm{1}_{ \pmb{\pi}_XS_{3r}}+r^{1+s}\\
      &\precsim_{q,s,K} \sup_{1\le j \le J_K+1} \mathbb{E}\mathbbm{1}_{\pmb{\pi}_X\bar{\Lambda}_{3r,j}(q)} \mathbbm{1}_{ \pmb{\pi}_XS_{3r}}+r^{1+s}.
    \end{align*}
    
    Now we can continue the estimate as $\mathbb{E} \left(\mathbbm{1}_{X'_0=1}  \mathbbm{1}_{\sum_{1\le i \le p-1}X'_i\ge 1}\right) $
    \begin{align*}
   &\precsim_q p\cdot r^{1+s}+\sum_{i\le p}\mathbb{E}\mathbbm{1}_{ S_r} \mathbbm{1}_{\bigcup_{1\le j \le J_K}\Lambda_{2r,i, j}}+\sum_{J_K\le i\le p}\sum_{j > J_K}\mathbb{E}\mathbbm{1}_{\pmb{\pi}_X S_r} \mathbbm{1}_{\pmb{\pi}_X\Lambda_{2r,i, j}}\\
        &\precsim_{q,K}p \cdot r^{1+s}+p\cdot \sup_{1\le j \le J_K+1} \mathbb{E}\mathbbm{1}_{\pmb{\pi}_X\bar{\Lambda}_{3r,j}(q)} \mathbbm{1}_{ \pmb{\pi}_XS_{3r}}+\sum_{J_K\le i\le p}\sum_{j > J_K}\mathbb{E}\mathbbm{1}_{\pmb{\pi}_X S_r} \mathbbm{1}_{\pmb{\pi}_X\Lambda_{2r,i, j}}.
    \end{align*} 
    
    Now a proof is concluded by taking $n \approx_T r^{-1}$.
    \end{proof}

Before estimating the measure of $\pmb{\pi}_X\bar{\Lambda}_{r,j}(q)$ ($j\le J_K+1$) we need one more lemma. Define \[\mathbb{S}_{J_K+1}:=\bigcup_j \bigcup_{0\le k < R|_{X_j}}f^k\big(\bigcup_{0 \le i\le J_K+1}(f^R)^{-i}\mathbb{S}\big). \]

Then any vertical manifold $\{q\}\times S^1$ is partitioned by $\mathbb{S}_{J_K+1}$ into countably many pieces $I_{q,i}$ so that $\{q\}\times S^1=\bigcup_{i}I_{q,i}$. Denote the piece containing $m\in \mathcal{M}$ by $I(m)$.  
\begin{lemma}\label{3}
$\mu_X(\pmb{\pi}_X\{m \in \mathcal{M}: \diam I(m)\le h\}) \precsim_{K} h^{g \alpha^{J_K+2}}$ for any $h>0$.
\end{lemma}
\begin{proof}
From the definition of $I(m)$ we know that $(f^R)^{J_K+1}$ is smooth on $\pmb{\pi}_X I(m)$, where $I(m)$ is contained in $f^kX_j$ for some $j$, and $k\in [0,R|_{X_j})$, $\pmb{\pi}_XI(m) \subseteq [\bigcup_{0 \le i\le J_K+1}(f^R)^{-i}\mathbb{S}]^c $. Besides, $\pmb{\pi}_XI(m)$ has the endpoints in $\bigcup_{0 \le i\le J_K+1}(f^R)^{-i}\mathbb{S}$. By Assumption \ref{assumption} there is a constant $C>0$, such that for any $m\in \{m \in \mathcal{M}: \diam I(m)\le h\}$ \[\diam \pmb{\pi}_X I(m) \le C [\diam I(m)]^{\alpha}\le C h^{\alpha}, \quad \pmb{\pi}_X m \in N_{Ch^{\alpha}} \Big(\bigcup_{0 \le i\le J_K+1}(f^R)^{-i}\mathbb{S}\Big).\]

Therefore, $\pmb{\pi}_X\{m \in \mathcal{M}: \diam I(m)\le h\}\subseteq N_{Ch^{\alpha}} \Big(\bigcup_{0 \le i\le J_K+1}(f^R)^{-i}\mathbb{S}\Big)$. By Lemma \ref{lentghofmfd}, 
\[\mu_X(\pmb{\pi}_X\{m \in \mathcal{M}: \diam I(m)\le h\}) \le \mu_X \Big\{N_{Ch^{\alpha}} \Big(\bigcup_{0 \le i\le J_K+1}(f^R)^{-i}\mathbb{S}\Big)\Big\} \precsim_{K} h^{g \alpha^{J_K+2}},\]which concludes a proof of lemma.
\end{proof}

Now we can estimate the measure of the set $\pmb{\pi}_X\bar{\Lambda}_{r,j}(q)$, $j \ge 1$, which is contained in $\pmb{\pi}_XB_r(q) \times S^1$.
\begin{lemma}\label{4}
For any $j \in [1,J_K+1]$, \begin{align*}
     \mu_{X}\big\{\pmb{\pi}_Xm : \diam I(m) \ge r^{\alpha^{J_K+2}/2}, m \in \bar{\Lambda}_{r,j}(q)\big\} \precsim_K  r^{\alpha^{J_K+2}/2},
\end{align*}\begin{align*}
     \mu_{X}\big(\pmb{\pi}_X\bar{\Lambda}_{r,j}(q)\big)&\le\mu_X\Big\{\pmb{\pi}_X\big(B_r(q)\times S^1\big) \bigcap \pmb{\pi}_X\big\{m \in \mathcal{M}: \diam I(m)\le r^{\alpha^{J_K+2}/2}\big\}\Big\}\\
    &\quad +\mu_X \Big\{\pmb{\pi}_X\big(B_r(q)\times S^1\big) \bigcap \pmb{\pi}_X\big\{m \in \mathcal{M}: \diam I(m)\ge r^{\alpha^{J_K+2}/2}, m \in \bar{\Lambda}_{r,j}(q)\big\}\Big\}.
\end{align*}
\end{lemma}
\begin{proof}If $m \in \bar{\Lambda}_{r,j}(q)$, then there is $i \ge 1$, such that the orbit $m, f(m), \cdots, f^{i}(m)$ visits $X$ exactly $j \le J_K+1$  times. Besides $f^i(m) \in S_r \subseteq B_r(q)\times S^1$, and $f^i$ is smooth on $I(m)$. Moreover, $j\ge 2$ if $m\in X$.

\textbf{Claim:} For every $m' \in I(m) \bigcap \bar{\Lambda}_{r,j}(q)$ the orbit $m', f(m'), \cdots, f^i(m')$ visits $X$ exactly $j$ times, and $f^{i}(m') \in S_r \subseteq B_r(q)\times S^1$.

It follows from the definition of $I(m)$, $(f^{R})^{J_K+1}$ that is smooth on $\pmb{\pi}_XI(m)$, and $f^i$ is smooth on $I(m)$. Then for each $k \in [0,i]$ the set $\pmb{\pi}_Xf^kI(m)$ is completely contained in some $X_{j_k}$. In particular, $\pmb{\pi}_Xf^im, \pmb{\pi}_Xf^im' \in X_{j_i}$, and each of the orbits $m, f(m), \cdots, f^i(m)$ and $m', f(m'), \cdots, f^i(m')$ visit $X$ exactly $j$ times. Moreover, $S_r$ is a section and $f^{i}m \in S_r$. Then $f^im'$ must belong to $S_r$, and the claim holds.

By (\ref{invariantcone}) and Assumption \ref{assumption}, the tangent vectors in $Df^i \mathcal{T} I(m)\subseteq C^u$ have uniformly bounded slopes. Besides, $f^i[I(m)\bigcap \bar{\Lambda}_{r,j}(q)] \subseteq S_r \subseteq B_r(q) \times S^1$. Then there is a constant $C>0$, which does not depend on $m,r,j,q$, and such that $\diam f^i[I(m)\bigcap \bar{\Lambda}_{r,j}(q)]\le Cr$. By Assumption \ref{assumption} \begin{align*}\diam I(m)\bigcap \bar{\Lambda}_{r,j}(q) \precsim \Big[\diam f^R\Big(\pmb{\pi}_X [I(m)\bigcap \bar{\Lambda}_{r,j}(q)]\Big)\Big]^{\alpha} &\precsim_K \big[\diam f^i\big(I(m)\bigcap \bar{\Lambda}_{r,j}(q)\big)\big]^{\alpha^{J_K+2}}\\
&\precsim_K r^{\alpha^{J_K+2}}.
\end{align*}

Let $h=r^{\alpha^{J_K+2}/2}$. Then
\begin{align*}
    \mu_{X}\big\{\pmb{\pi}_X m : \diam I(m) \ge h,  m &\in \bar{\Lambda}_{r,j}(q)\big\} \precsim \mu_{\mathcal{M}}\big\{m \in \mathcal{M}: \diam I(m) \ge h, m \in \bar{\Lambda}_{r,j}(q)\big\}\\
    &\precsim  \Leb_{\mathcal{M}} \big\{m \in \mathcal{M}: \diam I(m) \ge h, m \in \bar{\Lambda}_{r,j}(q)\big\}\\
    &=\int \Leb_{\{q'\}\times S^1}\big\{m\in \mathcal{M}: \diam I(m) \ge h, m \in \bar{\Lambda}_{r,j}(q)\big\} dq'\\
    &=\int \sum_{\diam I(m)\ge h, \pi_{\partial Q}m=q'} \Leb_{I(m)}\big\{\bar{\Lambda}_{r,j}(q)\big\}dq'\\
    &\precsim_K \int \pi h^{-1} r^{\alpha^{J_K+2}} dq'\precsim_K r^{\alpha^{J_K+2}/2},
\end{align*}\begin{align*}
    \mu_{X}\big(\pmb{\pi}_X\bar{\Lambda}_{r,j}(q)\big)&=\mu_{X}\big(\pmb{\pi}_X B_r(q) \times S^1 \bigcap \pmb{\pi}_X\bar{\Lambda}_{r,j}(q)\big)\\
    &=\mu_X \big\{\pmb{\pi}_Xm: m \in B_r(q)\times S^1, \diam I(m)\le h, m \in \bar{\Lambda}_{r,j}(q) \big\}\\
    & \quad +\mu_X \big\{\pmb{\pi}_Xm: m \in B_r(q)\times S^1, \diam I(m)\ge h, m \in \bar{\Lambda}_{r,j}(q) \big\}\\
    &\le \mu_X \big\{\pmb{\pi}_Xm: m \in B_r(q)\times S^1, \diam I(m)\le h \big\}\\
    & \quad +\mu_X \big\{\pmb{\pi}_Xm: m \in B_r(q)\times S^1, \diam I(m)\ge h, m \in \bar{\Lambda}_{r,j}(q) \big\}\\
    &\le \mu_X\Big\{\pmb{\pi}_X(B_r(q)\times S^1) \bigcap \pmb{\pi}_X\big(m \in \mathcal{M}: \diam I(m)\le r^{\alpha^{J_K+2}/2}\big)\Big\}\\
    &\quad +\mu_X \Big\{\pmb{\pi}_X\big(B_r(q)\times S^1\big) \bigcap \pmb{\pi}_X\big[m \in \mathcal{M}: \diam I(m)\ge r^{\alpha^{J_K+2}/2}, m \in \bar{\Lambda}_{r,j}(q)\big]\Big\}
\end{align*}which concludes a proof of the lemma.
\end{proof}

Now we are ready to estimate convergence rates for (\ref{shortreturn}).
\begin{lemma}\label{rateforshortreturn}
Let $p\ll n^{\min\{s/4,g/(8+8g), g\alpha^{2J_K+4}/16, \alpha^{J_K+2}/16\}}$. Then 
\[(\ref{shortreturn}) \precsim_{q,K,T} r^{\min\{s/2,g/(4+4g), g\alpha^{2J_K+4}/8, \alpha^{J_K+2}/8\}} \text{ for any }r>0.\]
\end{lemma}
\begin{proof}
First we will use a corollary of Schur's test. 

\textbf{Claim:} Let there is a constant $C\ge 1$, such that two families of sets $(C_r)_{r>0} and (A_r(q))_{q \in \partial Q, r>0}$ in $X$ satisfy  the relation $\mu_X(A_r(q))=C^{\pm 1} r$, and $\mu_X(C_r)\le C r^u$ for some $u>0$. Then there exists a decreasing sequence $r:=r_k=k^{-4/u}$, such that $\mu_X(A_r(q) \bigcap C_r) \precsim_q r^{1+u/2}$ for $\Leb_{\partial Q}$-a.e. $q\in \partial Q$.

Indeed, by Schur's test $\int_{\partial Q} \mu_X(A_r(q)\bigcap C_r) dq \precsim r \cdot r^u $. Then by Borel-Cantelli lemma, for $\Leb_{\partial Q}$-a.e. $q \in \partial Q$,  $\mu_X(A_r(q)\bigcap C_r)/r^{1+u/2}\precsim_q 1$ if $r:=r_k=k^{-4/u}$. So, the claim holds.

Now we choose $u:=\min\{g/(1+g), g\alpha^{2J_K+4}/2, \alpha^{J_K+2}/2\}$, and apply this claim. Set $r:=r_k=k^{-4/u}$. Observe that $B_r(q) \times S^1$ is a $s$-quasi-section $B_r(q) \times S^1$. Therefore $\mu_X(\pmb{\pi}_X S_r) \approx \mu_X(\pmb{\pi}_XB_{r}(q) \times S^1) \approx r$.

By Lemma \ref{1} we have $\mu_X(\pmb{\pi}_X\Lambda_{2r,i, j})\precsim_K r^{g/(g+1)}\precsim_K r^u$. Then for $\Leb_{\partial Q}$-a.s. $q\in \partial Q$,
\[n \cdot \sum_{J_K\le i\le p}\sum_{j > J_K}\mathbb{E}\mathbbm{1}_{\pmb{\pi}_X S_r} \mathbbm{1}_{\pmb{\pi}_X\Lambda_{2r,i, j}}\precsim_{K, q, T}n (p-J_K)^2 r^{1+u/2}\precsim_{K,q,T}p^2r^{u/2}. \]

Now by Lemma \ref{3} we have $\mu_X(\pmb{\pi}_X\{m \in \mathcal{M}: \diam I(m)\le (3r)^{\alpha^{J_K+2}/2}\}) \precsim_{K} r^{g \alpha^{2J_K+4}/2}\precsim_K r^u$. Then for $\Leb_{\partial Q}$-a.e. $q\in \partial Q$,
\begin{align*}
    np\mu_X\{\pmb{\pi}_X(B_{3r}(q)\times S^1) \bigcap \pmb{\pi}_X\{m \in \mathcal{M}: \diam I(m)\le (3r)^{\alpha^{J_K+2}/2}\}\}&\precsim_{K,q,T}npr^{1+u/2} \precsim_{K,q,T}pr^{u/2}.
\end{align*}

It follows from the Lemma \ref{4} that $\mu_{X}\{\pmb{\pi}_Xm : \diam I(m) \ge (3r)^{\alpha^{J_K+2}/2}, m \in \bar{\Lambda}_{3r,j}(q)\} \precsim_K  r^{\alpha^{J_K+2}/2} \precsim_{K} r^{u}$. Then for $\Leb_{\partial Q}$-a.e. $q\in \partial Q$ 
\begin{align*}
np\mu_X \{\pmb{\pi}_X(B_{3r}(q)\times S^1) \bigcap \pmb{\pi}_X\{m \in \mathcal{M}: \diam I(m)&\ge (3r)^{\alpha^{J_K+2}/2}, m \in \bar{\Lambda}_{3r,j}(q)\}\}\\
&\precsim_{K,q,T}np r^{1+u/2}\precsim_{K,q,T} p r^{u/2}.
\end{align*}

Therefore, by Lemmas \ref{2} and \ref{4}, we have for $\Leb_{\partial Q}$-a.e. $q\in \partial Q$ that 
\begin{align*}
  &n\sum_{i\le p}\mathbb{E}\mathbbm{1}_{B_r(q)\times S^1} \mathbbm{1}_{B_r(q)\times S^1} \circ f^i\\
    &\precsim_{q,K,T} p \cdot r^{s}+n\cdot p\cdot \sup_{1\le j \le J_K+1} \mathbb{E}\mathbbm{1}_{\pmb{\pi}_X\bar{\Lambda}_{3r,j}(q)} \mathbbm{1}_{ \pmb{\pi}_XS_{3r}}+n \cdot \sum_{J_K\le i\le p}\sum_{j > J_K}\mathbb{E}\mathbbm{1}_{\pmb{\pi}_X S_r} \mathbbm{1}_{\pmb{\pi}_X\Lambda_{2r,i, j}}\\
    & \precsim_{q,K,T} p \cdot r^{s}+n\cdot p\cdot \sup_{1\le j \le J_K+1} \mathbb{E}\mathbbm{1}_{\pmb{\pi}_X\bar{\Lambda}_{3r,j}(q)}\mathbbm{1}_{\pmb{\pi}_XS_{3r}} +p^2r^{u/2} \precsim_{q,K} pr^s+pr^{u/2}+p^2r^{u/2}.
\end{align*}

Since $p\ll n^{\min\{s,u/2\}/4}$, then for $r=r_k=k^{-4/u}$ we get
\begin{align*}
    n\sum_{i\le p}\mathbb{E}\mathbbm{1}_{B_r(q)\times S^1} \mathbbm{1}_{B_r(q)\times S^1} \circ f^i \precsim_{q,K,T} r^{\min\{s,u/2\}/2}.
\end{align*} 

For any $r\in [r_{k+1}, r_k]$ there exists $k \in \mathbb{N}$, such that \begin{align*}
    (\ref{shortreturn})& \le n\sum_{i\le p}\mathbb{E}\mathbbm{1}_{B_r(q)\times S^1} \mathbbm{1}_{B_r(q)\times S^1} \circ f^i \\
    &\precsim_{T} r_{k+1}^{-1} \sum_{i \le p} \mathbb{E}\mathbbm{1}_{B_{r_k}(q)\times S^1} \mathbbm{1}_{B_{r_k}(q)\times S^1} \circ f^i \\
    & \precsim_{T} r_{k+1}^{-1}r_k r_{k}^{-1} \sum_{i \le p} \mathbb{E}\mathbbm{1}_{B_{r_k}(q)\times S^1} \mathbbm{1}_{B_{r_k}(q)\times S^1} \circ f^i\precsim_{q,K,T} r^{\min\{s,u/2\}/2},
\end{align*} where the last ``$\precsim_{q,K,T}$" is due to $r^{-1}_{k+1}r_k \precsim 1$.
\end{proof}

\subsubsection{Conclusions of a proof of Theorem \ref{fplthm}}
The following lemma provides an estimate of convergence rate for Theorem \ref{fplthm}.
\begin{lemma}\label{rateconclusion}For $\Leb_{\partial Q}$-a.e. $q\in \partial Q$ holds the relation $d_{TV}\left(\mathcal{N}^{r,q,T},\mathcal{P}\right)\precsim_{q,T} r^a$, where\begin{align*}
    a:=\min\Big\{&\frac{s\xi^2}{48+48\xi},\frac{g\xi^2}{(6+6\xi)(16+16g)}, \frac{\xi^2g\alpha^{2J_K+4}}{192+192\xi},  \frac{\xi^2\alpha^{J_K+2}}{192+192\xi}, \frac{s}{2},\frac{g}{4+4g}, \frac{g\alpha^{2J_K+4}}{8}, \frac{\alpha^{J_K+2}}{8}\Big\},
\end{align*} where $J_K$ is a positive integer, which depends only upon $K$, which was defined in the proof of Lemma \ref{1}.
\end{lemma}
\begin{proof}
It follows from Lemma \ref{truncatehole}, Lemma \ref{rateforshortreturn} and Lemma \ref{rateforpoissoncorrelation}, that $p\in (2m(N),n]$, $m(N)> (\mathbb{E}R+\epsilon)N, N>7, \epsilon< \mathbb{E}R/18$, $p \ll  n^{\min\{s/4,g/(8+8g), g\alpha^{2J_K+4}/16, \alpha^{J_K+2}/16\}}$, $N=\lfloor j^{(4+4\xi)/[\xi^2(1-\delta)]} \rfloor$ ($j\in \mathbb{N}$), and $\delta \in (1/2,1)$. Choose now $\delta=2/3$, $\epsilon=\mathbb{E}R/20$,\[j=\Big\lfloor \Big[n^{\min\{s/8,g/(16+16g), g\alpha^{2J_K+4}/32, \alpha^{J_K+2}/32\}}/(84\mathbb{E}R)\Big]^{\xi^2/(12+12\xi)}\Big\rfloor,\quad N=\lfloor j^{(12+12\xi)/\xi^2} \rfloor \ll n,\]\[m(N)=\lceil 21N\mathbb{E}R \rceil, and \quad p=4m(N) \approx n^{\min\{s/8,g/(16+16g), g\alpha^{2J_K+4}/32, \alpha^{J_K+2}/32\}}.\]

Then $p \ll  n^{\min\{s/4,g/(8+8g), g\alpha^{2J_K+4}/16, \alpha^{J_K+2}/16\}} \ll n$.

Now, by (\ref{totalconvergencerate}), Lemma \ref{subtract}, Lemma \ref{rateforshortreturn} and Lemma \ref{rateforpoissoncorrelation}, we have
\begin{align*}
     &d_{TV}\left(\mathcal{N}^{r,q,T},\mathcal{P}\right)\precsim_{q,T} \mu_{\mathcal{M}}\big(B_r(q)\times S^1\big) +N^{-\xi^2(1-\delta)/(2+2\xi)}+ (\ref{shortreturn})+ (\ref{poissoncorrelation})+p\cdot \mu_{\Delta}\big(X'_0=1\big)\\
     &\precsim_{q,T,K, \delta} r+N^{-\xi^2/(6+6\xi)}+r^{\min\{s/2,g/(4+4g), g\alpha^{2J_K+4}/8, \alpha^{J_K+2}/8\}}\\
     &\quad +(p-2m(N))^{-\xi}+n^2\beta^{\frac{m(N)\alpha}{2\mathbb{E}R-4\epsilon}}+pr\\
     & \precsim_{q,T,K} r+r^{\min\{s/8,g/(16+16g), g\alpha^{2J_K+4}/32, \alpha^{J_K+2}/32\}\xi^2/(6+6\xi)}\\
     &\quad +r^{\min\{s/2,g/(4+4g), g\alpha^{2J_K+4}/8, \alpha^{J_K+2}/8\}}+r^{\min\{s/8,g/(16+16g), g\alpha^{2J_K+4}/32, \alpha^{J_K+2}/32\}\xi}\\
     &\quad +r^{1-\min\{s/8,g/(16+16g), g\alpha^{2J_K+4}/32, \alpha^{J_K+2}/32\}}\\
     & \precsim_{q,T,K} r^{\min\{s/8,g/(16+16g), g\alpha^{2J_K+4}/32, \alpha^{J_K+2}/32\}\xi^2/(6+6\xi)} +r^{\min\{s/2,g/(4+4g), g\alpha^{2J_K+4}/8, \alpha^{J_K+2}/8\}}\\
     &\quad +r^{1-\min\{s/8,g/(16+16g), g\alpha^{2J_K+4}/32, \alpha^{J_K+2}/32\}}\\
      & \precsim_{q,T,K} r^{\min\{s/8,g/(16+16g), g\alpha^{2J_K+4}/32, \alpha^{J_K+2}/32\}\xi^2/(6+6\xi)} +r^{\min\{s/2,g/(4+4g), g\alpha^{2J_K+4}/8, \alpha^{J_K+2}/8\}},
\end{align*}where the last ``$\precsim_{q,T,K}$" holds because $\min\{s/2,g/(4+4g), g\alpha^{2J_K+4}/8, \alpha^{J_K+2}/8\}\le 4/5$.\end{proof}

\subsection{Applications}\label{sectionapp}
In what follows we will consider four classes of two-dimensional slowly mixing billiards, which were studied in \cite{CZcmp,CZnon, Subbb}. The rates of mixing (decay of correlations) in these billiards are either of order $O(n^{-1})$, or $O(n^{-2})$.

\begin{figure}[!htb]
   \begin{minipage}{0.6\textwidth}
     \includegraphics[width=.6\linewidth]{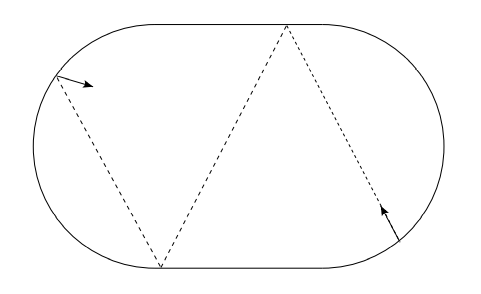}
     \caption{Stadium billiard}
     \label{F3}
   \end{minipage}\hfill
   \begin{minipage}{0.6\textwidth}
     \includegraphics[width=.6\linewidth]{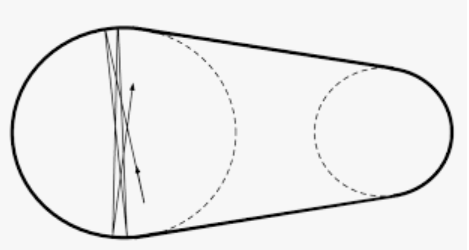}
     \caption{Squash billiard}\label{F4}
   \end{minipage}
\end{figure}

\begin{figure}[!htb]
   \begin{minipage}{0.6\textwidth}
\includegraphics[width=.6\linewidth]{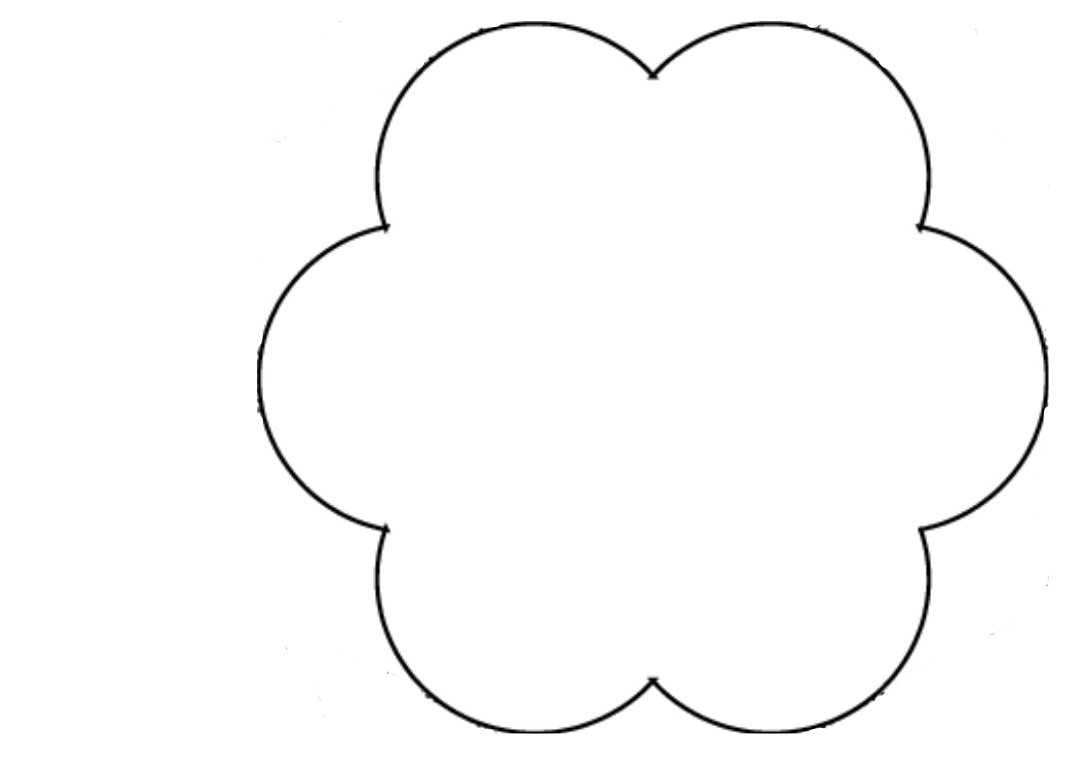}
     \caption{Circular lower billiard}
     \label{F5}
   \end{minipage}\hfill
   \begin{minipage}{0.48\textwidth}
     \includegraphics[width=.6\linewidth]{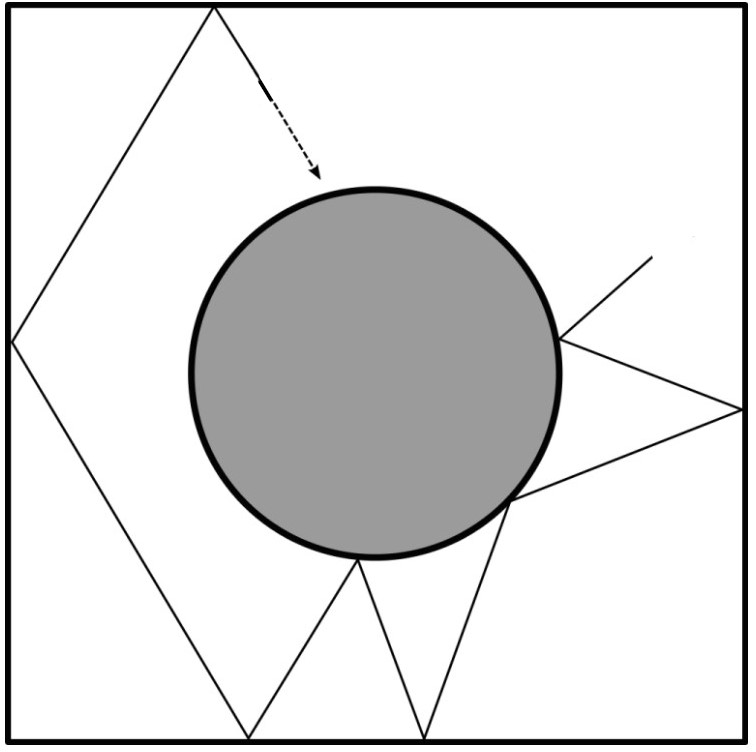}
     \caption{Semi-dispersing billiard}\label{F6}
   \end{minipage}
\end{figure}

A billiard table $Q$ of a squash billiard is a convex domain bounded by two circular arcs and two straight (flat) segments tangent to the arcs at their common endpoints. A squash billiard is called a stadium if the flat sides are parallel, see Figure \ref{F3}. Initially called squash billiards were  
later sometimes called ``skewed" stadia, drive-belt billiards, etc, see Figure \ref{F4}. Note that the squashes contain a boundary arc, which is larger than a half circle.

A billiard table $Q$ of a more general focusing billiard, e.g. a circular flower billiard (see the Figure \ref{F5}), is such, that each smooth component $\Gamma_i$ of the boundary $\partial Q$ is an arc of some circle satisfying SFC-conditions. 

The figure \ref{F6} represents a special class of semi-dispersing billiards. More precisely, let $R_0 \subseteq \mathbb{R}^2$ be a rectangle, and scatterers $B_1, \cdots, B_r \subseteq \interior R_0$ are open strictly convex sub-domains with smooth, (at least $C^3)$, or piece-wise smooth boundaries whose curvature is bounded away from zero,
and such that $B_i \bigcap B_j=\emptyset$ for $i\neq j$. The boundary of a billiard table $Q = R_0\setminus \bigcup_iB_i$ is partially dispersing (convex inwards) $\bigcup_{i}\partial B_i$ and partially neutral (flat) $\partial R_0$.

We verify now the Assumption \ref{assumption} for stadium billiards (Figure \ref{F3}). Let  $X\subseteq \mathcal{M}:=\partial Q \times S^1$ be the region where the first collisions (in a series of consecutive collisions with one and the same circular arc) with circular arcs occur. Denote by $R$ the first return time to $X$ for the billiard map $f$. Denote by $l$ the length of each straight segment. Without loss of generality, we may assume that the radius of each circular arc equals $1$. 

Let $\mathbb{S}_0$ be the union of points where the map $f^R$ is not continuous or not differentiable. Define $\mathbb{S}$ as the union $(f^R)^{-1}\mathbb{S}_0\bigcup \mathbb{S}_0 \bigcup f^R\mathbb{S}_0$. Then some curves in $\mathbb{S}$ are increasing, some are decreasing, and some are vertical. It has been verified in \cite{Subbb} that $\bigcup_{i\ge 1}\partial X_i \subseteq \mathbb{S}$ as well as H\"older continuity along (un)stable curves/manifolds. We have
\begin{align*}
        &d_{f^j{{\gamma}^k}}(f^jx, f^jy)\le C d_{{{\gamma}^k}}(x,y)^\alpha \text{ for all }j< R(x),\\
        &d_{f^j\pmb{{\gamma}^u}}(f^jx, f^jy)\le C d_{\pmb{{\gamma}^u}}(x,y)^\alpha \text{ when }j= R(x).
    \end{align*}
    
    It was proved in \cite{Subbb}  that for any $q\in \partial Q$ and  $B_r(q)\subseteq \partial Q$,  $B_r(q)\times S^1 $ is a  $1$-quasi-section. Besides, it was shown there that the relation $\mu_X(N_{\epsilon}(\mathbb{S}))\precsim  \epsilon^g$ in Assumption \ref{assumption} is in fact a typical property for hyperbolic systems. Therefore we just need to verify the rest of conditions in Assumption \ref{assumption}. 
    
    Throughout this subsection we set $x_m=(q_m,\phi_m)=f^m(x_0)$ for all $m \in \mathbb{Z}$, where $x_0 \in \mathcal{M}$. Denote by $B^{+}(x_m)$ (resp. $B^-(x_m)$) a curvature of a wave front right after (resp. before) collision with the boundary at the point $x_m\in \mathcal{M}$. We list now a few more basic properties and formulas  (see e.g. \cite{CMbook}) for two-dimensional billiards. \begin{gather}
    v_m:=d\phi_m/dq_m=B^-(x_m)\cos \phi_m +\mathcal{K}(q_m)=B^+(x_m)\cos \phi_m -\mathcal{K}(q_m),\nonumber\\
    1/B^{-}(x_{m+1})=\tau_{m}+1/B^{+}(x_{m}),\nonumber\\
    Df(x_m)=\begin{bmatrix}
    \frac{\tau_m\mathcal{K}(q_m)+\cos \phi_m }{-\cos \phi_{m+1}}   & \frac{\tau_m}{-\cos \phi_{m+1}} \\
    \frac{\tau_m\mathcal{K}(q_m)\mathcal{K}(q_{m+1})+\mathcal{K}(q_m)\cos\phi_{m+1} + \mathcal{K}(q_{m+1})\cos \phi_m}{-\cos \phi_{m+1}} \quad & \frac{\tau_m \mathcal{K}(q_{m+1})+\cos \phi_{m+1}}{-\cos \phi_{m+1}}\label{derivativeofbilliardmap}
\end{bmatrix},\\
\frac{||dx_{m+1}||_p}{||dx_m||_p}=\frac{||Df(dx_m)||_p}{||dx_m||_p}=|1+\tau_mB^{+}(x_m)|\label{pesudoderivative},\\
||dx_m||=\frac{||dx_m||_p}{\cos \phi_m} \sqrt{1+(\frac{d\phi_m}{dq_m})^2}\label{rnderivative},
\end{gather}where $\tau_m$ is the length of the free path from $x_{m}$ to $x_{m+1}$, $||dx_m||_p:=\cos \phi_m |dq_m|$, $||dx_m||:=\sqrt{(d\phi_m)^2+(dq_m)^2}$.

\begin{enumerate}
    \item We will verify now that slopes of unstable vectors are uniformly bounded. Let $x_0 \in X$, $dx_0 \in C^u_{x_0}$. Suppose that the points $x_0, x_1, \cdots, x_k$ $(k\ge 0)$ belong to one and the same circular arc, and $x_{k+1}$ does not. Then for any $i \in [0,k]$, $dx_i \in C^{u}(x_i) $, $B^+(x_0)\in [-\frac{2}{\cos \phi_0},-\frac{1}{\cos \phi_0}]$. Then  \begin{gather*}
        B^+(x_1)=\frac{-2}{\cos \phi_1}+\frac{1}{\tau_{0}+\frac{1}{B^+(x_{0})}} \text{ and } B^+(x_{1})\in [-\frac{4}{3\cos \phi_{1}}, -\frac{1}{\cos \phi_{1}}]. \end{gather*}
  Inductively, for any $i\in [0,k]$
  \begin{gather}\label{5}B^{+}(x_i)\in \big[-\frac{2i+2}{(2i+1)\cos \phi_i},-\frac{1}{\cos \phi_{i}}\big],\end{gather} and the slope of $dx_i$ is $\frac{d\phi_i}{dq_i}=\cos(\phi_i)B^{+}(x_i)+1 \in [-1,0]$. Suppose that the points $x_{k+1}, \cdots, x_{k+n}$, $(n\ge 1)$ are on flat sides, and $x_{k+n+1}$ belongs to another circular arc. Then for any $i \in [1,n]$, \begin{gather}\label{6}
        B^+(x_{k+1})=\frac{0}{\cos \phi_{k+1}}+\frac{1}{\tau_{k}+\frac{1}{B^+(x_{k})}}, \cdots \cdots, B^+(x_{k+i})=\frac{1}{\sum_{0\le j \le i-1}\tau_{k+j}+\frac{1}{B^+(x_{k})}}.
    \end{gather} 
    
    Hence, the slope of $dx_{k+i}$ is $\frac{d\phi_{k+i}}{dq_{k+i}}=\cos(\phi_{k+i})B^{+}(x_{k+i}) \in [0,1]$ due to SFC-conditions. Therefore, the slopes of the vectors in $Df^i(x_0)C^u_{x_0}$ $(0\le i< R(x_0))$ are uniformly bounded.
    
     \item Next we verify uniform transversality of (un)stable manifolds and singularities.  It follows from \cite{CMbook} that the slopes of vectors in $\mathcal{T}\gamma^u$ are of order $-1$, and the slopes of vectors in $\mathcal{T}\gamma^s$ are positive and uniformly bounded.  Some of the curves in the set $\bigcup_{i\ge 0 }(f^R)^{-i}\mathbb{S}=\bigcup_{i \ge -1}(f^R)^{-i}\mathbb{S}_0$ are vertical or decreasing. Decreasing curves appear because of $f^R(\mathbb{S}_0)$, and have slops which are $\approx -1$. They are uniformly transversal to the vectors in $\mathcal{T}\gamma^s$. Some curves are increasing with uniformly bounded slopes  due to  existence of the set $\bigcup_{i\ge 0}(f^R)^{-i}\mathbb{S}_0$. These curves are uniformly transversal to vectors of the type $\mathcal{T}\gamma^u$.
    \item Finally we verify that $\diam \gamma^u \le C \diam f^i \gamma^u$ and $\diam \pmb{\gamma^u} \le C (\diam f^i \pmb{\gamma^u})^{\alpha}$. Let $x_0 \in X$, $dx_0 \in C^u_{x_0}$. Suppose that the points $x_0, x_1, \cdots, x_k$ $(k\ge 0)$ belong to the same circular arc, and $x_{k+1}$ does not. From (\ref{5}) we have for any $i \in [0,k]$ that  $B^{+}(x_i)\in \big[-\frac{2i+2}{(2i+1)\cos \phi_i},-\frac{1}{\cos \phi_{i}}\big]$ and $|1+\tau_i B^+(x_i)|\ge 1$. Then \begin{gather*}
        \frac{||dx_i||_p}{||dx_0||_p}=\prod_{j\le i-1}|1+\tau_j B^+(x_j)| \ge 1, \quad \frac{||dx_i||}{||dx_0||}=\frac{||dx_i||_p}{||dx_0||_p} \frac{\cos \phi_0}{\cos \phi_i} \frac{\sqrt{1+(\frac{d\phi_{i}}{dq_{i}})^2}}{\sqrt{1+(\frac{d\phi_{0}}{dq_{0}})^2}} \succsim  1.
    \end{gather*}
    
    Let now the points $x_{k+1}, \cdots, x_{k+n}$ $(n>1)$ are on flat sides, and $x_{k+n+1}$ belongs to another circular arc. From (\ref{6}) we have for any $i \in [1,n]$ that $B^+(x_{k+i})=\big\{\sum_{0\le j \le i-1}\tau_{k+j}+\frac{1}{B^+(x_{k})}\big\}^{-1}$. Then \begin{gather*}
      |1+\tau_{k+i}B^+(x_{k+i})|=\frac{\sum_{0\le j \le i}\tau_{k+j}+\frac{1}{B^+(x_{k})}}{\sum_{0\le j \le i-1}\tau_{k+j}+\frac{1}{B^+(x_{k})}}.
    \end{gather*}
    
    By using SFC-conditions we obtain for any $i\in [k+1, n+k+1]$ that  \begin{gather*}
        \frac{||dx_i||_p}{||dx_0||_p}=\prod_{j\le i-1}|1+\tau_j B^+(x_j)|\ge  \big|\frac{\sum_{0\le j \le i-1-k}\tau_{k+j}+\frac{1}{B^+(x_{k})}}{\frac{1}{B^+(x_k)}}\big| \succsim 1,\\
        \frac{||dx_i||}{||dx_0||}=\frac{||dx_i||_p}{||dx_0||_p} \frac{\cos \phi_0}{\cos \phi_i} \frac{\sqrt{1+(\frac{d\phi_{i}}{dq_{i}})^2}}{\sqrt{1+(\frac{d\phi_{0}}{dq_{0}})^2}}\succsim 1.
    \end{gather*}
    
    Suppose that $x_{k+1}$ is on a flat side, and $x_{k+2}$ is on another circular arc. In this case $\cos \phi_k \ge \cos \phi_{k+1}$. The SFC-conditions imply that \begin{gather*}
    \frac{||dx_{k+1}||_p}{||dx_{k}||_p}=|1+\tau_{k}B^+(x_{k})|\ge \frac{\tau_k }{\cos \phi_k}-1 \ge 1,\\
        \frac{||dx_{k+2}||_p}{||dx_{k}||_p}=|1+(\tau_k+\tau_{k+1})B^+(x_{k})|\ge \frac{\tau_k+\tau_{k+1}-\cos \phi_k}{\cos \phi_k}\succsim_l \frac{1}{ \cos \phi_k},
    \end{gather*}
    
     Therefore,\begin{align*}
        \frac{||dx_{k+1}||}{||dx_{k}||}\approx \frac{||dx_{k+1}||_p}{||dx_{k}||_p} \frac{\cos \phi_{k}}{\cos \phi_{k+1}}\succsim 1,\quad
      \frac{||dx_{k+2}||}{||dx_{k}||}\approx \frac{||dx_{k+2}||_p}{||dx_{k}||_p} \frac{\cos \phi_{k}}{\cos \phi_{k+2}}\succsim \frac{1}{\cos \phi_{k+2}}\succsim 1,
  \end{align*}i.e.,
  \begin{gather*}
      \diam f^{k}\pmb{\gamma^u}(x_0) \precsim \diam f^{k+2}\pmb{\gamma^u}(x_0), \quad
      \diam f^{k}\pmb{\gamma^u}(x_0) \precsim \diam f^{k+1}\pmb{\gamma^u}(x_0).
  \end{gather*}

By summarizing all cases we get \begin{gather*}
d_{\gamma^u}(x,y) \precsim d_{f^i \gamma^u}(f^ix, f^i y) \text{ for any } i < R(x) \text{ and } x, y \in \gamma^u \subseteq X,\\
        d_{\pmb{{\gamma}^u}}(x,y)\precsim d_{f^j\pmb{{\gamma}^u}}(f^jx, f^jy) \text{ for all }j\le R(x) \text{ and any }x,y \in \pmb{\gamma^u} \subseteq \mathbb{S}^c \bigcap X.
    \end{gather*}
     
    \item Next we verify that $K=2$ in the Assumption \ref{assumption}.
    Suppose first that the points $x_0, x_1, \cdots, x_k$ $(k\ge 0)$ belong to one and the same circular arc, while $x_{k+1}, \cdots, x_{k+n}$, $(n>1)$ are on flat sides, and the point $x_{k+n+1}$ belongs to another circular arc. Then \begin{gather*}
        \frac{||dx_{k+n}||}{||dx_i||}=\frac{||dx_{k+n}||_p}{||dx_i||_p} \frac{\cos \phi_i}{\cos \phi_{k+n}} \frac{\sqrt{1+(\frac{d\phi_{k+n}}{dq_{k+n}})^2}}{\sqrt{1+(\frac{d\phi_{i}}{dq_{i}})^2}}\succsim 1 \text{ for any } i\in [k+1,k+n],\\
         \frac{||dx_{k+n}||}{||dx_i||}=\frac{||dx_{k+n}||}{||dx_{k+1}||}\frac{||dx_{k+1}||}{||dx_i||}\succsim \frac{||dx_{k+1}||_p}{||dx_i||_p} \frac{\cos \phi_i}{\cos \phi_{k+1}} \frac{\sqrt{1+(\frac{d\phi_{k+1}}{dq_{k+1}})^2}}{\sqrt{1+(\frac{d\phi_{i}}{dq_{i}})^2}}\succsim 1 \text{ for any } i\in [0,k].
    \end{gather*}
    
    Hence, $\diam f^i\gamma^u(x_0) \precsim \diam f^R \gamma^u (x_0) \precsim \diam (f^R)^2 \gamma^u (x_0)$ for any $i < R(x_0)$.
    
    Now we suppose that the points $x_0, x_1, \cdots, x_k$ $(k\ge 0)$ belong to the same circular arc, the point $x_{k+1}$ is on one of the flat sides, and $x_{k+2}$ belongs to another circular arc. It follows from \cite{CMbook} that $\frac{||dx_{k+2}||}{||dx_k||}\succsim 1$ on $f^k \gamma^u(x_0)$, since $\gamma^u(x_0) \subseteq (\bigcup_{0\le j\le 1}(f^R)^{-j}\mathbb{S})^c$. Hence, we can assume that $f^{k+2}\gamma^u(x_0)$ is in a $k'$-sliding cell. Thus, for any $i\in [0,k]$ we get \begin{gather*}
        \frac{||dx_{k+2}||}{||dx_i||}=\frac{||dx_{k+2}||}{||dx_{k}||}\frac{||dx_{k}||}{||dx_i||}\succsim \frac{||dx_{k}||_p}{||dx_i||_p} \frac{\cos \phi_i}{\cos \phi_{k}} \frac{\sqrt{1+(\frac{d\phi_{k}}{dq_{k}})^2}}{\sqrt{1+(\frac{d\phi_{i}}{dq_{i}})^2}}\succsim 1, \quad  \frac{||dx_{k+2}||}{||dx_{k+1}||} \succsim 1/k'.
    \end{gather*}
    
    From the relations (8.19) and (8.23) in \cite{CMbook} it follows that $\frac{||(Df^R) dx_{k+2}||}{||dx_{k+2}||} \succsim k'$. Hence, $\frac{||(Df^R) dx_{k+2}||}{||dx_{k+1}||} \succsim 1$. Therefore, for any $i\in [0,k]$ \begin{gather*}
        \diam (f^R)^2 \gamma^u(x_0) \succsim \diam (f^R)\gamma^u(x_0) \succsim \diam f^{i}\gamma^u(x_0),\\
        \diam (f^R)^2 \gamma^u(x_0) \succsim \diam f^{k+1}\gamma^u(x_0).
    \end{gather*} 
   
    By summarizing all cases we obtain that  $\diam f^i \gamma^u(x_0) \precsim \diam (f^R)^2 \gamma^u(x_0)$ for any $i< R(x_0)$.
   
    \item We verify now that the conditions hold on any vertical curve $\{q\}\times S^1$. The set $\pmb{\mathbb{S}}$ partitions any vertical curve $\{q\} \times S^1$ into countably many connected pieces. Denote by $I(x_0)$ the connected (sub)piece containing a point $x_0$. Let $dx_0=(0, d\phi_0) \in \mathcal{T}I(x_0)$. By (\ref{invariantcone}) we have $(Df)dx_0 \in C^u_{x_1}$ and $(Df^{-1}) dx_0 \in C^s_{x_{-1}}$. Therefore, $f[I(x_0)]$ is an unstable curve, and $f^{-1}[I(x_0)]$ is a stable curve. We will consider now different cases for $f[I(x_0)]$. 
    
     Suppose that $x_0$ is on a flat side or on a circular arc, and $x_1,\cdots,x_{n-1}$ are on one of the flat sides and $x_n$ is on another circular arc ($n\ge 1$). By (\ref{derivativeofbilliardmap}) \begin{gather*}
    dq_n=-\frac{\tau_0+\tau_1+\cdots+ \tau_{n-1}}{\cos \phi_n}d\phi_0, \quad d\phi_n=\big(\frac{\tau_0+\tau_1+\cdots+ \tau_{n-1}}{\cos \phi_n}-1\big) d\phi_0.
    \end{gather*} Then by the SFC-conditions 
    \[||dx_n|| \succsim ||d\phi_0|| \big(\frac{\tau_0+\tau_1+\cdots+ \tau_{n-1}-\cos \phi_n}{\cos \phi_n}\big) \succsim ||d\phi_0||,\] i.e., $\diam I(x_0) \precsim \diam f^n[I(x_0)]$.
    
    Let now the points $x_0,x_1,\cdots,x_{k}$ ($k\ge 1$) are on the same circular arc, $x_{k+1}, \cdots, x_{k+n}$ are on a flat side, and the point $x_{k+n+1}$ is on another circular arc. we have from (\ref{derivativeofbilliardmap}) that $dq_1=-2d\phi_0$, $d\phi_1=d\phi_0$. Then, $\diam f[I(x_0)]\approx \diam I(x_0)$, and $f[I(x_0)]$ is an unstable curve.  From (\ref{5})  it follows for any $i \in [1,k]$ that  $B^{+}(x_i)\in \big[-\frac{2(i-1)+2}{(2(i-1)+1)\cos \phi_i},-\frac{1}{\cos \phi_{i}}\big]$ and $|1+\tau_i B^+(x_i)|\ge 1$. Then \begin{gather*}
        \frac{||dx_k||_p}{||dx_1||_p}=\prod_{1 \le j\le k-1}|1+\tau_j B^+(x_j)| \ge 1, \quad \frac{||dx_k||}{||dx_1||}=\frac{||dx_k||_p}{||dx_1||_p} \frac{\cos \phi_1}{\cos \phi_k} \frac{\sqrt{1+(\frac{d\phi_{k}}{dq_{k}})^2}}{\sqrt{1+(\frac{d\phi_{1}}{dq_{1}})^2}} \succsim  1.
    \end{gather*} 
    
    If $n>1$, then from (\ref{6}) we have for any $i \in [1,n]$ that $B^+(x_{k+i})=\big\{\sum_{0\le j \le i-1}\tau_{k+j}+\frac{1}{B^+(x_{k})}\big\}^{-1}$. Using now the SFC-conditions we obtain  \begin{gather*}
        \frac{||dx_{k+n+1}||_p}{||dx_1||_p}=\prod_{1 \le j\le k+n}|1+\tau_j B^+(x_j)|\ge  \big|\frac{\sum_{0\le j \le n}\tau_{k+j}+\frac{1}{B^+(x_{k})}}{\frac{1}{B^+(x_k)}}\big| \succsim 1,\\
        \frac{||dx_{k+n+1}||}{||dx_1||}=\frac{||dx_{k+n+1}||_p}{||dx_1||_p} \frac{\cos \phi_1}{\cos \phi_{k+n+1}} \frac{\sqrt{1+(\frac{d\phi_{k+n+1}}{dq_{k+n+1}})^2}}{\sqrt{1+(\frac{d\phi_{0}}{dq_{0}})^2}}\succsim 1.
    \end{gather*}
    
    If $n=1$, then $x_{k+1}$ is on a flat side, and $x_{k+2}$ is on another circular arc. In this case, $\cos \phi_k \ge \cos \phi_{k+1}$. Then the SFC conditions imply  that \begin{gather*}
        \frac{||dx_{k+2}||_p}{||dx_{k}||_p}=|1+(\tau_k+\tau_{k+1})B^+(x_{k})|\ge \frac{\tau_k+\tau_{k+1}-\cos \phi_k}{\cos \phi_k}\succsim_l \frac{1}{ \cos \phi_k},
    \end{gather*}
    
     Therefore,\begin{align*}
      \frac{||dx_{k+2}||}{||dx_{k}||}\approx \frac{||dx_{k+2}||_p}{||dx_{k}||_p} \frac{\cos \phi_{k}}{\cos \phi_{k+2}}\succsim \frac{1}{\cos \phi_{k+2}}\succsim 1,\quad 
  \diam f^k[I(x_0)] \precsim \diam f^{k+2}[I(x_0)].\end{align*}

    Hence, $\diam I(x_0) \precsim \diam f^{k+n+1}[I(x_0)]$. By summarizing all cases we get \[\diam I(x_0)\precsim \diam f^R [\pmb{\pi}_X I(x_0)].\]
   
   Next we prove that $\diam \pmb{\pi}_XI(x_0)\precsim \diam I(x_0)^{\alpha}$. If $x_0 \in X$, then it is obviously true. So we just need to consider the case when $x_0 \notin X$. 
   
   Suppose that $x_0$ is on a flat side, $x_{-1}, \cdots, x_{-n}$ are on the same circular arc, and $x_{-n}\in X$. By (\ref{derivativeofbilliardmap}), we have \[dq_{-1}=\frac{\tau_{-1}}{\cos \phi_{-1}}d\phi_0, \quad d\phi_{-1}=\big(\frac{\tau_{-1}}{\cos \phi_{-1}}-1\big)d\phi_0, \quad dx_{-1} \in C^s_{x_{-1}}.\]
   Then, $||dx_{-1}||\precsim \frac{||dx_0||}{\cos \phi_{-1}}$ and $dx_{-1}, \cdots, dx_{-n} \in C^s$. For any $i\in [1,n]$, $\cos \phi_{-1}=\cos \phi_{-i}$ and \begin{align*}
        &B^+(x_{-1})=\frac{-2}{\cos \phi_{-1}}+\frac{1}{\tau_{-2}+\frac{1}{B^+(x_{-2})}}\in [-\frac{1}{\cos \phi_{-1}},0] \implies B^+(x_{-2})\in [\frac{-1}{\cos \phi_{-1}}, \frac{-2}{3\cos \phi_{-1}}].
    \end{align*} 
    
    Inductively, we obtain for $i>1$ \begin{gather*}
        B^{+}(x_{-i})\in \big[\frac{-1}{\cos \phi_{-1}}, \frac{-2(i-1)}{[2(i-1)+1]\cos \phi_{-1}}\big], \quad |1+\tau_{-i}B^+(x_{-i})|\ge \frac{2(i-1)-1}{2(i-1)+1},
        \end{gather*}
        \begin{gather}\label{7}
      \frac{||dx_{-1}||_p}{||dx_{-n}||_p}=\prod_{2\le i\le n} |1+\tau_{-i}B^+(x_{-i})|\succsim n^{-1},
        \end{gather} which implies that
    \[\frac{||dx_{-1}||}{||dx_{-n}||}=\frac{||dx_{-1}||_p}{||dx_{-n}||_p} \frac{\cos \phi_{-n}}{\cos \phi_{-1}} \frac{\sqrt{1+(\frac{d\phi_{-1}}{dq_{-1}})^2}}{\sqrt{1+(\frac{d\phi_{-n}}{dq_{-n}})^2}}\succsim \frac{1}{n}, \quad ||dx_{-n}|| \precsim n ||dx_{-1}|| \precsim \frac{n}{\cos \phi_{-1}}||dx_{0}||.\]
    
    Since $f^{-n}I(x_0)\subseteq \mathbb{S}^c\bigcap X \subseteq (f^R\mathbb{S}_0)^c \bigcap X$ is a non-decreasing stable curve, and it has an uniformly bounded slope, then $\diam f^{-n}I(x_0) \precsim n^{-2}$. Hence \[[\diam f^{-n}I(x_0)]^2 \precsim \frac{n}{\cos \phi_{-1}} \diam I(x_0) \cdot [\diam f^{-n}I(x_0)] \precsim \diam I(x_0),\] i.e., $\diam f^{-n}I(x_0) \precsim \diam I(x_0)^{1/2}$. 
    
    Suppose that $x_0, x_{-1}, \cdots, x_{-n}$ are on the same circular arc, and $x_{-n}\in X$. By (\ref{derivativeofbilliardmap}) we have \[d\phi_0=d\phi_{-1}, \quad dq_0=-2d\phi_{-1}+dq_{-1},\quad ||dx_{0}||\approx ||dx_{-1}||, \quad \diam I(x_0)\approx \diam f^{-1}I(x_0).\] 
    
    If $i \in (1,n]$, then  \begin{gather*}
        B^{+}(x_{-i})\in \big[\frac{-1}{\cos \phi_{-1}}, \frac{-2(i-1)}{[2(i-1)+1]\cos \phi_{-1}}\big], \quad |1+\tau_{-i}B^+(x_{-i})|\ge \frac{2(i-1)-1}{2(i-1)+1},\\
         \frac{||dx_{-1}||}{||dx_{-n}||}=\frac{||dx_{-1}||_p}{||dx_{-n}||_p} \frac{\cos \phi_{-n}}{\cos \phi_{-1}} \frac{\sqrt{1+(\frac{d\phi_{-1}}{dq_{-1}})^2}}{\sqrt{1+(\frac{d\phi_{-n}}{dq_{-n}})^2}}=\prod_{2\le i\le n} |1+\tau_{-i}B^+(x_{-i})| \frac{\sqrt{1+(\frac{d\phi_{-1}}{dq_{-1}})^2}}{\sqrt{1+(\frac{d\phi_{-n}}{dq_{-n}})^2}} \succsim \frac{1}{n}, \\
         ||dx_{-n}|| \precsim n ||dx_{-1}|| \approx n||dx_{0}||, \quad \diam f^{-n}I(x_0) \approx n\diam I(x_0).        \end{gather*}
         
         If $\cos \phi_{-n} \approx m^{-1}$, then $n\le m$. Since $f^{-n}I(x_0)\subseteq \mathbb{S}^c\bigcap X \subseteq (f^R\mathbb{S}_0)^c \bigcap X$ is a non-decreasing stable curve with  uniformly bounded slope, then $\diam f^{-n}I(x_0) \precsim m^{-2}$.  Therefore, \begin{align*}
         [\diam f^{-n}I(x_0)]^{3/2} &\precsim n \diam I(x_0) \cdot [\diam f^{-n}I(x_0)]^{1/2}\\
         &\precsim m \diam I(x_0) \cdot [\diam f^{-n}I(x_0)]^{1/2}\precsim \diam I(x_0),
         \end{align*}
         that is, $\diam f^{-n}I(x_0) \precsim \diam I(x_0)^{2/3}$. 
         
          Let now a point $x_0$ is on a flat side,  $x_{-1}, x_{-2}, \cdots, x_{-n}$ ($n\ge 1$) are on flat sides, the points  $x_{-n-1}, \cdots, x_{-n-k}$ are on one and the same circular arc, and $x_{-n-k}\in X$.  Since $n \ge 1$, then $k$ is bounded by a constant depending only on $Q$. From (\ref{derivativeofbilliardmap}) we have that \[-dq_{-1}-\frac{\tau_{-1}}{\cos \phi_{0}}d\phi_{-1}=0, \quad d\phi_{0} =-d\phi_{-1}.\]
          
          Then $||dx_{-1}||\approx||dx_{0}||$. From (\ref{7}) in view of $dq_0=0$ we obtain  \begin{gather*}0=\frac{||dx_0||_p}{||dx_{-n-1}||_p}=|1+(\tau_{-n-1}+\tau_{-n}+\cdots+\tau_{-1})B^{+}(x_{-n-1})|,\quad \frac{||dx_{-n-1}||_p}{||dx_{-n-k}||_p}\succsim \frac{1}{k}\succsim 1.
          \end{gather*}
          
          These relations imply that
          \begin{gather*}
         \frac{||dx_{-1}||_p}{||dx_{-n-1}||_p}=|1+(\tau_{-n-1}+\tau_{-n}+\cdots+\tau_{-2})B^{+}(x_{-n-1})|=\frac{\tau_{-1}}{\tau_{-n-1}+\tau_{-n}+\cdots+\tau_{-1}}\succsim \frac{1}{n},\\
         \frac{||dx_{-1}||_p}{||dx_{-n-k}||_p}\succsim \frac{1}{n},\quad \frac{||dx_{0}||}{||dx_{-n-k}||}\approx \frac{||dx_{-1}||}{||dx_{-n-k}||}=\frac{||dx_{-1}||_p}{||dx_{-n-k}||_p} \frac{\cos \phi_{-n-k}}{\cos \phi_{-1}} \frac{\sqrt{1+(\frac{d\phi_{-1}}{dq_{-1}})^2}}{\sqrt{1+(\frac{d\phi_{-n-k}}{dq_{-n-k}})^2}}\succsim \frac{1}{n}.
          \end{gather*}
          
        Since $f^{-n-k}I(x_0)\subseteq \mathbb{S}^c\bigcap X \subseteq (f^R\mathbb{S}_0)^c \bigcap X$ is a non-decreasing stable curve with uniformly bounded slope, then $\diam f^{-n-k}I(x_0) \precsim n^{-2}$.  Therefore, \begin{align*}
         [\diam f^{-n-k}I(x_0)]^{3/2} &\precsim n \diam I(x_0) \cdot [\diam f^{-n-k}I(x_0)]^{1/2}\precsim \diam I(x_0),
         \end{align*}
         i.e., $\diam f^{-n-k}I(x_0) \precsim \diam I(x_0)^{2/3}$.
         
         By summarizing now all cases, we obtain that $\diam \pmb{\pi}_XI(x_0)\precsim \diam I(x_0)^{1/2}$. 
\end{enumerate}

Now all conditions in Assumption \ref{assumption} are verified, and we have the following statement. \begin{corollary}\label{corstadium}
Theorem \ref{fplthm} holds for stadium billiards.
\end{corollary}

\begin{remark}
Although the Assumption \ref{assumption} is verified only for stadium-type billiards, we claim without proof that Theorem \ref{fplthm} holds for squashes, circular flowers, and for a special class of semi-dispersing billiards (see Figures \ref{F4}, \ref{F5} and \ref{F6}). It is true because these billiards share similar structures of singularities with those in stadium billiards. Moreover, the analysis for the stadium-type billiards is technically slightly more complicated
than for the other three types of billiards because a stadium has a continuous family of period two orbits, and it has focusing boundary components, rather than dispersing components in semi-dispersing billiards. Therefore for stadiums one need to consider more cases. We skip completely  similar verification of the Assumption \ref{assumption} for other classes of billiards.

\end{remark}

\section{Applications to One-dimensional Non-uniformly Expanding Maps}\label{onedsection}
In this section, we extend the results of \cite{demersadv}  on the first hitting and escape statistics for the fast mixing one-dimensional maps to arbitrarily slow  mixing one-dimensional maps. 

Let $f:I \to I$, where $I\subseteq \mathbb{R}$ is a bounded closed interval endowed with the Lebesgue measure $\Leb$. We assume that $f$ preserves an ergodic absolutely continuous probability measure $\mu$ on $I$. For any $\alpha, s \in (0, \infty)$ and $z \in I$ let $\tau_{r,z}(x):=\inf\{n \ge 1: f^n(x) \in B_r(z)\}$, and  \begin{gather}\label{demerlimit}
    L_{\alpha, s}(z): =\lim_{r\to 0} \frac{-1}{s \mu\big(B_r(z)\big)^{1-\alpha}}\log \mu\big\{\tau_{r,z}>s \mu\big(B_r(z)\big)^{-\alpha}\big\}.
\end{gather}

\begin{theorem}\label{extendpoissonthm} \par
Suppose that $f: I \to I$ admits a GMY structure (see the Definition \ref{GMY}), i.e., there exist a sub-interval $J \subsetneq I$, a partition $\mathcal{P}$: $\bigcup_{i}I_i=J$ (where each $I_i$ is a sub-interval) and a return time $R: J \to \mathbb{N}$.
Let the first return time $R \in L^{1+\beta}(J, \Leb)$ ($\beta>0$), $\gcd\{R\} =1$ and $z \in I_{cont}:=\{z\in J: f^k \text{ is continuous at } z \text{ for all }k \in \mathbb{N}_0\}$. Then \begin{enumerate}
    \item There is a mixing probability measure $\mu$ on $I$ with $\inf_{x\in I}\frac{d\mu}{d\Leb}(x)>0$.
    \item If $\beta>0$, $\alpha \in (\frac{1}{1+\beta/2}, 1]$, and $z$ is not a periodic point, then $L_{\alpha, s}(z)=1$.
    \item If $\beta>0$, $\alpha \in (\frac{1}{1+\beta/2}, 1]$, $z$ is a $p$-periodic point of $f$, $\frac{d\mu}{d\Leb}$ is continuous at $z$ and $f^p$ is monotonous at $z$, then $L_{\alpha, s}(z)=1-e^{(\phi + \phi\circ f + \cdots +\phi \circ f^{p-1})(z)}$, where $\phi$ is the potential $-\log |Df|$.
\end{enumerate}
\end{theorem}
\begin{remark}
Theorem \ref{extendpoissonthm} extends the result of Theorem 3.2 in \cite{demersadv} to the case of arbitrarily slow polynomial mixing, i.e., a mixing rate is of order $O(n^{-\beta})$, for dynamics $f:I \to I$, which is assumed to have a GMY structure. This result can be applied, e.g., to the intermittent maps. It follows from the argument in the pages 1284-1285 in \cite{demersadv}, that $L_{\alpha, s}=0$  if $\alpha>1$. Thus we will not consider this singular case.
\end{remark}
\begin{remark}
Although only the case $z \in J$ is considered in Theorem \ref{extendpoissonthm}, it is not much of a restriction, because, once the location of the holes is known, one can chose such $J$ that contains the holes (see the Remark 3.1 in \cite{demersadv}).
\end{remark}
\begin{remark}
According to the Definition \ref{GMY}, the function $F:=f^R$ is monotonous and uniformly expanding on each $I_i\in \mathcal{P}$. 
\end{remark}

The existence of $\mu$ is due to a Young tower \cite{Young2}, which is defined by
\[\Delta:=\{(x,n) \in J \times \mathbb{N}_0: n<R(x)\}.\] 

The dynamics  $F_{\Delta}:\Delta \to \Delta$ sends $(x,n)$ to $(x,n+1)$ if $n+1<R(x)$, and $(x,n)$ to $(f^R(x),0)$ if $n=R(x)-1$. Let $\pi:\Delta \to I$, $\pi(x,n)=f^n(x)$ be a sub-tower \[\Delta_n:=\{(x,l) \in \Delta: l\le n\}.\] 

We identify $\Delta_0$ with $J$. It follows from \cite{Young2} that there is a probability measure  $\mu_{\Delta}$ on $\Delta$, such that $\mu:=\pi_{*}\mu_{\Delta}$ is the one required in the Theorem \ref{extendpoissonthm}, and also that there exists a constant $C>0$, such that for any $I_i \in \mathcal{P}$ and $x,y \in I_i$  \[\frac{d\mu_{J}}{d\Leb}=C^{\pm 1}, \quad \Big|\frac{\frac{d\mu_{J}}{d\Leb}(x)}{\frac{d\mu_{J}}{d\Leb}(y)}-1\Big| \le C \rho^{s(x,y)},\] where  $\mu_{J}:=[\mu(J)]^{-1} \mu|_J$, and the first return time $R$, and $\rho, s(x,y)$ are the ones from the Definition \ref{GMY}. Besides,
\[\mu_{\Delta}=(\int R d\mu_{J})^{-1}\sum_{j} (F^j_{\Delta})_{*}(\mu_J|_{R>j}), \quad \mu_{\Delta_n}:=(\int \min\{R,n+1\} d\mu_J)^{-1}\sum_{j\le n} (F^j_{\Delta})_{*}\mu_J|_{R>j}.\] 

Since $\Delta_{n} \subseteq \Delta$, we denote the first return time on $\Delta_{n}$ by $R_n$, and the first return map on $\Delta_n$ by $F_n:=F_{\Delta}^{R_n}: \Delta_{n} \to \Delta_{ n}$. Clearly, $R_0=R\in L^{1+\beta}(J, \Leb)$ $\implies$ $R_0=R\in L^{1+\beta}(J, \mu_J)$, $R_n\in L^{1+\beta}(\Delta_{n}, \mu_{\Delta_n})$, and $F_0=F=f^R: J \to J$. Let $\mathcal{P}_n$ be the coarsest partition of $\Delta_n$, i.e., for each $Z \in \mathcal{P}_n$, $R_n|_{Z}$ is constant, $F_n^{n+1}|_Z$ is smooth, injective and $F_n^{n+1}(Z)\in \{\Delta_{i+1} \setminus \Delta_{i}, i =0,1,\cdots, n-1\}$. We extend the separation time $s: J \times J \to \mathbb{N}_0$ to $s: \Delta_n \times \Delta_n \to \mathbb{N}_0$ by setting s((x,l), (y,l)):=s(x,y). Before proving Theorem \ref{extendpoissonthm}, we will need one more statement. 
\begin{lemma}[Maximal large deviations for $R_n: \Delta_n \to \mathbb{N}$]\label{mldforrn}\ \par
For any sufficiently large $n\in \mathbb{N}_0$, $\epsilon>0$ and $\delta \in (1/2,1)$ one has  \[\mu_{\Delta_n}\Big(\sup_{m\ge N}\Big|\frac{\sum_{i \le m}R_n\circ F_n^{i}}{m}-\frac{1}{\mu_{\Delta}(\Delta_n)}\Big|\ge \epsilon\Big)\precsim_{n,\delta,\epsilon} N^{-\beta(1-\delta)}.\]
\end{lemma}
\begin{proof}
Note at first that $F_n^{n+1} (Z)$ is one of the levels $\{(x,l)\in \Delta_n: l =i\}$ ($i \le n$)  for any $Z \in \mathcal{P}_n$. Indeed, $\mathcal{P}_n$ is the coarsest partition, and $f$ admits a GMY structure. Then $\inf_{Z \in \mathcal{P}_n}\mu_{\Delta_n}\big(F^{n+1}_n(Z)\big)>0$, which is called a big image. Note now that the distortion estimate $\big|\frac{\det DF_n^{n+1}(x)}{\det DF_n^{n+1}(y)}-1\big|\le C \rho^{s(F_n^{n+1}x,F_n^{n+1}y)} $ for any $x,y \in Z \in \mathcal{P}_n$ holds because $F=f^R$ (see the Definition \ref{GMY}). Here $\det DF^{n+1}_n$ is the Radon-Nikodym derivative of $F_n^{n+1}$ with respect to $\mu_{\Delta_n}$. The map $F_n^{n+1}$ is a Gibbs-Markov map with big images (see the definition of big images on page 133 in \cite{MN}). Fix any $p\in \mathbb{N}$, and define 
\[||v||_{m}:=\sup_{Z_i \in \mathcal{P}_n}\frac{||v|_{Z_i}||_{\infty}}{R^{1/p}_n|_{Z_i}}, \quad ||v||_{l}:=\sup_{Z_i \in \mathcal{P}_n}\sup_{x,y \in Z_i}\frac{|v(x)-v(y)|}{\rho^{s(x,y)}\cdot R^{1/p}_n|_{Z_i}}. \]

Consider now a Banach space $B:=\{v: ||v||_l< \infty, ||v||_m< \infty \}$, and $||v||_{B}:=||v||_l+||v||_m$. A direct computation gives that $||v||_{p(1+\beta)} \le ||R_n||_{1+\beta}^{1/p} \cdot  ||v||_{B}$. By repeating the proof of Proposition 2.1 in \cite{MN}, and by using $R^{1/p}_n \le R_{n}$, we obtain Lasota-Yorke inequalities \begin{align*}
    ||P^mv||_{B}\precsim_{n,p}\rho^m||v||_{B}+\int |v| d\mu_{\Delta_n}\end{align*} 
    for any $m\ge 0$,
where $P$ is a transfer operator of $F_n^{n+1}$ (see Definition \ref{transferoperator}). If $n$ is sufficiently large, then there exists $\tau_{n,p} \in (0,1)$, such that \[||P^mv||_{p}\precsim_{n,p}||P^mv||_{B}\precsim_{n,p} \tau_{n,p}^m ||v||_{B} \text{ for any zero mean } v \in B, m \ge 1.\] 

Arguing as in Lemma \ref{mldforR}, (and replacing $\widetilde{R \circ \pi_e}$ by $R$, $\widetilde{F_e}$ by $F_n^{n+1}$, $\widetilde{\Delta_e}$ by $\Delta_n$, $\xi$ by $\beta$), we get 
\[\mu_{\Delta_n}\Big(\sup_{m\ge N}\Big|\frac{\sum_{i \le m}(R_n-\int R_nd\mu_{\Delta_n})\circ (F_n^{n+1})^{i}}{m}\Big|\ge \epsilon\Big)\precsim_{n,\delta,\epsilon} N^{-\beta(1-\delta)}.\]

By following now a proof of Lemma \ref{mldforR}, and by making use of $(F_n)_{*}\mu_{\Delta_n}=\mu_{\Delta_n}$, we have \[\mu_{\Delta_n}\Big(\sup_{m\ge N}\Big|\frac{\sum_{i \le m}(R_n-\int R_nd\mu_{\Delta_n})\circ F_n^{i}}{m}\Big|\ge \epsilon\Big)\precsim_{n,\delta,\epsilon} N^{-\beta(1-\delta)}.\]

Since $R_n$ is the first return time to $\Delta_n \subseteq \Delta$, then $\int R_n d\mu_{\Delta_n}=\mu_{\Delta}(\Delta_n)^{-1}$, which concludes a proof of lemma.
\end{proof}

\begin{proof}[Proof of Theorem \ref{extendpoissonthm}] First note that $F=f^R$ satisfies (F1)-(F4), (P) and (U) conditions in the pages 1268 and 1269 of \cite{demersadv}. A map $F_n^{n+1}$ also satisfies (F1)-(F4), (P) and (U) for any sufficiently large $n \in \mathbb{N}$,(see the page 1286 in \cite{demersadv}). Let $A_u:=\big\{y \in \Delta_n: \sup_{m\ge u}\Big|\frac{\sum_{i \le m}R_n\circ F_n^{i}(y)}{m}-\frac{1}{\mu_{\Delta}(\Delta_n)}\Big|\ge \epsilon \big\}$. Then by Lemma \ref{mldforrn} $\mu_{\Delta_n}(A_u)\precsim_{n,\delta,\epsilon} u^{-\beta(1-\delta)}$ for any $\delta \in (1/2,1)$. From Proposition 3.5 in \cite{demersadv}, and from the lines 1-4 on the page 1284 in \cite{demersadv}, if the condition $\alpha>\frac{1}{1+\beta (1-\delta)}$ holds, then we get
\begin{eqnarray*}
\lim_{r\to 0} \frac{-1}{s \mu\big(B_r(z)\big)^{1-\alpha}}\log \mu_{\Delta_n}\big\{\tau_{r,z}>s \mu\big(B_r(z)\big)^{-\alpha}\big\} =
\begin{cases}
1, & z \text{ is not periodic}, \\
1-e^{(\phi+\phi \circ f+ \cdots +\phi \circ f^{p-1})(z)}  & z \text{ is } p\text{-periodic}. \\
\end{cases}
\end{eqnarray*}We remark that the Proposition 3.5 in \cite{demersadv} requires a sufficiently fast decay of $\mu_{\Delta_n}(A_u)$. However, the  proof there deals with a polynomial decay of $\mu_{\Delta_n}(A_u)$ (see the lines 1-4 in the page 1284 in \cite{demersadv}).  Let $n\to \infty$, in order to exhaust the space. Then one can argue as in the pages 1286-1287 in \cite{demersadv}, and obtain a desired limit for (\ref{demerlimit}). Finally, note that $\delta$ is any number in $(1/2,1)$.  Therefore $\alpha >\frac{1}{1+\beta/2}$, which concludes a proof of the theorem.\end{proof}

We apply now the Theorem \ref{extendpoissonthm} to the following class of intermittent maps 

\begin{eqnarray}\label{pmmap}
T_{\gamma}(x) =
\begin{cases}
x+2^{\gamma}x^{1+\gamma}, & 0\le x \le 1/2\\
2x-1,  & 1/2< x \le 1 \\
\end{cases}, \quad \gamma \in [0,1).
\end{eqnarray}

A standard choice of $J$ in the Theorem \ref{extendpoissonthm} is within the interval $[1/2, 1]$, and for the first return time a standard choice is $R \in L^{1/\gamma}(J)$. But it is also possible to choose $J$ inside $[0,1/2]$, as long as $0 \notin \interior{J}$. For example, $J$ could be one of the connected components of the set $[\bigcup_{i\ge 0} (T_{\gamma}|_{[0,1/2]})^{-i}\{1\}]^c$. 

Theorem \ref{extendpoissonthm} immediately implies the following corollary.

\begin{corollary}[intermittent maps]\ \par Let $\beta$ in Theorem \ref{extendpoissonthm} equals $\gamma^{-1}-1$, $z\notin \bigcup_{i\ge 0} T_{\gamma}^{-i}\{1/2\}$. Then \begin{enumerate}
    \item If $\gamma \in (0,1)$, $\alpha \in (\frac{2\gamma}{1+\gamma}, 1]$, and $z$ is not a periodic point, then $L_{\alpha, s}(z)=1$,
    \item If $\gamma \in (0,1)$, $\alpha \in (\frac{2\gamma}{1+\gamma}, 1]$, $z$ is a $p$-periodic point of $T_{\gamma}$, and $\frac{d\mu}{d\Leb}$ is continuous at $z$, then $L_{\alpha, s}(z)=1-e^{(\phi + \phi\circ f + \cdots +\phi \circ f^{p-1})(z)}$, where $\phi$ is a potential $-\log |DT_{\gamma}|$.
\end{enumerate}
\end{corollary}


%
%

\medskip

\bibliography{bibfile}

\begin{thebibliography}{22}
\providecommand{\natexlab}[1]{#1}
\providecommand{\url}[1]{\texttt{#1}}
\expandafter\ifx\csname urlstyle\endcsname\relax
  \providecommand{\doi}[1]{doi: #1}\else
  \providecommand{\doi}{doi: \begingroup \urlstyle{rm}\Url}\fi

\bibitem[Alves et~al.(2011)Alves, Freitas, Luzzatto, and Vaienti]{vaientiadv}
J.~F. Alves, J.~M. Freitas, S.~Luzzatto, and S.~Vaienti.
\newblock From rates of mixing to recurrence times via large deviations.
\newblock \emph{Adv. Math.}, 228\penalty0 (2):\penalty0 1203--1236, 2011.
\newblock ISSN 0001-8708.
\newblock \doi{10.1016/j.aim.2011.06.014}.
\newblock URL \url{https://doi.org/10.1016/j.aim.2011.06.014}.

\bibitem[Arratia et~al.(1989)Arratia, Goldstein, and Gordon]{chenmethod}
R.~Arratia, L.~Goldstein, and L.~Gordon.
\newblock Two moments suffice for {P}oisson approximations: the {C}hen-{S}tein
  method.
\newblock \emph{Ann. Probab.}, 17\penalty0 (1):\penalty0 9--25, 1989.
\newblock ISSN 0091-1798.
\newblock URL
  \url{http://links.jstor.org/sici?sici=0091-1798(198901)17:1<9:TMSFPA>2.0.CO;2-X&origin=MSN}.

\bibitem[B\'{a}lint and Gou\"{e}zel(2006)]{balint}
P.~B\'{a}lint and S.~Gou\"{e}zel.
\newblock Limit theorems in the stadium billiard.
\newblock \emph{Comm. Math. Phys.}, 263\penalty0 (2):\penalty0 461--512, 2006.
\newblock ISSN 0010-3616.
\newblock \doi{10.1007/s00220-005-1511-6}.
\newblock URL \url{https://doi.org/10.1007/s00220-005-1511-6}.

\bibitem[Bruin et~al.(2018)Bruin, Demers, and Todd]{demersadv}
H.~Bruin, M.~F. Demers, and M.~Todd.
\newblock Hitting and escaping statistics: mixing, targets and holes.
\newblock \emph{Adv. Math.}, 328:\penalty0 1263--1298, 2018.
\newblock ISSN 0001-8708.
\newblock \doi{10.1016/j.aim.2017.12.020}.
\newblock URL \url{https://doi.org/10.1016/j.aim.2017.12.020}.

\bibitem[{Bunimovich} and {Su}(2022)]{Subbb}
L.~{Bunimovich} and Y.~{Su}.
\newblock {Back to Boundaries in Billiards}.
\newblock \emph{arXiv e-prints}, art. arXiv:2203.00785, Mar. 2022.

\bibitem[Chernov(1999)]{Chernovjsp}
N.~Chernov.
\newblock Decay of correlations and dispersing billiards.
\newblock \emph{J. Statist. Phys.}, 94\penalty0 (3-4):\penalty0 513--556, 1999.
\newblock ISSN 0022-4715.
\newblock \doi{10.1023/A:1004581304939}.
\newblock URL \url{https://doi.org/10.1023/A:1004581304939}.

\bibitem[Chernov and Markarian(2006)]{CMbook}
N.~Chernov and R.~Markarian.
\newblock \emph{Chaotic billiards}, volume 127 of \emph{Mathematical Surveys
  and Monographs}.
\newblock American Mathematical Society, Providence, RI, 2006.
\newblock ISBN 0-8218-4096-7.
\newblock \doi{10.1090/surv/127}.
\newblock URL \url{https://doi.org/10.1090/surv/127}.

\bibitem[Chernov and Zhang(2005)]{CZnon}
N.~Chernov and H.-K. Zhang.
\newblock Billiards with polynomial mixing rates.
\newblock \emph{Nonlinearity}, 18\penalty0 (4):\penalty0 1527--1553, 2005.
\newblock ISSN 0951-7715.
\newblock \doi{10.1088/0951-7715/18/4/006}.
\newblock URL \url{https://doi.org/10.1088/0951-7715/18/4/006}.

\bibitem[Chernov and Zhang(2008)]{CZcmp}
N.~Chernov and H.-K. Zhang.
\newblock Improved estimates for correlations in billiards.
\newblock \emph{Comm. Math. Phys.}, 277\penalty0 (2):\penalty0 305--321, 2008.
\newblock ISSN 0010-3616.
\newblock \doi{10.1007/s00220-007-0360-x}.
\newblock URL \url{https://doi.org/10.1007/s00220-007-0360-x}.

\bibitem[Haydn and Vaienti(2020)]{vaientinullset}
N.~Haydn and S.~Vaienti.
\newblock Limiting entry and return times distribution for arbitrary null sets.
\newblock \emph{Comm. Math. Phys.}, 378\penalty0 (1):\penalty0 149--184, 2020.
\newblock ISSN 0010-3616.
\newblock \doi{10.1007/s00220-020-03795-0}.
\newblock URL \url{https://doi.org/10.1007/s00220-020-03795-0}.

\bibitem[Kachurovskii and Podvigin(2016)]{anothermld2}
A.~G. Kachurovskii and I.~V. Podvigin.
\newblock Estimates of the rate of convergence in the von {N}eumann and
  {B}irkhoff ergodic theorems.
\newblock \emph{Trans. Moscow Math. Soc.}, pages 1--53, 2016.
\newblock ISSN 0077-1554.
\newblock \doi{10.1090/mosc/256}.
\newblock URL \url{https://doi.org/10.1090/mosc/256}.

\bibitem[Kachurovski\u{\i} and Sedalishchev(2011)]{anothermld1}
A.~G. Kachurovski\u{\i} and V.~V. Sedalishchev.
\newblock Constants of estimates for the rate of convergence in the von
  {N}eumann and {B}irkhoff ergodic theorems.
\newblock \emph{Mat. Sb.}, 202\penalty0 (8):\penalty0 21--40, 2011.
\newblock ISSN 0368-8666.
\newblock \doi{10.1070/SM2011v202n08ABEH004180}.
\newblock URL \url{https://doi.org/10.1070/SM2011v202n08ABEH004180}.

\bibitem[Markarian(2004)]{Markarianetds}
R.~Markarian.
\newblock Billiards with polynomial decay of correlations.
\newblock \emph{Ergodic Theory Dynam. Systems}, 24\penalty0 (1):\penalty0
  177--197, 2004.
\newblock ISSN 0143-3857.
\newblock \doi{10.1017/S0143385703000270}.
\newblock URL \url{https://doi.org/10.1017/S0143385703000270}.

\bibitem[Melbourne(2009)]{melbourne09}
I.~Melbourne.
\newblock Large and moderate deviations for slowly mixing dynamical systems.
\newblock \emph{Proc. Amer. Math. Soc.}, 137\penalty0 (5):\penalty0 1735--1741,
  2009.
\newblock ISSN 0002-9939.
\newblock \doi{10.1090/S0002-9939-08-09751-7}.
\newblock URL \url{https://doi.org/10.1090/S0002-9939-08-09751-7}.

\bibitem[Melbourne and Nicol(2005)]{MN}
I.~Melbourne and M.~Nicol.
\newblock Almost sure invariance principle for nonuniformly hyperbolic systems.
\newblock \emph{Comm. Math. Phys.}, 260\penalty0 (1):\penalty0 131--146, 2005.
\newblock ISSN 0010-3616.
\newblock \doi{10.1007/s00220-005-1407-5}.
\newblock URL \url{https://doi.org/10.1007/s00220-005-1407-5}.

\bibitem[P\`ene and Saussol(2010)]{penebacktoball}
F.~P\`ene and B.~Saussol.
\newblock Back to balls in billiards.
\newblock \emph{Comm. Math. Phys.}, 293\penalty0 (3):\penalty0 837--866, 2010.
\newblock ISSN 0010-3616.
\newblock \doi{10.1007/s00220-009-0911-4}.
\newblock URL \url{https://doi.org/10.1007/s00220-009-0911-4}.

\bibitem[P\`ene and Saussol(2016)]{peneetds}
F.~P\`ene and B.~Saussol.
\newblock Poisson law for some non-uniformly hyperbolic dynamical systems with
  polynomial rate of mixing.
\newblock \emph{Ergodic Theory Dynam. Systems}, 36\penalty0 (8):\penalty0
  2602--2626, 2016.
\newblock ISSN 0143-3857.
\newblock \doi{10.1017/etds.2015.28}.
\newblock URL \url{https://doi.org/10.1017/etds.2015.28}.

\bibitem[P\`ene and Saussol(2020)]{peneijm}
F.~P\`ene and B.~Saussol.
\newblock Spatio-temporal {P}oisson processes for visits to small sets.
\newblock \emph{Israel J. Math.}, 240\penalty0 (2):\penalty0 625--665, 2020.
\newblock ISSN 0021-2172.
\newblock \doi{10.1007/s11856-020-2074-0}.
\newblock URL \url{https://doi.org/10.1007/s11856-020-2074-0}.

\bibitem[{Su}(2022)]{Sutams}
Y.~{Su}.
\newblock {Vector-valued Almost Sure Invariance Principle For Non-stationary
  Dynamical Systems}.
\newblock \emph{accepted by Trans. AMS}, 2022.
\newblock \doi{10.1090/tran/8609}.
\newblock URL \url{https://doi.org/10.1090/tran/8609}.

\bibitem[Su and Bunimovich(2022)]{Su}
Y.~Su and L.~A. Bunimovich.
\newblock Poisson approximations and convergence rates for hyperbolic dynamical
  systems.
\newblock \emph{Comm. Math. Phys.}, 390\penalty0 (1):\penalty0 113--168, 2022.
\newblock ISSN 0010-3616.
\newblock \doi{10.1007/s00220-022-04309-w}.
\newblock URL \url{https://doi.org/10.1007/s00220-022-04309-w}.

\bibitem[Young(1998)]{Young}
L.-S. Young.
\newblock Statistical properties of dynamical systems with some hyperbolicity.
\newblock \emph{Ann. of Math. (2)}, 147\penalty0 (3):\penalty0 585--650, 1998.
\newblock ISSN 0003-486X.
\newblock \doi{10.2307/120960}.
\newblock URL \url{https://doi.org/10.2307/120960}.

\bibitem[Young(1999)]{Young2}
L.-S. Young.
\newblock Recurrence times and rates of mixing.
\newblock \emph{Israel J. Math.}, 110:\penalty0 153--188, 1999.
\newblock ISSN 0021-2172.
\newblock \doi{10.1007/BF02808180}.
\newblock URL \url{https://doi.org/10.1007/BF02808180}.

\end{thebibliography}

\end{document}